\documentclass[11pt, a4paper,reqno]{amsart}
\pdfoutput=1
\usepackage{tikz,genyoungtabtikz,enumitem,rotating,latexsym,bm,stmaryrd,caption,}
\usetikzlibrary{positioning,intersections,decorations}
\usetikzlibrary{shapes}
 \definecolor{ao(english)}{rgb}{0.0, 0.5, 0.0}

\usetikzlibrary{decorations.pathmorphing}

	\definecolor{eng}{rgb}{0.0, 0.5, 0.0}
\definecolor{apple}{rgb}{0.55, 0.71, 0.0}
\definecolor{cadmium}{rgb}{0.0, 0.42, 0.24}
\definecolor{darkspringgreen}{rgb}{0.09, 0.45, 0.27}

\newcommand{\algebra}{{\mathscr{H}^\sigma_n}}

 \newcommand{\kay}{a}
 
  \newcommand{\WHY}{{y }}

\usepackage{graphicx}

\def\psucc
{\mathrel
{\scalebox{.9}[1]{$\succ$} 
 \mkern-14.5mu\scalebox{.4}[1]{$\succ$} }  \ 
  }

\def\ppsucc
{\mathrel
{\scalebox{.9}[1]{$\succ$} 
 \mkern-17mu\scalebox{.4}[1]{$\succ$} }  \ 
  }

\def\ppsucceq
{\mathrel
{\scalebox{.9}[1]{$\succeq$} 
 \mkern-17mu\scalebox{.4}[1]{$\succeq$} }  \ 
  }

  \def\psucceq
{\mathrel
{\scalebox{.9}[1]{$\succeq$} 
 \mkern-14.5mu\scalebox{.4}[1]{$\succeq$} }  \ 
  }

\definecolor{toggle}{rgb}{0,0,0}

\newcommand{\reflectpath}{\SSTP_\al^\flat}
\newcommand\REMOVETHESE[2]{{{{\mathsf{M}}_{#1}^{#2}	}}}
\newcommand\ADDTHIS[2]{{{{\mathsf{P}}_{#1}^{#2}}}}

\usepackage[normalem]{ulem}

 \newcommand{\al}{{{{\color{magenta}\boldsymbol\alpha}}}}
  \newcommand{\gam}{{{{\color{ao(english)}\boldsymbol\gamma}}}}
 \newcommand{\bet}{{{\color{cyan}\boldsymbol\beta}}}
 \newcommand{\emp}{{{\color{magenta}\mathbf{\boldsymbol\clock}}}}
  \newcommand{\empb}{{{\color{cyan}\mathbf{\boldsymbol\clock}}}}

\newcommand{\exx}{{b_\al }}

\newcommand{\eps}{ \varepsilon}

\newcommand{\Shl}{\widehat{\mathfrak{S}}_{\aatch  }}
\newcommand{\w}{{\underline{w}}}
\newcommand{\x}{{\underline{x}}}
\newcommand{\y}{{\underline{y}}}

\newcommand{\vvv}{{\underline{v}}}

\newcommand{\Alc}{\text{\bf Alc}}

\usepackage{standalone}

\usepackage{etex}

\usepackage{todonotes}

\usepackage{tikz}
\usetikzlibrary{matrix,  intersections, calc, decorations.pathreplacing} 

\usepackage{tikz-cd}

\newlength{\superthick}
\newlength{\cornerradius}
\tikzstyle{corner}=[rounded corners=\cornerradius]
\tikzstyle{dot}=[circle, inner sep=0pt, minimum size=4.8pt]
\tikzstyle{string}=[line width=\superthick]
\tikzstyle{std}=[string,dash pattern=on 0.9pt off 0.9pt]
\definecolor{realcyan}{rgb}{0,1,1}

 \usepackage{amsmath,amsthm,amsfonts,amssymb,mathrsfs,pb-diagram}
\usepackage[
bookmarks=true,colorlinks=true,linktoc=page,citecolor=darkgreen,linkcolor=blue,urlcolor=mediumblue]{hyperref}
\usepackage{caption}
\usepackage{lipsum,wasysym}
\usepackage{mathtools}
\usepackage[a4paper,margin=0.8in]{geometry}
\usepackage{cleveref}
 \usepackage{amsmath}
\mathchardef\mhyphen="2D
\usepackage{color}
\usepackage{xcolor}
\usepackage{ifthen}
\usepackage{sidecap}   
\definecolor{mediumblue}{rgb}{0.0, 0.0, 0.8}
 \colorlet{darkgreen}{green!50!black}
\newcommand{\Res}{{\rm Res}}
\newcommand{\Rem}{{\rm Rem}}

\renewcommand{\geq}{\geqslant}
\renewcommand{\leq}{\leqslant}

\tikzset{wei/.style= 
{red,double=red,double
distance=0.5pt}}

\tikzset{wei2/.style={red,double=red,double
distance=0.5pt}}

\allowdisplaybreaks
\numberwithin{equation}{section}
\usepackage{scalefnt}

\newtheorem{thm}{Theorem}[section]
\newtheorem{cor}[thm]{Corollary}

\newtheorem{lem}[thm]{Lemma}
\newtheorem{prop}[thm]{Proposition}

\newtheorem*{prop*}{Proposition}
\newtheorem*{thmA*}{Theorem A}

\newtheorem*{thmB*}{Theorem B}
\newtheorem*{thmC*}{Theorem C}\newtheorem*{thm*}{Theorem D}
\newtheorem*{cor*}{Corollary}

\newtheorem*{conj*}{Conjecture A}

\newtheorem*{conj1*}{Conjecture B}
\newtheorem*{Acknowledgements*}{Acknowledgements}

\theoremstyle{rmk}

\theoremstyle{defn}
\newtheorem{rmk}[thm]{Remark}
\newtheorem{defn}[thm]{Definition}
\newtheorem{eg}[thm]{Example}

\newtheorem{notn}[thm]{Notation}

\newcommand{\great}{>}
\newcommand{\less}{<}

\newcommand{\rad}{\mathrm{rad}}
\newcommand{\res}{\mathrm{res}}

\newcommand{\Std}{{\rm Std}}

\newcommand{\Shape}{\operatorname{Shape}} 
\newcommand{\Path}{{\rm Path}}
\newcommand{\RStd}{{\rm RStd}}
\newcommand{\CStd}{{\rm CStd}}

\newcommand{\la}{\lambda}
\newcommand{\I}{i}
\newcommand{\J}{j}

\newcommand{\M}{m}

\newcommand{\SSTS}{\mathsf{S}}

\newcommand{\SSTT}{\mathsf{T}}  
\newcommand{\SSTP}{\mathsf{P}}  
\newcommand{\SSTU}{\mathsf{U}}  
\newcommand{\SSTQ}{\mathsf{Q}}  
\newcommand{\sts}{\mathsf{S}}  
\newcommand{\stt}{\mathsf{T}}  
\newcommand{\ZZ}{{\mathbb Z}}
\newcommand{\NN}{{\mathbb N}}

\let\<=\langle
\let\>=\rangle

\newcommand\mydots{\makebox[1em][c]{.\hfil.\hfil.}}
\tikzset{
ultra thin/.style= {line width=0.05pt},
very thin/.style=  {line width=0.2pt},
thin/.style=       {line width=0.1pt},
semithick/.style=  {line width=0.6pt},
thick/.style=      {line width=0.8pt},
very thick/.style= {line width=1.2pt},
ultra thick/.style={line width=1.6pt}
}

\crefname{defn}{Definition}{Definitions}
\crefname{thm}{Theorem}{Theorems}
\crefname{prop}{Proposition}{Propositions}
\crefname{lem}{Lemma}{Lemmas}
\crefname{cor}{Corollary}{Corollaries}
\crefname{conj}{Conjecture}{Conjectures}
\crefname{section}{Section}{Sections}
\crefname{subsection}{Subsection}{Subsections}
\crefname{eg}{Example}{Examples}
\crefname{figure}{Figure}{Figures}
\crefname{rem}{Remark}{Remarks}
\crefname{rmk}{Remark}{Remarks}
\crefname{equation}{equation}{equation}

\Crefname{defn}{Definition}{Definitions}
\Crefname{thm}{Theorem}{Theorems}
\Crefname{prop}{Proposition}{Propositions}
\Crefname{lem}{Lemma}{Lemmas}
\Crefname{cor}{Corollary}{Corollaries}
\Crefname{conj}{Conjecture}{Conjectures}
\Crefname{section}{Section}{Sections}
\Crefname{subsection}{Subsection}{Subsections}
\Crefname{eg}{Example}{Examples}
\Crefname{figure}{Figure}{Figures}
\Crefname{rem}{Remark}{Remarks}
\Crefname{rmk}{Remark}{Remarks}

      \newcommand{\aatch}{h}
      \newcommand{\aatchpair}{{\underline{h}}}
      \newcommand{\enn}{{h}}

\usepackage[hang,flushmargin]{footmisc}

\begin{document}
  
 \title[Path combinatorics and light leaves  for  quiver Hecke algebras] {Path combinatorics and light leaves  \\ for  quiver Hecke algebras}
 
\begin{abstract}
We recast the classical notion of ``standard tableaux" 
in an alcove-geometric setting and 
extend these classical ideas to all ``reduced paths" in our geometry. 
This broader path-perspective is   essential 
for implementing the higher categorical ideas of 
Elias--Williamson  in the setting of quiver Hecke algebras.  
Our first main result is the construction of  light leaves bases
of  quiver Hecke algebras.  
These bases are richer and  encode more structural  information than their classical counterparts, even in the case of the symmetric groups.  
Our second main result provides   path-theoretic generators 
for the ``Bott--Samelson truncation" of the quiver Hecke  algebra.  
\end{abstract}
 
  \author{Chris Bowman}
       \address{Department of Mathematics, 
University of York, Heslington, York, YO10 5DD, UK}
\email{Chris.Bowman-Scargill@york.ac.uk}
  
  \author{Anton Cox}

 	\address{Department of Mathematics, City, University of London,   London, UK}
\email{A.G.Cox@city.ac.uk}

   \author{Amit Hazi}
   	\address{Department of Mathematics, City, University of London,   London, UK}
   \email{Amit.Hazi@city.ac.uk}

       \author{Dimitris Michailidis }
          \address{School of Mathematics, Statistics and Actuarial Science University of Kent, Canterbury,  UK}
\email{D.Michailidis@kent.ac.uk}

   \maketitle
   
   \renewcommand{\tau}{h}
   
 The  
  symmetric group lies at the intersection of two great categorical theories:  
  Khovanov--Lauda and Rouquier's  categorification of quantum groups and their knot invariants \cite{MR2525917,ROUQ} and Elias--Williamson's diagrammatic categorification in terms of  endomorphisms of Bott--Samelson bimodules.  
   The purpose of this paper and its companion \cite{cell4us2} is to     construct an explicit isomorphism between   these two  diagrammatic worlds.  
The backbone of this isomorphism is provided by the ``light-leaves" bases of these algebras.  
 
The light leaves bases of diagrammatic Bott--Samelson endomorphism algebras 
 were crucial 
in  the    calculation of  counterexamples to the expected bounds of Lusztig's and James' conjectures \cite{w13}.  
 These bases are structurally  far richer than any known basis of the  quiver  Hecke algebra ---    they vary with respect to each  possible  choice  of   reduced word/path-vector in the alcove geometry --- and this richer structure is necessary in order to construct a basis in terms of the ``Soergel 2-generators" of these algebras.  
 In this paper we show that the   (quasi-hereditary quotients of) quiver Hecke algebras,  
      $\mathcal{H}_n^\sigma $ for $\sigma \in \ZZ^\ell$, 
       have analogues of the light leaves bases in the ``non-singular" or ``regular" case.

Without loss of generality, we assume that  $\sigma  \in \ZZ^\ell$ is weakly increasing.  
We then let    $\underline{\tau}= (\tau_0,\dots,\tau_{\ell-1})  \in \NN^\ell$
      be such that  $\tau_m\leq   \sigma_{m+1}-\sigma_{m} $ for $0\leq m< \ell-1$ and  $\tau_{\ell-1}<  e+\sigma_0-\sigma_{\ell-1}  $.  
The light leaves bases we construct are     indexed  by paths in an alcove geometry  of type 
$$      A_{\tau_0-1}\times \mydots \times A_{\tau_{\ell-1}-1}  \backslash \widehat{A}_{\tau_0+\dots +\tau_{\ell-1} -1} $$
for $e> h= \tau_0+\dots+\tau_{\ell-1} $, where each point in this geometry corresponds to an $\ell$-multipartition with 
 at most $\tau_m$ columns in the $m$th component 
(we denote the set of such $\ell$-multipartitions by $\mathscr{P}_{\underline{\tau} }(n)$). 
 In this manner, we obtain   cellular bases of  the largest possible quasi-hereditary quotients of       $\mathcal{H}_n^\sigma $  controlled by non-singular Kazhdan--Lusztig theory (this is the broadest possible generalisation, in the context of cyclotomic quiver Hecke algebras, of studying the category of tilting modules of the principal block of the general linear group, ${\rm GL}_h(\Bbbk)$, in characteristic $p>h$).

\begin{thmA*}[Regular light leaves bases for quiver Hecke algebras]  
For each $\la\in \mathscr{P}_{\underline{\tau}}(n)$ we fix a reduced path $\SSTQ_\la\in \Path_{\aatchpair  }(\la)$ and 
for each $\SSTS \in \Path_{\aatchpair  }(\la)$, 
 	we fix an associated
 reduced path vector 	 ${\underline{  \SSTP}_\SSTS}$ terminating with $\SSTQ_\la$ 
 (this notation is defined in  \cref{9999,sec3}).  
  We have that 
$$
\{\Upsilon^\SSTS_{\underline{\SSTP}_\SSTS}
 \Upsilon_\SSTT^{\underline{\SSTP}_\SSTT}\mid 
\SSTS, \SSTT \in \Path_{\aatchpair  }(\la),\la\in \mathscr{P}_{\underline{\tau}}(n)\}
$$
is a cellular basis of $\mathcal{H}^\sigma_n /\mathcal{H}^\sigma_n {\sf y}_{\aatchpair}\mathcal{H}^\sigma_n $ 
where we quotient by the ideal generated by the element 
$ {\sf y}_{\underline{\tau}  } $  
which    kills all simples indexed by $\ell$--partitions with more than $\tau_m$ columns in the $m$th component.     
\end{thmA*}
         \color{black}

We then   consider the so-called  ``Bott--Samelson truncations" to the principal block
    $${\sf f}_{n,\sigma }(\mathcal{H}^\sigma_n /\mathcal{H}^\sigma_n {\sf y}_{\aatchpair}\mathcal{H}^\sigma_n) {\sf f}_{n,\sigma } 
 \qquad \text{for}\qquad  
 {\sf f}_{n,\sigma} =\sum_{\begin{subarray}c
   \SSTS\in \Std_{n,\sigma} (\la)
   \\
   \la  \in \mathscr{P}_{\underline{\tau}}(n)
   \end{subarray}} { e}_{\SSTS}
   $$ 
which we will show (in the companion paper \cite{cell4us2}) are isomorphic to the  (breadth enhanced) diagrammatic  Bott--Samelson endomorphism algebras of Elias--Williamson \cite{MR3555156}.  
 The charm of this isomorphism is that  it allows one to view the current state-of-the-art regarding  $p$-Kazhdan--Lusztig theory  (in type $A$) 
  entirely within  the context of the group algebra of the symmetric group, 
without the need for calculating intersection cohomology groups, or working with
  parity sheaves, or appealing to the deepest results of 2-categorical Lie theory. 
  In this paper, we specialise Theorem A by making certain path-theoretic choices which allow us to 
reconstruct Elias--Williamson's  generators entirely within the quiver Hecke algebra itself, using our language of  paths.

\begin{thmB*}
The Bott--Samelson  truncation of the Hecke  algebra
 ${\sf f}_{n,\sigma }(\mathcal{H}^\sigma_n /\mathcal{H}^\sigma_n {\sf y}_{\aatchpair}\mathcal{H}^\sigma_n) {\sf f}_{n,\sigma } $ is generated  
  by  horizontal and vertical concatenation of 
 the elements 
 $$
   e_{\SSTP_\al}, \quad  {\sf fork}_{\al\al}^{\al\emp},
 \quad
  {\sf spot}_{\al}^{\emp},
 \quad
  {\sf hex}_{\al\bet\al}^{\bet\al\bet},  
   \quad
  {\sf com}_{\bet\gam}^{\gam\bet},
   \quad e_{\SSTP_\emptyset}, \ \text{and} \ \
  {\sf adj}_{\al\emptyset}^{\emptyset\al}
 $$
(this notation is defined in \cref{geners})  for $\al,\bet,\gam \in \Pi$ such that $\al$ and $\bet$ label an arbitrary pair  of non-commuting reflections and 
$\bet$ and $\gam$ label  an arbitrary pair  of commuting reflections.  
 \end{thmB*}

  The paper is structured as follows.  
In Section 1 we construct a ``classical-type" cellular basis of $\mathcal{H}^\sigma_n /\mathcal{H}^\sigma_n {\sf y}_{\aatchpair}\mathcal{H}^\sigma_n$ in terms of tableaux but using a slightly exotic dominance ordering ---   
the proofs in this section are a little dry and can be skipped on the first reading.  
In Section 2, we upgrade this basis to a ``light leaf type" construction and prove Theorem A.  
Finally, in Section 3   of the paper we  
illustrate how  Theorem A  allows us to reconstruct the precise analogue of the light leaves basis for the Bott--Samelson endomorphism algebras for regular blocks of quiver Hecke algebras
 (as a special case of Theorem A, written in terms of the generators of Theorem B).  We do this in the exact   language used by Elias, Libedinsky, and Williamson in order to make the construction clear for a reader whose background lies in either field.

This paper is a companion to \cite{cell4us2}, but the reader should note that the results here are entirely self-contained (although we refer to \cite{cell4us2} for further development of ideas, examples, etc).

\section{A tableaux basis of the quiver Hecke algebra }  \label{sec3}
   
   We let $\mathfrak{S}_n$ denote the symmetric group on $n$ letters and let $\ell:  \mathfrak{S}_n \to \ZZ$ denotes its length function.  
 We let $\leq $ denote the (strong) Bruhat order on $\mathfrak{S}_n$.  
 Given $\underline{i}=(i_1,\dots, i_n)\in (\ZZ/e\ZZ)^n$ and $s_r=(r,r+1)\in \mathfrak{S}_n$ we set 
 $s_r(\underline{i})= (i_1,\dots, i_{r-1}, i_{r+1}, i_r ,i_{r+2}, \dots ,i_n)$.  
 
\begin{defn}[\cite{MR2551762,MR2525917,ROUQ}]\label{defintino1}
Fix $e > 2 $.  
The {\sf quiver Hecke algebra} (or KLR algebra),  $\mathcal{H}_n $,   is defined to be the unital, associative $\ZZ$-algebra with generators
$$ 
\{e_{\underline{i}}  \ | \ {\underline{i}}=(i_1,\mydots,i_n)\in   (\ZZ/e\ZZ)^n\}\cup\{y_1,\mydots,y_n\}\cup\{\psi_1,\mydots,\psi_{n-1}\},
$$
subject to the    relations 
\begin{align}
\tag{R1}\label{rel1}  e_{\underline{i}} e_{\underline{j}} =\delta_{{\underline{i}},{\underline{j}}} e_{\underline{i}} 
 \qquad \textstyle 
\sum_{{\underline{i}} \in   (\ZZ/e\ZZ )^n } e_{{\underline{i}}} =1_{\mathcal{H} _n}   
\qquad y_r		e_{{\underline{i}}}=e_{{\underline{i}}}y_r 
\qquad
\psi_r e_{{\underline{i}}} = e_{s_r({\underline{i}})} \psi_r 
\qquad
y_ry_s =y_sy_r
 \end{align}
for all $r,s,{\underline{i}},{\underline{j}}$ and 
\begin{align}
\tag{R2}\label{rel2}
\psi_ry	_s  = y_s\psi_r \  \text{ for } s\neq r,r+1&
\qquad  
&\psi_r\psi_s = \psi_s\psi_r \ \text{ for } |r-s|>1
 \\ \tag{R3}\label{rel3}
y_r \psi_r e_{\underline{i}}
  =
(\psi_r y_{r+1} \color{black}-\color{black} 
\delta_{i_r,i_{r+1}})e_{\underline{i}}  &
 \qquad 
&y_{r+1} \psi_r e_{\underline{i}}   =(\psi_r y_r \color{black}+\color{black} 
\delta_{i_r,i_{r+1}})e_{\underline{i}} 
\end{align}
\begin{align}
\tag{R4}\label{rel4}
\psi_r \psi_r  e_{\underline{i}} &=\begin{cases}
\mathrlap0\phantom{(\psi_{r+1}\psi_r\psi_{r+1} + 1)e_{\underline{i}} \qquad}& \text{if }i_r=i_{r+1},\\
e_{\underline{i}}  & \text{if }i_{r+1}\neq i_r, i_r\pm1,\\
(y_{r+1} - y_r) e_{\underline{i}}  & \text{if }i_{r+1}=i_r +1  ,\\
(y_r - y_{r+1}) e_{\underline{i}} & \text{if }i_{r+1}=i_r -1   
\end{cases}\\
\tag{R5}\label{rel5}
\psi_r \psi_{r+1} \psi_r e_{\underline{i}} &=\begin{cases}
(\psi_{r+1}\psi_r\psi_{r+1} - 1)e_{\underline{i}}\qquad& \text{if }i_r=i_{r+2}=i_{r+1}+1  ,\\
(\psi_{r+1}\psi_r\psi_{r+1} + 1)e_{\underline{i}}& \text{if }i_r=i_{r+2}=i_{r+1}-1    \\
 \psi_{r+1} \psi_r\psi_{r+1}e_{\underline{i}}   &\text{otherwise} 
\end{cases} 
\end{align}
 for all  permitted $r,s,i,j$.  We identify such elements with decorated permutations  and the multiplication with vertical concatenation, $\circ$,  of these diagrams  in the  standard fashion of \cite[Section 1]{MR2551762} and as illustrated in \cref{fkgjhdfgjkhdhfjgkdgfhjkdhgfjk}.  We let $\ast$ denote the anti-involution which fixes the generators.  
 
\end{defn}

\begin{defn}Fix $e > 2 $ and $\sigma\in  \ZZ ^\ell$.  
The {\sf cyclotomic quiver Hecke algebra}, $\mathcal{H}_n^\sigma$,   is defined to be the
 quotient of $\mathcal{H}_n$ by the relation
\begin{align}
\label{rel1.12} y_1^{\sharp\{\sigma_m | \sigma_m= i_1 ,1\leq m \leq \ell 	\}} e_{\underline{i}} &=0 
\quad 
\text{ for    ${\underline{i}}\in (\ZZ/e\ZZ)^n$.}
\end{align}  
 \end{defn}

\color{black}

   As we see in \cref{fkgjhdfgjkhdhfjgkdgfhjkdhgfjk}, the $y_k$ elements are visualised as dots on strands; we hence refer to them as {\sf KLR dots}. Given   $p<q$ we set 
$$
w^p_q= s_p s_{p+1}\dots s_{q-1}\quad
w^q_p=  s_{q-1} \dots s_{p+1} s_{p }
\qquad \psi^p_q= \psi_{p}\psi_{p+1}\mydots \psi_{q-1} 
\qquad \psi^q_p=   \psi_{q-1}  \mydots
 \psi_{p+1}\psi_{p}.
 $$  
and  given an expression $\w=s_{i_1}\dots s_{i_p}\in \mathfrak{S}_n$ we set 
$\psi_\w= \psi_{i_1}\dots \psi_{i_p}\in \mathcal{H}_n$.   
We let $\boxtimes$ denote the horizontal concatenation of  KLR diagrams.  
 Finally, we define the degree   as follows, 
$$
{\rm deg}(e_{\underline{i}})=0 \quad 
{\rm deg}( y_r)=2\quad 
{\rm deg}(\psi_r e_{\underline{i}})=
\begin{cases}
-2		&\text{if }i_r=i_{r+1} \\
1		&\text{if }i_r=i_{r+1}\pm 1 \\
 0 &\text{otherwise} 
\end{cases}. $$ 

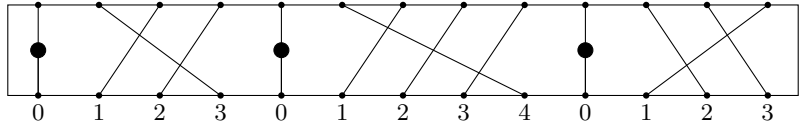
\begin{figure}[ht!]
 $$   
 \scalefont{0.8}\begin{tikzpicture}
[xscale=0.8,yscale=  
   0.8]
  
     \draw(0,0) rectangle (13,1.5);
        \foreach \i in {0.5,1.5,...,12.5}
  {
   \fill(\i,0) circle(1.5pt) coordinate (a\i);
          \fill(\i,1.5)circle(1.5pt)  coordinate (d\i);
    } 
    
 \draw[fill](0.5,0)--++(90:0.75) circle (3.5pt);       
 \draw[fill](4.5,0)--++(90:0.75) circle (3.5pt);

 \draw(0.5,0)--++(90:1.5) node[below, at start] {$0$}; 
 \draw(1.5,0)-- (2.5,1.5) node[below, at start] {$1$};    
 \draw(2.5,0)-- (3.5,1.5) node[below, at start] {$2$};       
  \draw(1.5,1.5)-- (3.5,0) node[below, at end] {$3$};

 \draw(4+0.5,0)--++(90:1.5) node[below, at start] {$0$}; 
 \draw(4+1.5,0)-- (4+2.5,1.5) node[below, at start] {$1$};    
 \draw(4+2.5,0)-- (4+3.5,1.5) node[below, at start] {$2$};       
 \draw(5+2.5,0)-- (5+3.5,1.5) node[below, at start] {$3$};       

  \draw(4+1.5,1.5)-- (5+3.5,0) node[below, at end] {$4$};

  \draw(10.5,0)-- (9.5+3,1.5) node[below, at start] {$1$};    
  \draw(9.5,0)-- (9.5,1.5) node[below, at start] {$0$};        
  \draw(1+10.5,0)-- (1+9.5,1.5) node[below, at start] {$2$};        
    \draw(1+1+10.5,0)-- (1+1+9.5,1.5) node[below, at start] {$3$};           
   
    \draw[fill](9.5,0)--++(90:0.75) circle (3.5pt);

  \end{tikzpicture}  
  $$
 
 \caption{The element  $y_1\psi^2_4 e_{(0,1,2,3)} \boxtimes y_1\psi^2_5 e_{(0,1,2,3,4)} \boxtimes
  y_1   \psi^4_2 e_{(0,1,2,3)}	
 	$. }
\label{fkgjhdfgjkhdhfjgkdgfhjkdhgfjk} \end{figure}

\subsection{Box configurations, partitions, residues and tableaux}   %
For a fixed $n\in \NN$ and $\ell\geq 1$ we define a  {\sf  box-configuration}  to be a subset of
$$\{[i, j, m] \mid  0\leq m < \ell,  1\leq i, j \leq n  \}$$
of $n$ elements, which we call boxes, and we let  $\mathcal{B}_{ \ell}(n)$ denote the set of all such  box-configurations.  
We refer to a box $[i,j,m]\in \la \in \mathcal{B}_{ \ell}(n)$ as being in the $i$th row and $j$th column of the $m$th component of $\lambda$.    Given a box, $[i,j,m]$,  
we define the {\sf content} of this box  to be  ${\sf ct}[i,j,m] = \sigma_m+ j-i$ and we define its {\sf residue} to be ${\rm res}[i,j,m]= {\sf ct}[i,j,m] \pmod e$.  
We refer to a box of residue  $r\in \ZZ/e\ZZ$ as an $r$-box.

We define a {\sf composition}, $\lambda$,  of $n$ to be a   finite   sequence  of non-negative integers $ (\lambda_1,\lambda_2, \ldots)$ whose sum, $|\lambda| = \lambda_1+\lambda_2 + \mydots$, equals $n$.    We say that $\la$ is a partition if, in addition, this sequence is weakly decreasing.  
An     {\sf $\ell $-multicomposition} (respectively {\sf 
   $\ell$-multipartition})   $\lambda=(\lambda^{(0)},\mydots,\lambda^{(\ell-1)})$ of $n$ is an $\ell $-tuple of   compositions (respectively  partitions)  such that $|\lambda^{(0)}|+\mydots+ |\lambda^{(\ell-1)}|=n$. We denote the sets of 
$\ell$-multipartitions  and $\ell$-multi-compositions of $n$ by $ \mathscr{P}_{\ell}(n)$ and 
$ \mathcal{C}_{\ell}(n)$, respectively.  
 Given  $\lambda=(\lambda^{(0)},\lambda^{(1)},\ldots ,\lambda^{(\ell-1)}) \in \mathscr{P}_{\ell}(n)$, the {\sf Young diagram} of $\lambda$    is defined to be the box configuration, 
\[
\{[i,j,m] \mid  1\leq  j\leq \lambda^{(m)}_i,0
\leq m <\ell \}.
\]
 We do not distinguish between the multipartition and its Young diagram. 
    We let $\varnothing$ denote the empty multipartition.   
  Given $\la \in\mathcal{B}_\ell (n)$, we let   
  $${\rm Rem} (\la)=\{[r,c,m]  \in \la \mid [r+1,c,m]\not\in \la,  [r,c+1,m]\not\in \la\}$$ 
 $${\rm Add}  (\la)=\{[r,c,m] \not \in \la \mid [r-1,c,m] \in \la \text{ if }r>1,  [r,c-1,m] \in \la\text{ if }c>1\}$$  for $\la \in\mathscr{P}_\ell(n)\subseteq \mathcal{B} _\ell (n)$ this coincides with the usual definition of addable and removable boxes.  
  We let  ${\rm Rem}_r  (\la)\subseteq {\rm Rem}  (\la)$ and  ${\rm Add}_r (\la)\subseteq {\rm Add}  (\la)$  denote the subsets of boxes of residue $r\in \ZZ/e\ZZ$.  
 
  Given $\lambda\in \mathcal{B}_{\ell}(n)$, we define a {\sf $\la$-tableau}  to be a filling of the boxes of $[\la] $ with the numbers 
$\{1,\mydots , n\}$. For  $\lambda\in \mathcal{B}_{\ell}(n)$, 
we say that a $\la$-tableau is {\sf row-standard}, {\sf column-standard},  or simply {\sf standard} if  the entries in each component increase along the rows, increase along the columns, or increase along both rows and columns, respectively.  
We say that a $\la$-tableau $\sts$ has {\sf shape} $\la$ and write ${\sf Shape}(\sts)=\la$. 
Given $\lambda\in\mathcal{B}_{ \ell}(n)$, we let 
${\rm Tab}(\la)$ denote the set of all    tableaux of shape $\lambda\in\mathcal{B}_{ \ell}(n)$.   Given $\SSTT \in {\rm Tab}(\la)$ and $1\leq k\leq n$, we let $\SSTT^{-1}(k)$ denote the box  $\square\in \la$ 
 such that  $\SSTT(\square)=k$. 
We let 
$\RStd (\lambda), \CStd (\lambda), \Std (\lambda) \subseteq 
{\rm Tab}(\la)
$ denote the subsets of all row-standard, column-standard, and 
 standard tableaux, respectively.  
We let 
$\Std(n)=\cup_{\la\in \mathscr{P}_\ell (n)} \Std(\la)   $
 for    $n\in \NN$.

\begin{defn} 
We define the {\sf reverse cylindric   ordering}, $\succ$,  as follows.  
Let $1\leq i,i',j,j' \leq n$ and $0\leq m,m' < \ell$.  
 We write $[i,j,m] \succ [i',j',m']$ 
 if  $i<i'$, 
  or  $i=i'$ and $m<m'$,
 or $i=i'$ and $m=m'$ and $j<j'$.  
%
%
%
%
For $\la,\mu \in \mathcal{B}_{\ell}(n)$, we write $\la  \succ \mu$ if the $\succ$-minimal  box $\square  \in 
( \la\cup\mu) \setminus ( \la \cap \mu)  $ belongs to $\mu$.  
 \end{defn}

\begin{defn} 
We define the {\sf dominance   ordering}, $\rhd$,    as follows.  
Let $1\leq i,i',j,j' \leq n$ and $0\leq m,m' < \ell$.  
 We write $[i,j,m] \rhd [i',j',m']$ 
 if  $m<m'$, 
  or  $m=m'$ and $i<i'$,
 or $i=i'$ and $m=m'$ and $j<j'$.  
Given $\la,\mu \in \mathcal{B}_{\ell}(n)$, we write $\la  \rhd \mu$ if the $\rhd$-maximal  box $\square  \in 
( \la\cup\mu) \setminus ( \la \cap \mu)  $ belongs to $\la$.  
 \end{defn}
  
   Given $\sts$, we write $\sts{\downarrow}_{\leq k} $ or $\sts{\downarrow}_{\{1,\dots, k\}} $ (respectively $\sts{\downarrow}_{\geq k} $) for the subtableau of $\sts$ consisting solely of the entries $1$ through $k$ (respectively of the entries $k$ through $n$).  
 Given $\la \in  \mathcal{B}_{\ell}(n)$, we let $\stt_\la$ denote the     $\la$-tableau in which we place the  entry $n$ in the minimal $\succ$-node of $\la$, then continue in this fashion inductively.  
Given $\la \in  \mathcal{B}_\ell(n)$, we let $\SSTS_\la$ denote the     $\la$-tableau in which we place the  entry $n$ in the minimal $\rhd$-node of $\la$, then continue in this fashion inductively.  
\color{black} 
Finally, given $\sts,\stt$ two $\la$-tableaux, we let $w^\sts_ {\stt } \in \mathfrak{S}_n$ 
 be the permutation such that $w^\sts_ {\stt }(\stt )=\sts$.

\newcommand{\tener}{10}
\newcommand{\ten}{11}
\newcommand{\tenty}{12}
\newcommand{\tentyr}{13}
\newcommand{\tentyrr}{14}
\newcommand{\tentyrrr}{15}
\begin{eg}\label{oftttt}For $\la=((2,1^2),(2^2,1),(1^3))$, we have that $w^\sts _{\stt_\la}=(4,5)(2,6)$ 
for $\sts$ and $\stt_\la$ the   tableaux 
$$\stt_\la= \Bigg(\; \Yvcentermath1\gyoung(1;2,6,\tener)	\; , \;  	\gyoung(3;4,7;8,\ten)
\; , \;  \gyoung(5,9,\tenty)\; \Bigg)
\qquad
\sts=  \Bigg(\; \Yvcentermath1\gyoung(1;6,2,\tener)	\; , \;  	\gyoung(3;5,7;8,\ten)
\; , \;  \gyoung(4,9,\tenty)\; \Bigg)   $$
\end{eg}


\begin{defn}  Given any box, $[r,c,m] $,  we define the associated  ($\succ$)-{\sf  Garnir belt}  to be the collection of boxes, $ {\sf Gar}_{\succ}([r,c,m])$, as follows 
 $$   
\{[r,j,k]  \mid j\geq 1, 1\leq k <  m \}
\cup\{[r,j,m]\  \mid 1\leq   j \leq c  \}
\cup\{[r-1,j,m]   \mid c\leq j    \}
\cup
\{[r-1,j,k] \mid  j\geq 1,   k > m \}		 
$$
with the convention that we ignore any box from the ``zeroth" row.  
We define    
$  {\sf Gar}_{\rhd}([r,c,m]) $ to be the intersection of    ${\sf Gar}_{\succ}([r,c,m]) $ with the $m$th component and refer to this as the 
($\rhd$)-{\sf  Garnir belt}.    
\end{defn}

\begin{eg}\label{continuation}
Let  $\sigma=(0,3,8)\in \ZZ^3$ and $e=14$.  Given $\la= ((3^2,2^2,1 ),(5^2,3,2,1 ),(4^2,3,1^2 ))$ and  the box  $[3,3,1]\in\la$, we    colour  
  ${\sf Gar}_{\succ}([3,3,1]) \cap \la$ 
below  
 $$   \left( \ 
 \begin{minipage}{2.1cm}
\scalefont{0.9} \begin{tikzpicture}
  [xscale=0.5,yscale=  
  0.5]

  \draw[thick ] (1,-4)--++(180:1)--++(-90:1)
--++(0:1)--++(90:1);
    \draw(0.5,-4.5) node {$10$};
 
\draw[thick, fill=white] (0,0)--++(0:3)--++(-90:2)
--++(180:1)
--++(-90:2)
--++(180:1) 
--++(180:1)--++(90:4);

\clip (0,0)--++(0:3)--++(-90:2)
--++(180:1)
--++(-90:2)
--++(180:1) 
--++(180:1)--++(90:4);

  \fill[yellow!60] (2,-2)  --++(-90:1)--++(180:2)--++(90:1)--++(0:2);

\draw(0.5,-0.5) node {$0$};
\draw(1.5,-0.5) node {$1$};
\draw(2.5,-0.5) node {$2$};
  \draw(0.5,-1.5) node {$13 $};
\draw(1.5,-1.5) node {$0$};
\draw(2.5,-1.5) node {$1$};
 \draw(0.5,-2.5) node {$12$};
\draw(1.5,-2.5) node {$13$};

\draw(0.5,-3.5) node {$11$};
\draw(1.5,-3.5) node {$12$};

	\foreach \i in {0,1,2,...,10}
	{
		\path  (0,0)++(-90:1*\i cm)  coordinate (a\i);
		\path  (0.5,-0.5)++(0:1*\i cm)  coordinate (b\i);
		\path  (0.5,-1.5)++(0:1*\i cm)  coordinate (c\i);
		\path  (0.5,-2.5)++(0:1*\i cm)  coordinate (d\i);
		\draw[thick] (a\i)--++(0:6);
		\draw[thick] (1,0)--++(-90:6);
		\draw[thick] (2,0)--++(-90:5);
				\draw[thick] (3,0)--++(-90:5);
	}

\draw[thick ] (0,0)--++(0:3)--++(-90:2)
--++(180:1)
--++(-90:2)
--++(180:1)--++(-90:1)
--++(180:1)--++(90:5);

  \end{tikzpicture}  \end{minipage}  
   \!\!\! , \;\; 
    \begin{minipage}{2.8cm}
\scalefont{0.9} \begin{tikzpicture}
  [xscale=0.5,yscale=  
  0.5]

\draw[thick, fill=white] (0,0)--++(0:5)--++(-90:2)
--++(180:2)--++(-90:1)
--++(180:1)--++(-90:1)
--++(180:1) 
--++(180:1) --++(-90:1)
--++(90:5);
\clip (0,0)--++(0:5)--++(-90:2)
--++(180:2)--++(-90:1)
--++(180:1)--++(-90:1)
--++(180:1) --++(-90:1)
--++(180:1)--++(90:5);

  \fill[orange!60] (2,-1) --++(0:3)--++(-90:1)--++(180:2)--++(-90:1)--++(180:3)
  --++(90:1)--++(0:2)--++(90:1);

\draw(0.5,-0.5) node {$3$};
\draw(1.5,-0.5) node {$4$};
\draw(2.5,-0.5) node {$5$};
\draw(3.5,-0.5) node {$6$};
\draw(4.5,-0.5) node {$7$};
\draw(0.5,-1.5) node {$2$};
\draw(1.5,-1.5) node {$3$};
 \draw(2.5,-1.5) node {$4$};
\draw(3.5,-1.5) node {$5$};
\draw(4.5,-1.5) node {$6$};

  \draw(0.5,-2.5) node {$1$};
\draw(1.5,-2.5) node {$2$};
\draw(2.5,-2.5) node {$3$};

\draw(0.5,-3.5) node {$0$};
\draw(1.5,-3.5) node {$1$};

\draw(0.5,-4.5) node {$14$};

	\foreach \i in {0,1,2,...,10}
	{
		\path  (0,0)++(-90:1*\i cm)  coordinate (a\i);
		\path  (0.5,-0.5)++(0:1*\i cm)  coordinate (b\i);
		\path  (0.5,-1.5)++(0:1*\i cm)  coordinate (c\i);
		\path  (0.5,-2.5)++(0:1*\i cm)  coordinate (d\i);
		\draw[thick] (a\i)--++(0:6);
		\draw[thick] (1,0)--++(-90:6);
		\draw[thick] (2,0)--++(-90:5);
				\draw[thick] (3,0)--++(-90:5);
	}

\draw[thick ] (0,0)--++(0:4)--++(-90:2)
--++(180:1)--++(-90:1)
--++(180:1)--++(-90:1)
--++(180:1)--++(-90:1)
--++(180:1)--++(90:5);

  \end{tikzpicture}  \end{minipage} 
, \;\; 
   \begin{minipage}{2.4cm}
\scalefont{0.9} \begin{tikzpicture}
  [xscale=0.5,yscale=  
  0.5]
   \draw[thick ] (1,-4)--++(180:1)--++(-90:1)
--++(0:1)--++(90:1);
    \draw(0.5,-4.5) node {$4$};

\draw[thick, fill=white] (0,0)--++(0:4)--++(-90:2)
--++(180:1)--++(-90:1)
--++(180:2)--++(-90:1)
--++(180:1)--++(90:4);
\clip (0,0)--++(0:4)--++(-90:2)
--++(180:1)--++(-90:1)
--++(180:2)--++(-90:1)
--++(180:1)--++(90:4);

  \fill[yellow!60] (4,-1)  --++(-90:1)--++(180:4)--++(90:1)--++(0:4); 
\draw  (0,0)--++(0:4)--++(-90:2)
--++(180:1)--++(-90:1)
--++(180:1)--++(-90:1)
--++(180:1)--++(-90:1)
--++(180:1)--++(90:5);


 \draw(0.5,-0.5) node {$8$};
\draw(1.5,-0.5) node {$9$};
\draw(2.5,-0.5) node {$10$};
\draw(3.5,-0.5) node {$11$};

 \draw(0.5,-2.5+1) node {$7$};
\draw(1.5,-2.5+1) node {$8$};
\draw(2.5,-2.5+1) node {$9$};
\draw(3.5,-2.5+1) node {$10$};

 \draw(0.5,-3.5+1) node {$6$};
 \draw(1.5,-3.5+1) node {$7$};
 \draw(2.5,-3.5+1) node {$8$};

    \draw(0.5,-3.5) node {$5$};

\draw[thick ] (0,0)--++(0:4)--++(-90:2)
--++(180:1)--++(-90:1)
--++(180:1)--++(-90:1)
--++(180:1)--++(-90:1)
--++(180:1)--++(90:5);

	\foreach \i in {0,1,2,...,10}
	{
		\path  (0,0)++(-90:1*\i cm)  coordinate (a\i);
		\path  (0.5,-0.5)++(0:1*\i cm)  coordinate (b\i);
		\path  (0.5,-1.5)++(0:1*\i cm)  coordinate (c\i);
		\path  (0.5,-2.5)++(0:1*\i cm)  coordinate (d\i);
		\draw[thick] (a\i)--++(0:6);
		\draw[thick] (1,0)--++(-90:6);
		\draw[thick] (2,0)--++(-90:5);
				\draw[thick] (3,0)--++(-90:5);
	}

  \end{tikzpicture}  \end{minipage}  \right).
   $$
The $(\rhd)$-Garnir belt is  the subset of boxes coloured in orange.  
   \end{eg}

    Given $\aatchpair\in \NN^\ell$, we let  $\mathscr{P}_{\underline{\tau}}(n)$ (respectively $\mathscr{C}_{\aatchpair}(n)$)  denote the subsets of  $\ell$-multipartitions (respectively  $\ell$-multicompositions)  with at most $\tau_m$ columns in the $m$th component.

 \begin{lem}\label{easy1}  
Let  $\la\in \mathscr{P}_{\underline{\tau}}(n)$.  
For any  box $[i,j,m]$,  the multiset of  residues of the boxes in  $\la \cap {\sf Gar}_{\succ}[i,j,m]   $ 
  is multiplicity-free (i.e. no residue appears more than once).   
\end{lem}

\begin{proof}
This follows immediately from the definitions since $\la\in \mathscr{P}_{\underline{\tau}}(n)$.  
\end{proof}

Let $\geq $ be an ordering on $  
  \mathcal{B}_\ell (n)$.   Given $1\leq k\leq n$, we let ${\mathcal A}^{\geq}_\stt(k)$, 
(respectively ${\mathcal R}^{\geq}_\stt(k)$)  denote the set of   all addable $\res (\stt^{-1}(k))$-boxes (respectively all
removable   $\res (\stt^{-1}(k))$-boxes)  of the box-configuration $\Shape(\stt{\downarrow}_{\{1,\dots ,k\}})$ which
   are less than  $\stt^{-1}(k)$ in the $\geq $-order.     We define the $(\geq)$-degree of $\stt\in \Std(\la)$ for $\la \in  \mathscr{P}_\ell  (n)$  as follows,
$$
\deg_{\geq}(\stt) = \sum_{k=1}^n \left(	|{\mathcal A}^{\geq}_\stt(k)|-|{\mathcal R}^{\geq}_\stt(k)|	\right).
$$
\color{black} 
We let ${{\rm res} (\SSTT)}$ denote  the {\sf residue sequence} consisting  of  $\res (\stt^{-1}(k))$ for $k=1,\dots, n$ in order.  
We set $e_\SSTT:= e_{{\rm res} (\SSTT)}\in \mathscr{H}_n^\sigma$.  
We set 
\begin{align}
\label{tableau not}
\WHY_\la^{\succ} = \prod _{k=1}^n y_k^{ |{\mathcal A}^{\succ}_{\SSTT_\la}(k)|}e_{\SSTT_\la}
\qquad \text{ and } \qquad 
\WHY_\la^{\rhd} = \prod _{k=1}^n y_k^{ |{\mathcal A}^{\rhd}_{\SSTS_\la}(k)|}e_{\SSTS_\la}.
\end{align}
 For $\la \in \mathscr{P}_\ell(n)$, the element $\WHY_\la^{\rhd} $ was first defined in \cite[Definition 4.15]{hm10}.   
We remark that $\WHY^\succ_\la = e_{\SSTT_\la}$ for $\la \in 
  \mathscr{P}_{\aatchpair} (n)$.   
Given $\SSTS, \SSTT \in \Std(\la)$ and $\w$ any fixed  reduced word for $w^\SSTS_\SSTT$ we 
let $\psi^\SSTS_\SSTT:=e_\SSTS\psi_\w e_\SSTT$.

 \begin{defn} We set $\mathcal{Y}_n=\langle e_{\underline{i}}, y_k \mid \underline{i}\in(\ZZ/e\ZZ)^n, 1\leq k \leq n\rangle$.  
 \end{defn}

\subsection{The quotient}
  
A long-standing belief in modular Lie theory  is that we should (first) restrict our attention to fields whose characteristic, $p$, is greater than the Coxeter number, $ h$, of the algebraic  group we are studying.   
This allows one to consider a ``regular block"     of 
the  algebraic group in question.
  What does this mean on the other side of the Schur--Weyl duality  relating ${\rm GL}_h(\Bbbk)$ and $\Bbbk\mathfrak{S}_n$?  
 By the second fundamental theorem of invariant theory, the kernel of the group algebra of the symmetric group acting on $n$-fold $h$-dimensional tensor space is the 2-sided ideal generated by the  element 
$$
\textstyle \sum _{g \in \mathfrak{S}_{h+1}\leq \mathfrak{S}_n} {{\rm sgn}(g)} g\in \Bbbk\mathfrak{S}_n
$$
Modulo ``more dominant terms"  this  element  can be rewritten in the   form we introduced in \cref{tableau not}, as follows 
  \begin{equation} \label{ppppps}
y_{ {(h+1)} }  ^{\succ}\boxtimes {1}_{\mathcal{H}_{n-h-1}^\sigma}
 \end{equation} 
      by \cite[4.16 Corollary]{hm10}.
The simples of this algebra are indexed by partitions with at most $h$  columns.   
Given $\sigma \in \ZZ^\ell$, we let    $\underline{\tau}= (\tau_0,\dots,\tau_{\ell-1})  \in \NN^\ell$
      be such that  $\tau_m\leq   \sigma_{m+1}-\sigma_{m} $ for $0\leq m< \ell-1$ and  $\tau_{\ell-1}<  e+\sigma_0-\sigma_{\ell-1}  $.
      We define  
   \begin{equation}\label{idempotent} {\sf y}_{\underline{\tau}  }=
 \sum_{\begin{subarray}c
     0\leq m < \ell 
      \end{subarray}} y^{\succ}_{(\emptyset,\dots,\emptyset , (h_m+1) ,\emptyset, \dots,\emptyset)}\boxtimes 1_{\mathcal{H}_{n-h_m-1}^\sigma},
\end{equation}  
to be  the higher level analogue of the element in (\ref{ppppps}).

 \begin{defn}\label{ajghkfdlsfkghsldkjfghsdkjfghsdkfjghsdk}
Given weakly increasing $\sigma \in \ZZ^\ell$, we let    $\underline{\tau}= (\tau_0,\dots,\tau_{\ell-1})  \in \NN^\ell$
      be such that  $\tau_m\leq   \sigma_{m+1}-\sigma_{m} $ for $0\leq m< \ell-1$ and  $\tau_{\ell-1}<  e+\sigma_0-\sigma_{\ell-1}  $. 
 We define $ \algebra:=  \mathcal{H}^\sigma_n /\mathcal{H}^\sigma_n {\sf y}_{\aatchpair}\mathcal{H}^\sigma_n$. 
 \end{defn}
 
  In level $\ell=1$, the condition of \cref{ajghkfdlsfkghsldkjfghsdkjfghsdkfjghsdk} is equivalent to $h_0< e$ (and $\mathscr{P}_{\underline{\tau}}(n)$  is the set of  partitions of $n$ with  strictly fewer than $e$ columns).

\subsection{Generator/partition combinatorics}
 Our cellular basis will provide a stratification of $\algebra $ in which each layer is generated by an idempotent correspond to some multipartition.   
Whence we wish to understand the effect of multiplying a generator of a given layer in the cell-stratification by a   KLR ``dot generator".  This leads us to define combinatorial analogues of 
the   dot generators as maps on the set of  box configurations.

 \begin{defn}\color{black} 
Let $\la \in \mathcal{B}_{ \ell}(n)$ and let $[i,j,m]\in \la$.  
 We say that $[i,j,m]$ is left-justified if either $j\leq e$ or   there exists some $[i,j-p,m]\in \la$ for $1\leq p \leq e$.  
\end{defn}

  \begin{defn}\color{black} \label{reddefn}
Let $\la\in \mathcal{B}_{ \ell}(n)$.  
 For $\alpha \in \la $   an $r$-box, we define 
 $Y_{\alpha}(\la)
=
  \la  -   \alpha \cup   \beta  $ 
 where  the $r$-box   $\beta\not \in \la$ is such that $\beta\succ\alpha$, it  is left-justified, 
 and   is  minimal in $\succ$ with respect to these properties (if such a box exists).  
 If such a box does not exist then
   we say that $Y_\alpha(\la)$ and $\beta$ are both  undefined.

   We write $\la\psucc \mu$ if    $ 
 \la= Y_\alpha(\mu)
  $ 
  for some $\alpha \in \mu$ and we then extend $\ \psucc$ to a partial ordering on $ \mathcal{B}_{ \ell}(n)$ by taking the transitive closure.  
  Suppose that 
 $\{[i_k,j_k,m_k]\mid {0< k\leq p}  \} $  is a set of   $r$-boxes and that 
 $Y_{[i_k,j_k,m_k]}(\la\cup [i_k,j_k,m_k])
=
  \la  \cup  [i_{k+1},j_{k+1},m_{k+1}]  $ 
  for $k\geq 1$.  We define 
$$Y^p _ {[i_1,j_1,m_1]}(\la\cup [i_1,j_1,m_1])
=(\la\cup [i_p,j_p,m_p]).
$$
 \end{defn}

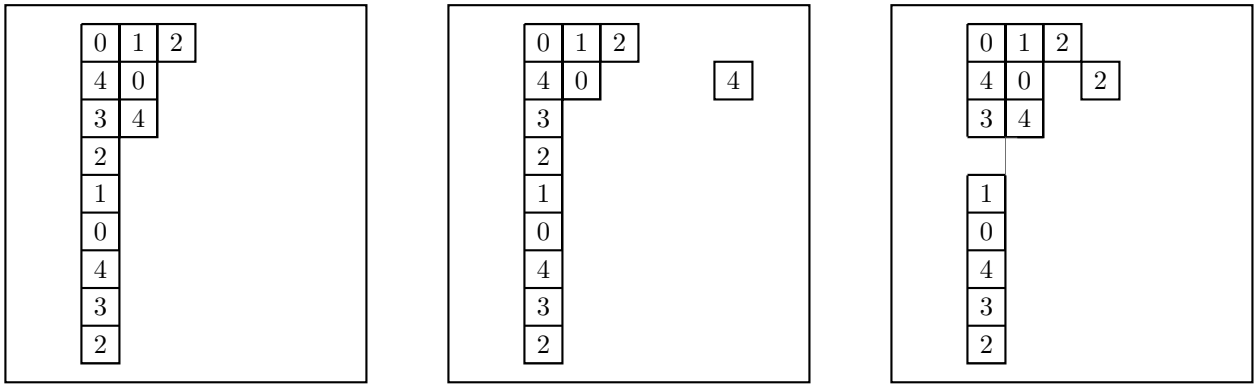
\begin{figure}[ht!]
$$
   \begin{minipage}{3.52cm}
\scalefont{0.9} \begin{tikzpicture}
  [xscale=0.5,yscale=  
  0.5]

\draw[thick ](2,0) rectangle (3,-1);  
\draw(2.5,-0.5) node {$2$};

\draw[thick](-2,0.5) rectangle (7.5,-9.5);

\draw[thick, fill=white] (0,0)--++(0:2)--++(-90:3)--++(180:1)--++(-90:6)--++(180:1)--++(90:9);
  \clip (0,0)--++(0:2)--++(-90:3)--++(180:1)--++(-90:6)--++(180:1)--++(90:9);
\draw(0.5,-0.5) node {$0$};
\draw(1.5,-0.5) node {$1$};
\draw(2.5,-0.5) node {$2$};
\draw(0.5,-1.5) node {$4$};
\draw(1.5,-1.5) node {$0$};
\draw(2.5,-1.5) node {$1$};
\draw(0.5,-2.5) node {$3$};
\draw(1.5,-2.5) node {$4$};
\draw(2.5,-2.5) node {$0$};

\draw(0.5,-3.5) node {$2$};
\draw(1.5,-3.5) node {$3$};

\draw(0.5,-4.5) node {$1$};
\draw(1.5,-4.5) node {$2$};

\draw(0.5,-5.5) node {$0$};
\draw(1.5,-5.5) node {$1$};

\draw(0.5,-6.5) node {$4$};\draw(0.5,-7.5) node {$3$};\draw(0.5,-8.5) node {$2$};
 
	\foreach \i in {0,1,2,...,20}
	{
		\path  (0,0)++(-90:1*\i cm)  coordinate (a\i);
		\path  (0.5,-0.5)++(0:1*\i cm)  coordinate (b\i);
		\path  (0.5,-1.5)++(0:1*\i cm)  coordinate (c\i);
		\path  (0.5,-2.5)++(0:1*\i cm)  coordinate (d\i);
		\draw[thick] (a\i)--++(0:6);
		\draw[thick] (1,0)--++(-90:6);
		\draw[thick] (2,0)--++(-90:5);
	}

  \end{tikzpicture}  \end{minipage}  
  \qquad  \qquad  \qquad
   \begin{minipage}{3.52cm}
\scalefont{0.9} \begin{tikzpicture}
  [xscale=0.5,yscale=  
  0.5]

\draw[thick ](2+3,-1) rectangle (3+3,-2);  
\draw(2.5+3,-1.5) node {$4$};

\draw[thick](-2,0.5) rectangle (7.5,-9.5); 
 
\draw[thick ](2,0) rectangle (3,-1);  
\draw(2.5,-0.5) node {$2$};

\draw[thick, fill=white] (0,0)--++(0:2)--++(-90:2)--++(180:1)--++(-90:7)--++(180:1)--++(90:9);
  \clip (0,0)--++(0:2)--++(-90:2)--++(180:1)--++(-90:7)--++(180:1)--++(90:9);
\draw(0.5,-0.5) node {$0$};
\draw(1.5,-0.5) node {$1$};
\draw(2.5,-0.5) node {$2$};
\draw(0.5,-1.5) node {$4$};
\draw(1.5,-1.5) node {$0$};
\draw(2.5,-1.5) node {$1$};
\draw(0.5,-2.5) node {$3$};
 \draw(2.5,-2.5) node {$0$};

\draw(0.5,-3.5) node {$2$};
\draw(1.5,-3.5) node {$3$};

\draw(0.5,-4.5) node {$1$};
\draw(1.5,-4.5) node {$2$};

\draw(0.5,-5.5) node {$0$};
\draw(1.5,-5.5) node {$1$};

\draw(0.5,-6.5) node {$4$};\draw(0.5,-7.5) node {$3$};\draw(0.5,-8.5) node {$2$};
 
	\foreach \i in {0,1,2,...,20}
	{
		\path  (0,0)++(-90:1*\i cm)  coordinate (a\i);
		\path  (0.5,-0.5)++(0:1*\i cm)  coordinate (b\i);
		\path  (0.5,-1.5)++(0:1*\i cm)  coordinate (c\i);
		\path  (0.5,-2.5)++(0:1*\i cm)  coordinate (d\i);
		\draw[thick] (a\i)--++(0:6);
		\draw[thick] (1,0)--++(-90:6);
		\draw[thick] (2,0)--++(-90:5);
	}

  \end{tikzpicture}  \end{minipage}  \qquad  \qquad   \qquad
  \begin{minipage}{3.52cm}
\scalefont{0.9} \begin{tikzpicture}
  [xscale=0.5,yscale=  
  0.5]

\draw[thick ](2,0) rectangle (3,-1);  
\draw(2.5,-0.5) node {$2$};

\draw[thick](-2,0.5) rectangle (7.5,-9.5);

\draw[thick](3,-1) rectangle (4,-2); 
 \draw(3.5,-1.5) node {$2$};
 
\draw[thick, fill=white] (0,0)--++(0:2)--++(-90:3)--++(180:1)--++(-90:6)--++(180:1)--++(90:9);

  \draw[white, fill=white] (-0.3,-3) rectangle (1.3,-4);
   \draw[thick] (0,-3) --++(0:2);
   \draw[thick] (0,-4) --++(0:1);   \clip (0,0)--++(0:2)--++(-90:3)--++(180:1)--++(-90:6)--++(180:1)--++(90:9);

\draw(0.5,-0.5) node {$0$};
\draw(1.5,-0.5) node {$1$};
\draw(2.5,-0.5) node {$2$};
\draw(0.5,-1.5) node {$4$};
\draw(1.5,-1.5) node {$0$};
\draw(2.5,-1.5) node {$1$};
\draw(0.5,-2.5) node {$3$};
\draw(1.5,-2.5) node {$4$};
\draw(2.5,-2.5) node {$0$};

\draw(0.5,-3.5) node {$2$};
\draw(1.5,-3.5) node {$3$};

\draw(0.5,-4.5) node {$1$};
\draw(1.5,-4.5) node {$2$};

\draw(0.5,-5.5) node {$0$};
\draw(1.5,-5.5) node {$1$};

\draw(0.5,-6.5) node {$4$};\draw(0.5,-7.5) node {$3$};\draw(0.5,-8.5) node {$2$};

	\foreach \i in {0,1,2,...,20}
	{
		\path  (0,0)++(-90:1*\i cm)  coordinate (a\i);
		\path  (0.5,-0.5)++(0:1*\i cm)  coordinate (b\i);
		\path  (0.5,-1.5)++(0:1*\i cm)  coordinate (c\i);
		\path  (0.5,-2.5)++(0:1*\i cm)  coordinate (d\i);
		\draw[thick] (a\i)--++(0:6);
		\draw[thick] (1,0)--++(-90:6);
		\draw[thick] (2,0)--++(-90:5);
	}

  \draw[white, fill=white] (-0.3,-3) rectangle (1.3,-4);
  
   \draw[thick] (0,-3) --++(0:2);
   \draw[thick] (0,-4) --++(0:1);

  \end{tikzpicture}  \end{minipage}   
\qquad\quad
    $$
  \caption{The   box configurations, $\la$, $Y_{[3,2,0]}(\la)$, and 
     and $Y^2_{[4,1,0]}(\la)$ 
  }
  \label{creg}
    \end{figure}

We remark that $\la \psucc \mu$  implies that $\la \succ \mu$ and so $\psucc$  is a coarsening of $\succ$.

  \begin{eg}\label{creg2}
Let  $e=5$ and $\ell=1$.  For $\la=(3,2^2,1^6) \in\mathscr{P}_{1}(13)$, we have that 
$Y_{[3,2,0]} (\la) =  (3,2,1^7) \cup [2,6,0] $.
 We have that $Y^2_{[3,2,0]} (\la) =  Y_{[2,6,0]}((3,2,1^7) \cup [2,6,0]) = (3,2,1^7) \cup [1,5,0] $.

 \end{eg}

\begin{eg}\label{refereback}
Let   $e=5$ and $\ell=1$ and $r=2$.  For $\la=(3,2^2,1^6) \in\mathscr{P}_{1}(13)$, we have  
$Y^1_{[4,1,0]} (\la) =  (3,2^2,1^6) \cup [3,5,0] - [4,1,0] $  and $Y^2_{[4,1,0]} (\la) =  (3,2^2,1^6) \cup [2,4,0] - [4,1,0] $ (see  \cref{creg}).
 \end{eg}

Given an idempotent  generator, $e_{\underline{j} }$ for $\underline{j} \in (\ZZ/e\ZZ)^n$, of   the KLR algebra, we wish to identify 
to which layer of our stratification our idempotent belongs. To this end we make the following definition.

  \begin{defn}
\label{makeatable}
 Associated to any  
$\underline{j} =  (j_1,\dots, j_n) \in (\ZZ/e\ZZ)^n$,  
we have an element ${\sf J} \in\Std(n) \cup \{0\}$ given by 
placing the entry $k= 1,2,\dots,n+1$ in the lowest addable $j_k$-box under  $\succ$ of the partition $\Shape({\sf J}_{\leq k-1})$, and formally setting   ${\sf J}=0$ and 
$\Shape({\sf J})=0$ if no such box exists for some $1\leq k\leq n$. 
  \end{defn}

 \begin{figure}[ht!!!]
\!\! $$  \begin{minipage}{3.52cm}
\scalefont{0.9} \begin{tikzpicture}
  [xscale=0.5,yscale=  
  0.5]

\draw[thick ](2,0) rectangle (3,-1);  
\draw(2.5,-0.5) node {$2$};

\draw[thick](-2,0.5) rectangle (7.5,-9.5);

\draw[thick, fill=white] (0,0)--++(0:2)--++(-90:3)--++(180:1)--++(-90:6)--++(180:1)--++(90:9);
  \clip (0,0)--++(0:2)--++(-90:3)--++(180:1)--++(-90:6)--++(180:1)--++(90:9);
\draw(0.5,-0.5) node {$0$};
\draw(1.5,-0.5) node {$1$};
\draw(2.5,-0.5) node {$2$};
\draw(0.5,-1.5) node {$4$};
\draw(1.5,-1.5) node {$0$};
\draw(2.5,-1.5) node {$1$};
\draw(0.5,-2.5) node {$3$};
\draw(1.5,-2.5) node {$4$};
\draw(2.5,-2.5) node {$0$};

\draw(0.5,-3.5) node {$2$};
\draw(1.5,-3.5) node {$3$};

\draw(0.5,-4.5) node {$1$};
\draw(1.5,-4.5) node {$2$};

\draw(0.5,-5.5) node {$0$};
\draw(1.5,-5.5) node {$1$};

\draw(0.5,-6.5) node {$4$};\draw(0.5,-7.5) node {$3$};\draw(0.5,-8.5) node {$2$};
 
	\foreach \i in {0,1,2,...,20}
	{
		\path  (0,0)++(-90:1*\i cm)  coordinate (a\i);
		\path  (0.5,-0.5)++(0:1*\i cm)  coordinate (b\i);
		\path  (0.5,-1.5)++(0:1*\i cm)  coordinate (c\i);
		\path  (0.5,-2.5)++(0:1*\i cm)  coordinate (d\i);
		\draw[thick] (a\i)--++(0:6);
		\draw[thick] (1,0)--++(-90:6);
		\draw[thick] (2,0)--++(-90:5);
	}

  \end{tikzpicture}  \end{minipage}  
   \qquad   \qquad\qquad
     \begin{minipage}{3.52cm}
\scalefont{0.9} \begin{tikzpicture}
  [xscale=0.5,yscale=  
  0.5]

\draw[thick ](2,0) rectangle (3,-1);  
\draw(2.5,-0.5) node {$13$};

\draw[thick ](-2,0.5) rectangle (7.5,-9.5);

\draw[thick, fill=white] (0,0)--++(0:2)--++(-90:3)--++(180:1)--++(-90:6)--++(180:1)--++(90:9);
  \clip (0,0)--++(0:2)--++(-90:3)--++(180:1)--++(-90:6)--++(180:1)--++(90:9);
\draw(0.5,-0.5) node {$1$};
\draw(1.5,-0.5) node {$2$};
 \draw(0.5,-1.5) node {$3$};
\draw(1.5,-1.5) node {$4$};
 \draw(0.5,-2.5) node {$5$};
\draw(1.5,-2.5) node {$6$};
 
\draw(0.5,-3.5) node {$7$};
 
\draw(0.5,-4.5) node {$8$};
 
\draw(0.5,-5.5) node {$9$};
 
\draw(0.5,-6.5) node {$10$};\draw(0.5,-7.5) node {$11$};\draw(0.5,-8.5) node {$12$};
 
	\foreach \i in {0,1,2,...,20}
	{
		\path  (0,0)++(-90:1*\i cm)  coordinate (a\i);
		\path  (0.5,-0.5)++(0:1*\i cm)  coordinate (b\i);
		\path  (0.5,-1.5)++(0:1*\i cm)  coordinate (c\i);
		\path  (0.5,-2.5)++(0:1*\i cm)  coordinate (d\i);
		\draw[thick] (a\i)--++(0:6);
		\draw[thick] (1,0)--++(-90:6);
		\draw[thick] (2,0)--++(-90:5);
	}

  \end{tikzpicture} 
    \end{minipage}  $$  
  \caption{Let $e=5$ and $\ell=1$. 
On the left, we have the residues for the partition $\la=(3,2^2,1^6)$.  On the right, we have the tableau
${\sf J} \in \Std(3,2^2,1^6)$ for  $\underline{j}=(0,1,4,0,3,4,2,1,0,4,3,2,2)$ as in \cref{makeatable}. 	 }
 \label{reducedeg224343432}
 \end{figure}
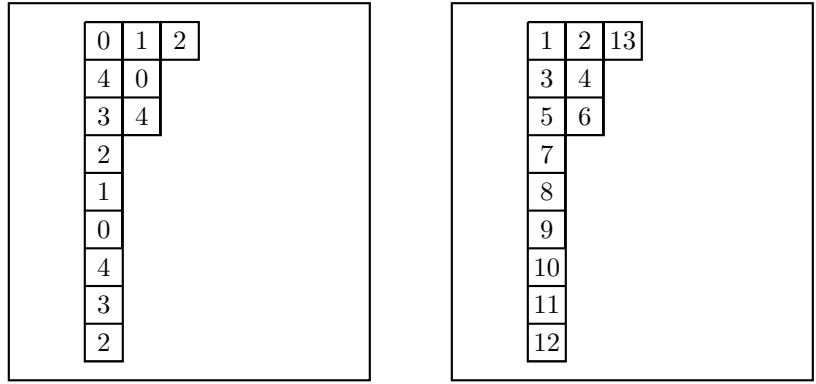

\subsection{A tableaux theoretic basis}
We are now ready to construct our first basis of $\mathscr{H}^\sigma_n $.  This basis will serve as the starting point for our light-leaves bases of Theorem A.  The combinatorics of this basis will be familiar to anyone who has studied symmetric groups and cyclotomic Hecke algebras (but with respect to the, less familiar, $(\succ)$-ordering). 

\begin{defn}\label{reddefn2} Let $  \la \in   \mathscr{P}_{\underline{\tau}}(n )$.  
We define the       {\sf  Garnir adjacency set}  of an $r$-box $\alpha=[i,j,m]\not \in \la$ 
to be the set of boxes, $\gamma\in   \la \cap {\sf Gar}_{\succ}(\alpha)$
 such that 
  $|{\rm res}(\gamma ) - r |\leq 1$ and denote this set by ${\sf Adj\text{-}Gar}_{\succ}(\alpha)$.  
  We set ${\rm res}( {\sf Adj\text{-}Gar}_{\succ}(\alpha))=\{  \res(\gamma ) \mid \gamma  \in {\sf Adj\text{-}Gar}_{\succ}(\alpha)
   \}$.  

    \end{defn}

         \color{black}

\!\! 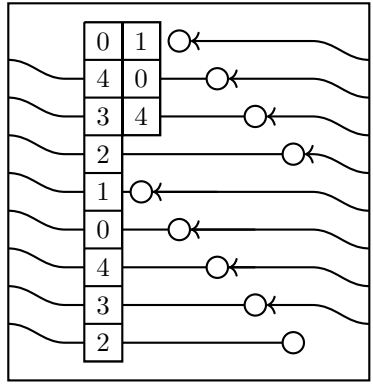
\begin{figure}[ht!]
    \begin{minipage}{3.52cm}
\scalefont{0.9} \begin{tikzpicture}
  [xscale=0.5,yscale=  
  0.5]

\draw[thick,fill=white](-2,0.5) rectangle (7.5,-9.5);

\draw[thick ](5.5,-8.5) to (-0.5,-8.5) to [out=180,in=0] (-2,-8); 
\draw[thick,->](7.5,-8)  to [out=180,in=0] (6,-7.5)--++(180:1.2) ;

\draw[thick ](4.5,-7.5) to (-0.5,-8.5+1) to [out=180,in=0] (-2,-8+1); 
\draw[thick,->](7.5,-8+1)  to [out=180,in=0] (6,-7.5+1)--++(180:2.2) ;

\draw[thick ](4.5,-7.5+1) to (-0.5,-8.5+1+1) to [out=180,in=0] (-2,-8+1+1); 
\draw[thick,->](7.5,-8+1+1)  to [out=180,in=0] (6,-7.5+1+1)--++(180:3.2) ;

\draw[thick ](4.5,-7.5+1+1) to (-0.5,-8.5+1+1+1) to [out=180,in=0] (-2,-8+1+1+1); 
\draw[thick,->](7.5,-8+1+1+1)  to [out=180,in=0] (6,-7.5+1+1+1)--++(180:4.2) ;

\draw[thick ](3.5,-7.5+1+1+1) to (-0.5,-8.5+1+1+1+1) to [out=180,in=0] (-2,-8+1+1+1+1); 
 \draw[thick,->](7.5,-8+1+1+1+1)  to [out=180,in=0] (6,-7.5+1+1+1+1)--++(180:0.2) ;

\draw[thick ](5.5,-7.5+1+1+1+1) to (-0.5,-8.5+1+1+1+1+1) to [out=180,in=0] (-2,-8+1+1+1+1+1); 
  \draw[thick,->](7.5,-8+1+1+1+1+1)  to [out=180,in=0] (6,-7.5+1+1+1+1+1)--++(180:1.2) ;

\draw[thick ](4.5,-7.5+1+1+1+1+1) to (-0.5,-8.5+1+1+1+1+1+1) to [out=180,in=0] (-2,-8+1+1+1+1+1+1); 
  \draw[thick,->](7.5,-8+1+1+1+1+1+1)  to [out=180,in=0] (6,-7.5+1+1+1+1+1+1)--++(180:2.2) ;

\draw[thick ](3.5,-7.5+1+1+1+1+1+1) to (-0.5,-8.5+1+1+1+1+1+1+1) to [out=180,in=0] (-2,-8+1+1+1+1+1+1+1); 
  \draw[thick,->](7.5,-8+1+1+1+1+1+1+1)  to [out=180,in=0] (6,-7.5+1+1+1+1+1+1+1)--++(180:3.2) ;

\draw[thick,fill=white](1.5,-2.5-2) circle (8pt)  ;
\draw[thick,fill=white](2.5,-2.5-3) circle (8pt)  ;
\draw[thick,fill=white](3.5,-3.5-3) circle (8pt);  
\draw[thick,fill=white](4.5,-4.5-3) circle (8pt);  
\draw[thick,fill=white](5.5,-5.5-3) circle (8pt);  
\draw[thick,fill=white](5.5,-2.5-1) circle (8pt)  ;
\draw[thick,fill=white](4.5,-2.5 ) circle (8pt)  ;
\draw[thick,fill=white](3.5,-1.5 ) circle (8pt)  ; 
\draw[thick,fill=white](2.5,-0.5 ) circle (8pt)  ;

\draw[thick, fill=white] (0,0)--++(0:2)--++(-90:3)--++(180:1)--++(-90:6)--++(180:1)--++(90:9);
  \clip (0,0)--++(0:2)--++(-90:3)--++(180:1)--++(-90:6)--++(180:1)--++(90:9);
\draw(0.5,-0.5) node {$0$};
\draw(1.5,-0.5) node {$1$};
\draw(2.5,-0.6) node {$2$};
\draw(0.5,-1.5) node {$4$};
\draw(1.5,-1.5) node {$0$};
\draw(2.5,-1.5) node {$1$};
\draw(0.5,-2.5) node {$3$};
\draw(1.5,-2.5) node {$4$};
\draw(2.5,-2.5) node {$0$};

\draw(0.5,-3.5) node {$2$};
\draw(1.5,-3.5) node {$3$};

\draw(0.5,-4.5) node {$1$};
\draw(1.5,-4.5) node {$2$};

\draw(0.5,-5.5) node {$0$};
\draw(1.5,-5.5) node {$1$};

\draw(0.5,-6.5) node {$4$};\draw(0.5,-7.5) node {$3$};\draw(0.5,-8.5) node {$2$};
 
	\foreach \i in {0,1,2,...,20}
	{
		\path  (0,0)++(-90:1*\i cm)  coordinate (a\i);
		\path  (0.5,-0.6)++(0:1*\i cm)  coordinate (b\i);
		\path  (0.5,-1.5)++(0:1*\i cm)  coordinate (c\i);
		\path  (0.5,-2.5)++(0:1*\i cm)  coordinate (d\i);
		\draw[thick] (a\i)--++(0:6);
		\draw[thick] (1,0)--++(-90:6);
		\draw[thick] (2,0)--++(-90:5);
	}

  \end{tikzpicture} 
    \end{minipage}   
  \caption{ For $n=13$ and $\la =(2^3,1^6)$,
 we illustrate how   the idempotent, $e_{(0,1,4,0,3,4,2,1,0,4,3,2,2)}$, labelled by ${\sf J}$ in  \cref{reducedeg224343432}
 is rewritten in the form (\ref{Claim 1}).  
     The box moves through each row  until it comes to rest at the point    ${\sf J}^{-1}(13)=[1,3,0]$. 
     This involves repeated applications of \cref{Claim 3} to deduce  (\ref{Claim 1}).  
   For the purposes of later referencing, we label the 9 boxes from bottom-to-top by $\alpha_1=[9,6,0],\alpha_2=[8,5,0] , \dots ,   \alpha_9=[1,3,0]$.  
   This visualisation is explained in \cref{endeavour11}.  
    }
 \label{reducedeg222}
  \label{reducedeg2223}
 \end{figure}
 
\begin{eg}
Continuing with \cref{continuation}, 
we have that 
$  {\sf Adj\text{-}Gar}_{\succ}([3,3,1]) =\{[3,2,1],[2,3,1]\}$.  
\end{eg}

\begin{rmk}\label{endeavour11}
 We are endeavouring to construct a 2-sided chain  of ideals of $\mathscr{H}_n^\sigma$,  ordered by $\psucc$,   
 in which each 2-sided ideal is generated by an idempotent $e_{\SSTT_\la}$ for $\la \in \mathscr{P}_{\underline{\tau}}(n)$.  
 \Cref{Claim 3,Claim 2} of \cref{mainproof} will allow us  to rewrite any element of 
 $\mathcal{Y}_n $ in the required form by moving a given box $\square$ through the partition $\la$ one row at a time along the $\psucc$ ordering until it comes to rest at some point  $\la\cup\square'\in \mathscr{P}_{\underline{\tau}}(n )$.
For ${\sf J}$   the tableau in \cref{reducedeg224343432}, then   the eight steps involved in 
  rewriting $e_{\sf J}$ as an element of ${\mathscr{H}}^{\succeq (3,2^2,1^6)}$ 
   are illustrated in 
\cref{reducedeg2223}. 

\end{rmk}

The following theorem is the technical heart of  this section.  
We define  $\mathscr{H}^ {\geq \la }=
\mathscr{H}^\sigma_n\langle e_{\stt_\nu} \mid \nu \geq  {\la }\rangle\mathscr{H}^\sigma_n$ (respectively $\mathcal{H}^ {\geq \la }=
\mathcal{H}^\sigma_n\langle e_{\stt_\nu} \mid \nu \geq  {\la }\rangle\mathcal{H}^\sigma_n$)  for $\geq$ any ordering  on $ \mathscr{P}_{\underline{\tau}}(n)$ (respectively $\mathscr{P} _\ell(n)$).  
Given $\geq $ a total order on $ \mathscr{P}_{\underline{\tau}}(n)$,  we let 
 $
\la^{[0]} > \la^{[1]} >\dots  >\la^{[m]}
$ 
denote the complete set of elements of $ \mathscr{P}_{\underline{\tau}}(n)$ enumerated according to the   total ordering $>$. 
Given $\la \in \mathcal{B}_\ell (n-1)$ and $\alpha\not \in \la$, we  
define 
$$\mathscr{H}^{(\ppsucc  \la)\cup\alpha }_{n} =\!\! \sum_{
		{   
		\begin{subarray}c 
		\{ \mu \in  \mathcal{B}_\ell  (n-1) \mid \mu  \hspace{1pt}\ppsucc   \la, 
	\alpha  \not \in	\mu    \}	 	\end{subarray}   	}
	 }		\!\!\!\!    	
 \mathscr{H}_n^ {\sigma }  e_{\SSTT_{\mu\cup\alpha }}		 \mathscr{H}^ {\sigma }_n 
  \leq \mathscr{H}_n^ { \ppsucc  (\la \cup\alpha )  }.
  $$

         \begin{thm} \label{mainproof}\label{mainproof2}
     Let   $\la\in  \mathscr{P}_{\underline{\tau}}(n-1 )$ and 
    assume  $ \alpha = [i,j,m] \not \in \la$ is left-justified. 
    We set   $a= \SSTT_{\la\cup\alpha}(\alpha)  	 $.  
    \begin{itemize}[leftmargin=*]
 
 \item[$(a)$]If $\alpha= [1,j,0] $ for some $j\geq 1$ then 
  $(i)$       $    \WHY_{\SSTT_{\la \cup\alpha}} =0$ if $\la \cup \alpha \not \in \mathscr{P}_{\underline{\tau}}(n)$ and
  $(ii)$  $    y_{a}  \WHY_{\SSTT_{\la \cup\alpha}} =0$ if 
  $\la \cup \alpha  \in \mathscr{P}_{\underline{\tau}}(n)$.  

 \item[$(b)$]
 
 For $\alpha \neq [1,j,0] $ for some $j\geq 1$, we set $\beta$ to be  the box  determined by 
\begin{equation}\label{L+!!!!!}
  \begin{cases}  
   Y^{\ell+1}_\alpha(\la\cup\alpha) 	= 
   \la\cup  \beta     &\text{ if ${\rm res}( {\sf Adj\text{-}Gar}(\alpha))=\{r-1\} $;}  
      \\
 Y^1_\alpha(\la\cup\alpha) 	= 
   \la\cup  \beta 	&\text{ otherwise } 
    	\end{cases}	 
	\end{equation}
 and we set  $b=   \SSTT_{\la \cup  \beta   } (\beta)	 $.   
 \begin{itemize}[leftmargin=*]
 
\item[$\bullet$]   If $\la\cup \alpha \not \in \mathscr{P}_{\underline{\tau}}(n)$,    
 then we have that 
\begin{equation}\label{Claim 3}
  \WHY^{\succ}_{\stt_{\la\cup\alpha}}
  \in 
\begin{cases}
 \pm \psi^a_b  
\WHY^{\succ}_{\stt_{\la\cup\beta}}
 \psi^b_a
+  \mathscr{H}_n^ {(\ppsucc  \la)\cup\beta  }    & \text{or} 
\\ 
 \pm (y_{a-1} \psi^a_b  
\WHY^{\succ}_{\stt_{\la\cup\beta}}
 \psi^b_a
 -
  \psi^a_b  
\WHY^{\succ}_{\stt_{\la\cup\beta}}
 \psi^b_a  y_{a})
+  \mathscr{H}_n^ {(\ppsucc  \la)\cup\beta  }   
\end{cases} 
 \end{equation}
   (the cases are detailed in the proof).  

 \item[$\bullet$] 
If $\la\cup \alpha \in \mathscr{P}_{\underline{\tau}}(n)$, then we have that 
  \begin{align}\label{Claim 2}
    y_{a} (\WHY^{{\succ}}_{\SSTT_{\la \cup\alpha}}) \in 
    \pm 	\psi^a_b 		\WHY^{\succ}_{\SSTT_{\la \cup  \beta   }  } \psi^b_a   +  \mathscr{H}_n^ {(\ppsucc  \la)\cup\beta  } 
  \end{align} 
\end{itemize}
\item[$(c)$]
For $\underline{j}\in (\ZZ/e\ZZ)^{n }$,  using the notation of \cref{makeatable}, 
we have that: 
  \begin{align}  \label{Claim 1}
\begin{cases}
e_{\underline{j} }  \in  \pm 
\tilde{\psi}_{\stt_{\nu }}^{\sf J} \tilde{\psi}^{\stt_{\nu }}_{\sf J} + 
 \mathscr{H}^{ \ppsucc\nu  }
   &\text{if }  \Shape( {\sf J} )=\nu  \in \mathscr{P}_{\underline{\tau}}(n) \\
e_{\underline{j} }  =  0   &\text{if }  \Shape( {\sf J} )= 0
  \end{cases}
      \end{align}      where $\tilde{\psi}_{\stt_{\nu }}^{\sf J} $ is obtained from $ {\psi}_{\stt_{\nu }}^{\sf J} $ by possibly adding some dot decorations along the strands.  
 Thus if  $\geq $ is any total refinement  of $\psucc$ 
 then the $\ZZ$-algebra $\mathscr{H}^\sigma_n$
  has a chain of two-sided ideals  
\begin{equation}\label{Claim 4}0  \subset  \mathscr{H}^{\geq  \la^{[0]}}_n
  \subset  \mathscr{H}^{\geq  \la^{[1]}}_n
  \subset \dots \subset  \mathscr{H}^{\geq  \la^{[m]}}_n= \mathscr{H}^\sigma_n .\end{equation}
  \end{itemize}     \end{thm}

 \begin{rmk} 
In the proof, we will often relate ideals 
in smaller and larger algebras using horizontal concatenation of diagrams, this is made possible by the definition of the reverse cylindric  ordering $\succ$ (which distinguishes between box configurations based on the first discrepancy upon reading a pair of box configurations {\em backwards}).  
    \end{rmk}

\begin{rmk}    If $\la \cup \alpha \in  \mathscr{P}_{\underline{\tau}}(n)$,  
then $\WHY^{{\succ}}_{{{\SSTT_{\la\cup\alpha}}}}=e_{\SSTT_{\la\cup\alpha}}$ by \cref{easy1}.
    \end{rmk}

\begin{proof}[Proof of \cref{mainproof}]
Part $(a)$, let  $\alpha= [1,j,0] $ for some $j\geq 1$. Claim $(ii)$ follows by    applying   case~3 of \ref{rel4} $a$ times, followed by the commuting KLR relations and the cyclotomic relation.  
The proof of  claim $(i)$ is similar. 
Thus $(a)$ follows.  

For parts  $(b)$ and $(c)$, 
we assume that \cref{Claim 2,Claim 3,Claim 1,Claim 4} all hold for rank $n-1$.    
 We   further  assume that    \cref{Claim 2,Claim 3} have been proven for all $\nu=\mu\cup\alpha$  with $\mu\in \mathscr{P}_{\underline{\tau}}(n-1)$ such that 
  $\mu \psucc \la$;
   thus leaving us to  prove \cref{Claim 2,Claim 3} for  $\nu=\la\cup\alpha $ for $ \la\in \mathscr{P}_{\underline{\tau}}(n-1)$ and \cref{Claim 1} for all $\underline{j}\in (\ZZ/e\ZZ)^{n }$.     
\Cref{Claim 4} follows immediately from \cref{Claim 1} and the idempotent decomposition of the identity in  relation \ref{rel1}.
 By \cref{reddefn},   $\res(\alpha)=\res(\beta)$ and we set this residue equal to $r \in \ZZ/e\ZZ$.


\smallskip

 \noindent {\bf Proof of  \cref{Claim 3} for a given $\la$ and $\alpha$}.   
   We include a running example    for  $e=5$ and $\ell=1$ and $\la=(2^3,1^6)$.    We assume that $\la \cup\alpha \not \in \mathscr{P}_\aatchpair(n)$.  
 There are four cases to consider, depending on the residue of     $\alpha':=\SSTT_{\la\cup\alpha}^{-1}(a-1)$.  We assume that $\alpha'$ is in the same row of the same component as $\alpha$ (as otherwise 
  $y_{\SSTT_{\la\cup \alpha}}=y_{\SSTT_{\la\cup \beta}}$ by definition).  
  We let $\beta'$  be  the box  determined by 
 $$
  \begin{cases}  \label{popopopopopopopo}
     Y^{\ell+1}_{\alpha'}(\la\cup\alpha') 	= 
   \la\cup  \beta '    &\text{ if ${\rm res}( {\sf Adj\text{-}Gar}(\alpha'))=\{r-1\} $;}  
      \\
 Y^1_{\alpha'}(\la\cup\alpha') 	= 
   \la\cup  \beta' 	&\text{ otherwise.} 
    	\end{cases}	 $$
    \begin{itemize}[leftmargin=*]

 \item[$(i)$]
Suppose $\alpha'  $ has residue $r\in \ZZ/e\ZZ$ and so $ \WHY^{{\succ}}_{\stt_{\la\cup\alpha}}=e_{\stt_{\la\cup\alpha}}$ by \cref{easy1}.
   By application of relations \ref{rel3} and \ref{rel4}, we have that 
\begin{align}\label{kkkkkkkk} 
e_{\stt_{\la\cup\alpha}}
&=\psi
_{\kay-1} 
 y_{\kay-1}e_{\stt_{\la\cup\alpha}}\psi
  _{\kay}
y_{\kay-1} - y_{\kay }  \psi_{\kay-1}y_{\kay-1}e_{\stt_{\la\cup\alpha}}\psi_{\kay-1} .
\end{align} 
  An example   of the visualisation of the idempotents on the  righthand-side  of \cref{kkkkkkkk} is   given in the first step  of \cref{reducedeg2223};   the corresponding  righthand-side of \cref{kkkkkkkk} is depicted in \cref{adjustmntexample55}.
  Now, we have that 
$$
y_{\kay-1}e_{\stt_{\la\cup\alpha}}
=
(y_{\kay-1}e_{\stt_{\la \cup\alpha }   \downarrow_{\leq a-1} }) \boxtimes 
e_r \boxtimes  e_{\stt_{\la \cup\alpha }   \downarrow_{> a} } 
$$
and so  by our  inductive assumption for \cref{Claim 2} for rank $a-1<n$, we have that 
\begin{align*}
y_{\kay-1}e_{\stt_{\la\cup\alpha}}
&\in  \pm 
\psi^{a-1}_b 
 \WHY^{{\succ}}_{\stt_{\la \cup\beta }   \downarrow_{\leq a-1} }  \psi_{a-1}^b  \boxtimes 
 e_{\stt_{\la \cup\beta }   \downarrow_{\geq a} } 
  + \mathscr{H}_n^ {(\ppsucc  \la)\cup\beta  } 
   =
 \pm   \psi^{a-1}_b \WHY^{{\succ}}_{\stt_{\la \cup\beta }   }  \psi_{a-1}^b    + \mathscr{H}_n^ {(\ppsucc  \la)\cup\beta  } 
\end{align*}
 where we have implicitly used the following facts:
$(i)$  
 $Y_{\alpha'} (\la\cup \alpha)=\la\cup\alpha\cup\beta-\alpha' $ 
$(ii)$  ${\stt_{\la \cup\alpha }   {\downarrow}_{> a} }= {\stt_{\la \cup\beta }   {\downarrow}_{> a} } 
  $ and 
  $(iii)$ 
 $e_{\SSTT_{\la\cup\beta\cup\alpha-\alpha'}}=
e_{\SSTT_{\la\cup\beta }}$.     
 Substituting this back into \cref{kkkkkkkk}, we obtain 
 \begin{align}
e_{\stt_{\la\cup\alpha}}
&\in   
 \pm (  \psi^{\kay}_b 
\WHY^{{\succ}}_{\stt_{\la\cup\beta}}
\psi_ {\kay}^b y_{a-1}
- y_{\kay}
\psi^{\kay}_b 
\WHY^{{\succ}}_{\stt_{\la\cup\beta}}
\psi_ {\kay}^b  ) + \mathscr{H}_n^ {(\ppsucc  \la)\cup\beta  } 
\end{align}
as required (as in the second case in \cref{Claim 3}).   An example is depicted in \cref{adjustmntexample55} (although we remark that the error terms belonging to $\mathscr{H}_n^ {(\ppsucc  \la)\cup\beta  } $ are actually all zero in this case).  
  
 \begin{figure}[ht!] 
  $$  
 \begin{minipage}{7.7cm}\begin{tikzpicture}
  [xscale=0.45,yscale=  
  0.45]

 \draw (3,0) rectangle (16,3);
  
   \foreach \i in {3.5,4.5,...,15.5}
  {
   \fill[black](\i,0) circle(1.5pt); 
      \fill[black](\i,3)  circle(1.5pt);
       }

 \draw[black] (3.5-3+3,-0.6) node  {$0$};
\draw[black] (4.5-3+3,-0.6) node  {$1$};
\draw[black] (5.5-3+3,-0.6) node  {$4$};
\draw[black] (3.5+3,-0.6) node  {$0$};
\draw[black] (4.5+3,-0.6) node  {$3$};
\draw[black] (5.5+3,-0.6) node  {$4$};
 \draw[black] (6.5+3,-0.6) node  {$2$};
\draw[black] (8.5+2,-0.6) node  {$1$};
\draw[black] (7.5+3+1,-0.6) node  {$0$};
  \draw[black] (6.5+6,-0.6) node  {$4$};
\draw[black] (7.5+6,-0.6) node  {$3$};
\draw[black] (8.5+6,-0.6) node  {$2$};
\draw[black] (9.5+6,-0.6) node  {$2$};

\draw[black] (0.5,3.3) node {$\scalefont{0.8}  \stt_{\la \cup \alpha_1} $}; 

\draw[black] (0.5,1.5) node {$\scalefont{0.8}  \stt_{\la \cup \alpha_2} $};

\draw[black] (0.5,-0.3) node {$\scalefont{0.8}  \stt_{\la \cup \alpha_1} $};

\draw(14.5,3) to [out=-90,in=90] (15.5,1.5)  
 to [out=-90,in=90] (14.5,0) ; 
 
 \draw(15.5,3) to [out=-90,in=90] (14.5,1.5) 
 to [out=-90,in=90] (15.5,0) ; 

 \fill (14.5,1.5) circle (4.5pt);

 \fill (14.65,0.39) circle (4.5pt);

    \foreach \i in {3.5,4.5,...,13.5}
 {
\draw(\i,0)--++(90:1.5*2);
 }

\draw[black] (3.5-3+3,3.6) node  {$0$};
\draw[black] (4.5-3+3,3.6) node  {$1$};
\draw[black] (5.5-3+3,3.6) node  {$4$};
\draw[black] (3.5+3,3.6) node  {$0$};
\draw[black] (4.5+3,3.6) node  {$3$};
\draw[black] (5.5+3,3.6) node  {$4$};
 \draw[black] (6.5+3,3.6) node  {$2$};
\draw[black] (8.5+2,3.6) node  {$1$};
\draw[black] (7.5+3+1,3.6) node  {$0$};
  \draw[black] (6.5+6,3.6) node  {$4$};
\draw[black] (7.5+6,3.6) node  {$3$};
\draw[black] (8.5+6,3.6) node  {$2$};
\draw[black] (9.5+6,3.6) node  {$2$}; 
 
    \end{tikzpicture} 
    \end{minipage}
  \;\;\;  -\; \; 
    \begin{minipage}{7.1cm}\begin{tikzpicture}
  [xscale=0.45,yscale=  
  0.45]

 \draw (3,0) rectangle (16,3);
  
   \foreach \i in {3.5,4.5,...,15.5}
  {
   \fill[black](\i,0) circle(1.5pt); 
      \fill[black](\i,3)  circle(1.5pt);
       }

 \draw[black] (3.5-3+3,-0.6) node  {$0$};
\draw[black] (4.5-3+3,-0.6) node  {$1$};
\draw[black] (5.5-3+3,-0.6) node  {$4$};
\draw[black] (3.5+3,-0.6) node  {$0$};
\draw[black] (4.5+3,-0.6) node  {$3$};
\draw[black] (5.5+3,-0.6) node  {$4$};
 \draw[black] (6.5+3,-0.6) node  {$2$};
\draw[black] (8.5+2,-0.6) node  {$1$};
\draw[black] (7.5+3+1,-0.6) node  {$0$};
  \draw[black] (6.5+6,-0.6) node  {$4$};
\draw[black] (7.5+6,-0.6) node  {$3$};
\draw[black] (8.5+6,-0.6) node  {$2$};
\draw[black] (9.5+6,-0.6) node  {$2$};

\draw(14.5,3) to [out=-90,in=90] (15.5,1.5)  
 to [out=-90,in=90] (14.5,0) ; 
 
 \draw(15.5,3) to [out=-90,in=90] (14.5,1.5) 
 to [out=-90,in=90] (15.5,0) ; 

 \fill (14.5,1.5) circle (4.5pt);

 \fill (15.5-.15,3-0.39) circle (4.5pt);

\draw[black] (12.5+6,3.3) node {$\scalefont{0.8}  \stt_{\la \cup \alpha_1} $}; 

\draw[black] (12.5+6,1.5) node {$\scalefont{0.8}  \stt_{\la \cup \alpha_2} $};

\draw[black] (12.5+6,-0.3) node {$\scalefont{0.8}  \stt_{\la \cup \alpha_1} $};

    \foreach \i in {3.5,4.5,...,13.5}
 {
\draw(\i,0)--++(90:1.5*2);
 }

\draw[black] (3.5-3+3,3.6) node  {$0$};
\draw[black] (4.5-3+3,3.6) node  {$1$};
\draw[black] (5.5-3+3,3.6) node  {$4$};
\draw[black] (3.5+3,3.6) node  {$0$};
\draw[black] (4.5+3,3.6) node  {$3$};
\draw[black] (5.5+3,3.6) node  {$4$};
 \draw[black] (6.5+3,3.6) node  {$2$};
\draw[black] (8.5+2,3.6) node  {$1$};
\draw[black] (7.5+3+1,3.6) node  {$0$};
  \draw[black] (6.5+6,3.6) node  {$4$};
\draw[black] (7.5+6,3.6) node  {$3$};
\draw[black] (8.5+6,3.6) node  {$2$};
\draw[black] (9.5+6,3.6) node  {$2$}; 
 
    \end{tikzpicture} 
    \end{minipage}     $$  
    
    \!\!\!
    \caption{  We continue with the example in \cref{reducedeg222} 
     for $\la=(2^6,1^3)$.   
    This is the   righthand-side of \cref{kkkkkkkk}    for $ e_{\SSTT_{\la\cup\alpha_1}}$.  
  }
\label{adjustmntexample55}    \end{figure}
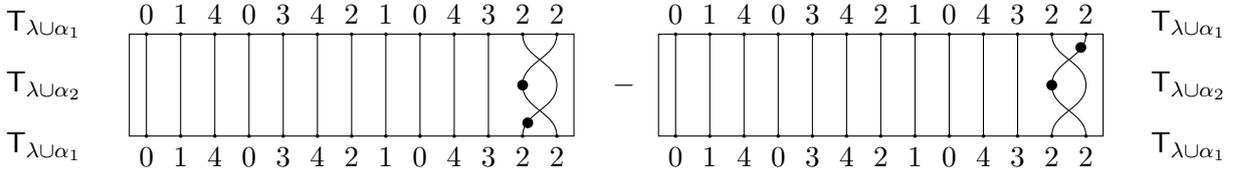

\item  [$(ii)$] 
Now suppose   $\alpha'=[x,y,z]$ has residue $r+1\in \ZZ/e\ZZ$.  We have two subcases to consider.   
We first consider the easier subcase, in which 
  $  
  [x,y,z]=[i,1,m]$ and so $b=a-1$.
We have that  $ [i+1,1,m] \in {\rm Add}_r (\SSTT_{\la\cup\alpha}{\downarrow}_{\leq a})$
whereas  $  {\rm Add}_r (\SSTT_{\la\cup\beta}{\downarrow}_{\leq b})=\emptyset$.   
By  relation \ref{rel4}, we have that 
$$\WHY^{{\succ}}_{\stt_{\la\cup\alpha}}=
y_a e_{\stt_{\la\cup\alpha}}
=
y_{a-1}e_{\stt_{\la\cup\alpha}}
-
e_{\stt_{\la\cup\alpha}}\psi_{a-1} 
e_{\stt_{\la\cup\beta}}
\psi_{a-1} e_{\stt_{\la\cup\alpha}}
=
y_{a-1}e_{\stt_{\la\cup\alpha}}
-
 \psi^a_b 
y^{\succ}_{\stt_{\la\cup\beta}}
\psi^b_{a}.  
$$
%
%
We have that  $y_{a-1}\WHY_{\stt_{\la }} 
 \in  \mathscr{H}_{n-1}^{ \ppsucceq \la \cup\beta'-\alpha' }
$
and so the former term  is of the required form   by our assumption for $\la\cup\beta'-\alpha'=\mu \psucc \la$.

Now, if $y>1$  then the $(\kay-2)$th, $(\kay-1)$th  and  $\kay$th strands have residues $r$, $r+1$, and $r$ respectively.  
We have that 
 \begin{align}\notag
\WHY^{\succ}_{\stt_{\la\cup\alpha}}=e_{\stt_{\la\cup\alpha}}
&= 
e_{\stt_{\la\cup\alpha}}
\psi_{\kay-2} \psi_{\kay-1} \psi_{\kay-2} 
e_{\stt_{\la\cup\alpha}}  
-
e_{\stt_{\la\cup\alpha}}
  \psi_{\kay-1}  \psi_{\kay-2} \psi_{\kay-1}
  e_{\stt_{\la\cup\alpha}}  
\\[3pt]
\label{kkkkkkkk2}
&= 
- e_{\stt_{\la\cup\alpha}}
\psi_{\kay-2} \psi_{\kay-1}y_{\kay-1} \psi_{\kay-1} \psi_{\kay-2} 
e_{\stt_{\la\cup\alpha}}  
+
e_{\stt_{\la\cup\alpha}}
  \psi_{\kay-1}  \psi_{\kay-2} y_{\kay-2} \psi_{\kay-2}  \psi_{\kay-1}
  e_{\stt_{\la\cup\alpha}}.   
    \end{align}
    where the first equality follows from \cref{easy1}, the second from  relation~\ref{rel5} and the third follows from relations  \ref{rel3} and \ref{rel4}.  
We set $\xi=\Shape(\SSTT_\la{\downarrow}_{<\kay-2})$.  The  two terms in \cref{kkkkkkkk2}   factor through the elements 
\begin{equation}\label{hjhjh}
 \underbrace{    e_{\SSTT_\lambda {\downarrow}_{<\kay-2}}    
   \boxtimes 
   e_{r+1}}_{\xi\cup[x,y,z]  }  \boxtimes y_1e_{r,r}
   \boxtimes 
      e_{\SSTT_{\lambda\cup\alpha} {\downarrow}_{> a}}
 \qquad 
  \underbrace{  e_{\SSTT_\lambda {\downarrow}_{<\kay-2}}    
   \boxtimes 
 y_1  e_r 	}_{\xi\cup[x,y-1,z]}	 \boxtimes e_{r,r+1}
   \boxtimes 
      e_{\SSTT_{\lambda\cup\alpha} {\downarrow}_{> a}}
\end{equation} respectively.

\begin{itemize}[leftmargin=*]
\item [$\bullet$]We first consider the latter term on the righthand-side of \cref{kkkkkkkk2} (which we will see, is the required non-zero term).  
 We note that $[x,y-1,z]$ and $\alpha$ have the same residue and so  
$Y_{[x,y-1,z]} (\xi  \cup {[x,y-1,z]})=  \xi \cup \beta $.  By our  inductive assumption  that  \cref{Claim 2} holds for rank $a-2<n$, we have that 
\begin{align}\label{ssjlkfgflkgjdflg}
 e_{\SSTT_\lambda {\downarrow}_{<\kay-2}}    
   \boxtimes 
 y_1  e_r 
 = y_{a-2} e_{\SSTT_{\xi    \cup {[x,y-1,z]} } }
 \in  
 \pm 
  \psi_b^{a-2}   \WHY^{{\succ}}_{\stt_{\xi \cup \beta}}   \psi^b_{a-2}     
  + \mathscr{H}_{a-2}^{(\ppsucc  \xi)\cup\beta  } 
 \end{align}
Substituting this back into the  second term of   \cref{hjhjh} we obtain 
$$
 \psi_b^{a-2}  \WHY^{{\succ}}_{\stt_{\xi \cup \beta}}    \psi^b_{a-2}  
 \boxtimes e_{r,r+1}
   \boxtimes 
      e_{\SSTT_{\lambda\cup\alpha} {\downarrow}_{> a}}\in   
 \pm  \psi_b^{a-2}   \WHY^{{\succ}}_{\stt_{\la\cup\beta}}
      \psi^b_{a-2}
        + \mathscr{H}_{n}^{(\ppsucc  \la)\cup\beta  } 
    $$
and then  substituting  into the  second term of   \cref{kkkkkkkk2} we obtain
\begin{align*}
e_{\stt_{\la\cup\alpha}}
  \psi_{\kay-1}  \psi_{\kay-2} y_{\kay-2} \psi_{\kay-2}  \psi_{\kay-1}
  e_{\stt_{\la\cup\alpha}}
  & \in
 \pm    \psi_{a-1} \psi_{a-2} 
 \psi_b^{a-2}   \WHY^{{\succ}}_{\stt_{\la\cup\beta}}
      \psi^b_{a-2}
        \psi_{a-2} \psi_{a-1} 
       + \mathscr{H}_{n}^{(\ppsucc  \la)\cup\beta  } 
\\  & = \pm \psi_b^{a}  \WHY^{{\succ}}_{\stt_{\la\cup\beta}}
      \psi^b_{a}   + \mathscr{H}_{n}^{(\ppsucc  \la)\cup\beta  } 
\end{align*}as required.

\item [$\bullet$] We now consider the former term of \cref{kkkkkkkk2} (which, we will see, is zero modulo   the ideal).  
 We have that    $Y_{[x,y,z]}	(\xi\cup[x,y,z]) = ( \xi\cup 	\gamma )\psucc (\xi\cup[x,y,z])$  for 
 $\gamma$  a box of residue $r+1\in \ZZ/e\ZZ$.  
 We set $c=\SSTT_{\xi \cup \gamma}^{-1}(\gamma)$.   
   We have that 
 $$
  e_{\SSTT_\xi   }  \boxtimes    e_{r+1} 
 = 
 e_{\SSTT_{\xi   \cup [x,y,z]   }}
 \in \pm 
\psi_c^{a-2} \WHY^{{\succ}}_{\SSTT_{\xi   \cup\gamma}}\psi  _{a-2}^c
  +   \mathscr{H}^{(\ppsucc \xi ) \cup \gamma}_{a-2}    $$
by our  inductive assumption  that  \cref{Claim 3} holds for rank $a-2<n$.
Concatenating, we have  that 
 $$
      e_{\SSTT_\lambda {\downarrow}_{<\kay-2}}    
   \boxtimes 
   e_{r+1}  \boxtimes y_1e_{r,r}
   \boxtimes 
      e_{\SSTT_{\lambda\cup\alpha} {\downarrow}_{> a}}
       \in\pm 
    \psi_c^{a-2} 
      \WHY^{{\succ}}_{\SSTT_{\lambda  \cup \gamma\cup \alpha -[x,y,z]}} 
   \psi^c_{a-2}      +\mathscr{H}_{n}^{(\ppsucc ( \la-[x,y,z]\cup \gamma)      )\cup \alpha }
 $$
 and we note that the idempotent on the righthand-side belongs to the ideal $\mathscr{H}_{n}^{(\ppsucceq( \la-[x,y,z]\cup \gamma))      \cup \alpha }$ 
 and so the result follows by our assumption for 
 $ \lambda  \cup \gamma   -[x,y,z]  =\mu  \psucc \la  $.  
     \color{black}
 \end{itemize}

   \item[$(iii)$]
Now suppose $\alpha' $ has residue $d\in \ZZ/e\ZZ$ such that $|d-r|>1$.    We  set $\xi=\Shape(\SSTT_\la{\downarrow}_{< \kay-1})$.  
By case 2 of relation~\ref{rel4}, we have that  
\begin{align}\label{kkkkkkkk3}
  y_a^k e_{\stt_{\la\cup\alpha}} 
&=   \psi_{\kay-1} \big( \underbrace{e_{\SSTT_\la{\downarrow}_{<\kay-1}}\boxtimes
  y_1^k e_{r } }_{\xi\cup\alpha}\boxtimes e_{d} \boxtimes e_{\SSTT_{\la\cup\alpha}{\downarrow}_{>a}}   \big)\psi_{\kay-1}
  \end{align}
for $k\in\{0,1\}$. By the  inductive  assumption for rank $a-1<n$ of \cref{Claim 3}, we have that 
$$
\left.\begin{array}{rr}
\text{  if }\alpha'  =[i,2,m] \text{ and } d=r+2 \text{  then}	&		e_{\SSTT_\la{\downarrow}_{<\kay-1}}\boxtimes y_1  e_{r }
 \\
\text{otherwise} & e_{\SSTT_\la{\downarrow}_{<\kay-1}}\boxtimes  e_{r } 
\end{array}\right\}
\in\pm  \psi^{a-1}_b \WHY^{{\succ}}_{\SSTT_{\xi\cup\beta}} \psi_{a-1}^b + 
\mathscr{H}_{a-1}^{(\ppsucc  \xi) \cup \beta}
 $$
 As in the case $(ii)$ above, we concatenate to deduce the result.  
  Two examples of the visualisation of the    righthand-side  of \cref{kkkkkkkk3} are   given in the third and fourth steps of \cref{reducedeg2223};   the corresponding elements are depicted in \cref{adjustmntexample44}.

 \begin{figure}[ht!] 
  $$   
     \begin{minipage}{7.5cm}\begin{tikzpicture}
  [xscale=0.45,yscale=  
  0.45]

 \draw (3,0) rectangle (16,3);
  
   \foreach \i in {3.5,4.5,...,15.5}
  {
   \fill[black](\i,0) circle(1.5pt); 
      \fill[black](\i,3)  circle(1.5pt);
       }

 \draw[black] (3.5-3+3,-0.6) node  {$0$};
\draw[black] (4.5-3+3,-0.6) node  {$1$};
\draw[black] (5.5-3+3,-0.6) node  {$4$};
\draw[black] (3.5+3,-0.6) node  {$0$};
\draw[black] (4.5+3,-0.6) node  {$3$};
\draw[black] (5.5+3,-0.6) node  {$4$};
 \draw[black] (6.5+3,-0.6) node  {$2$};
\draw[black] (8.5+2,-0.6) node  {$1$};
\draw[black] (7.5+3+1,-0.6) node  {$0$};
  \draw[black] (6.5+6,-0.6) node  {$4$};
\draw[black] (7.5+6,-0.6) node  {$2$};
\draw[black] (8.5+6,-0.6) node  {$3$};
\draw[black] (9.5+6,-0.6) node  {$2$};

\draw[black] (3.5-3+3,3.6) node  {$0$};
\draw[black] (4.5-3+3,3.6) node  {$1$};
\draw[black] (5.5-3+3,3.6) node  {$4$};
\draw[black] (3.5+3,3.6) node  {$0$};
\draw[black] (4.5+3,3.6) node  {$3$};
\draw[black] (5.5+3,3.6) node  {$4$};
 \draw[black] (6.5+3,3.6) node  {$2$};
\draw[black] (8.5+2,3.6) node  {$1$};
\draw[black] (7.5+3+1,3.6) node  {$0$};
  \draw[black] (6.5+6,3.6) node  {$4$};
\draw[black] (7.5+6,3.6) node  {$2$};
\draw[black] (8.5+6,3.6) node  {$3$};
\draw[black] (9.5+6,3.6) node  {$2$};

    \foreach \i in {3.5,4.5,...,11.5}
 {
\draw(\i,0)--++(90:1.5*2);
 }
 \draw(14.5,0)--++(90:1.5*2);
  \draw(15.5,0)--++(90:1.5*2);

\draw[black] (12.5+6,3.3)node {$\scalefont{0.8}  \stt_{\la \cup \alpha_3} $};

\draw(-2+14.5,3) to [out=-90,in=90] (-2+15.5,1.5)  
 to [out=-90,in=90] (-2+14.5,0) ; 
 
 \draw(-2+15.5,3) to [out=-90,in=90] (-2+14.5,1.5) 
 to [out=-90,in=90] (-2+15.5,0) ;

\draw[black] (12.5+6,1.5) node {$\scalefont{0.8}  \SSTT_ {\la \cup \alpha_4} $}; 

\draw[black] (12.5+6,-0.3) node {$\scalefont{0.8}  \stt_{\la \cup \alpha_3} $}; 
 
    \end{tikzpicture} 
    \end{minipage}
    \qquad\quad
        \begin{minipage}{7.5cm}\begin{tikzpicture}
  [xscale=0.45,yscale=  
  0.45]

 \draw (3,0) rectangle (16,3);
  
   \foreach \i in {3.5,4.5,...,15.5}
  {
   \fill[black](\i,0) circle(1.5pt); 
      \fill[black](\i,3)  circle(1.5pt);
       }

 \draw[black] (3.5-3+3,-0.6) node  {$0$};
\draw[black] (4.5-3+3,-0.6) node  {$1$};
\draw[black] (5.5-3+3,-0.6) node  {$4$};
\draw[black] (3.5+3,-0.6) node  {$0$};
\draw[black] (4.5+3,-0.6) node  {$3$};
\draw[black] (5.5+3,-0.6) node  {$4$};
 \draw[black] (6.5+3,-0.6) node  {$2$};
\draw[black] (8.5+2,-0.6) node  {$1$};
\draw[black] (7.5+3+1,-0.6) node  {$0$};
  \draw[black] (6.5+6,-0.6) node  {$2$};
\draw[black] (7.5+6,-0.6) node  {$4$};
\draw[black] (8.5+6,-0.6) node  {$3$};
\draw[black] (9.5+6,-0.6) node  {$2$};

\draw[black] (3.5-3+3,3.6) node  {$0$};
\draw[black] (4.5-3+3,3.6) node  {$1$};
\draw[black] (5.5-3+3,3.6) node  {$4$};
\draw[black] (3.5+3,3.6) node  {$0$};
\draw[black] (4.5+3,3.6) node  {$3$};
\draw[black] (5.5+3,3.6) node  {$4$};
 \draw[black] (6.5+3,3.6) node  {$2$};
\draw[black] (8.5+2,3.6) node  {$1$};
\draw[black] (7.5+3+1,3.6) node  {$0$};
  \draw[black] (6.5+6,3.6) node  {$2$};
\draw[black] (7.5+6,3.6) node  {$4$};
\draw[black] (8.5+6,3.6) node  {$3$};
\draw[black] (9.5+6,3.6) node  {$2$};

    \foreach \i in {3.5,4.5,...,10.5}
 {
\draw(\i,0)--++(90:1.5*2);
 }
 \draw(14.5,0)--++(90:1.5*2);
  \draw(15.5,0)--++(90:1.5*2);
 \draw(13.5,0)--++(90:1.5*2);

\draw[black] (12.5+6,3.3)node {$\scalefont{0.8} \SSTT_ {\la \cup \alpha_4} $};

\draw(-2+14.5-1,3) to [out=-90,in=90] (-2+15.5-1,1.5)  
 to [out=-90,in=90] (-2+14.5-1,0) ; 
 
 \draw(-2+15.5-1,3) to [out=-90,in=90] (-2+14.5-1,1.5) 
 to [out=-90,in=90] (-2+15.5-1,0) ;

\draw[black] (12.5+6,1.5) node {$\scalefont{0.8}  \SSTT_{\la \cup \alpha_5} $}; 

\draw[black] (12.5+6,-0.3) node {$\scalefont{0.8}   \SSTT_{\la \cup \alpha_4} $}; 
 
    \end{tikzpicture} 
    \end{minipage}
    $$  
    
    \!\!\!
    \caption{The     righthand-side of \cref{kkkkkkkk3} for $\la=(2^6,1^3)$ and  $\alpha=\alpha_3$ and $\alpha_4$ respectively.   
       }
\label{adjustmntexample44}    \end{figure}
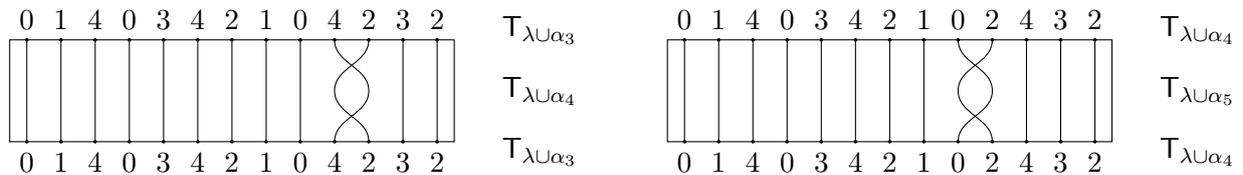

 \item[$(iv)$] Suppose $\alpha'=[x,y,z]$ has residue $r-1\in \ZZ/e\ZZ$ 
 (thus $[x,y,z]= [i,j-1,m]$ by residue considerations)  and 
that $ [i-1,j,m]\not \in \la$ (if $[i-1,j,m]\in \la$, this implies  that  $\la\cup \alpha \in \mathscr{P}_{\underline{\tau}}(n)$ and so the process would terminate).
We remark that ${\rm res}( {\sf Adj\text{-}Gar}(\alpha))=\{r-1\} $ and this is the unique case of the proof for which this holds.

 Let $\gamma = [i-1,j-1,m]$
and we set   $c=\SSTT_{\la\cup\alpha} (\gamma) $ and let 
$\xi =  
\Shape(\SSTT_\la{\downarrow}_{<\kay-1}) 
 $ (see \cref{reducedeg22244} for an example).   
By \cref{easy1}, $ e_{{\stt_  {\la\cup\alpha}}}=  \WHY^{\succ}_{{\stt_  {\la\cup\alpha}}}$. 
 We have that 
\begin{align}
 e_{{\stt_  {\la\cup\alpha}}}&= 
 e_{{\stt_  {\la  \cup\alpha}}}\psi^{c}  
 _{ {\kay-2}  }   \psi_{c}  
 ^{ {\kay-2} }  
  e_{{\stt_  {\la\cup\alpha}}} \\[2pt] &
 =
 e_{{\stt_  {\la  \cup\alpha}}}\psi_{\kay-1}\psi^{c}  
 _{ {\kay-1} }  
  \psi_{c}  
 ^{ \kay-2 }  \psi_{\kay-1} 
 e_{{\stt_  {\la\cup\alpha}}}
 -   e_{{\stt_  {\la  \cup\alpha}}}
    \psi^{c}  
 _{ {\kay-1} }  
  \psi_{c}  
 ^{ \kay  }   e_{{\stt_  {\la  \cup\alpha}}}
 \\[2pt]
\label{gfgfgfgfgfgf} &=
 e_{{\stt_  {\la\cup\alpha}}}   \psi^ {\kay}  
 _{ c+1 }  
  \psi_ {\kay}  
 ^{ c  }   e_{{\stt_  {\la  \cup\alpha}}}
 -
      \psi^{c}  
 _{ {\kay-1} }  
 ( e_{\stt_{\xi-\gamma }	}
\boxtimes 
e_{r-1,r,r}
\boxtimes
e_{\SSTT_{\la\cup\alpha}{\downarrow}_{>a}}
) 
 \psi_{c}  
 ^{ \kay }    
  \\[2pt]
\label{gfgfgfgfgfgf2} &=
-   e_{{\stt_  {\la\cup\alpha}}} \psi^ {\kay}  
 _{ c  }  
y_{c }
  \psi_ {\kay}  
 ^{ c  }   e_{{\stt_  {\la  \cup\alpha}}}
 -
      \psi^{c}  
 _{ {\kay-1} }  
 ( e_{\stt_{\xi-\gamma }	}
\boxtimes 
e_{r-1,r,r}
\boxtimes
e_{\SSTT_{\la\cup\alpha}{\downarrow}_{>a}}
) 
 \psi_{c}  
 ^{ \kay }    
  \\[2pt] \label{1.19ffddfdf}
  &\in 
  -   \psi^ {\kay}  
 _{ b  }  
 \WHY^{\succ}_{{\stt_  {\la  \cup\beta} }} 
   \psi_ {\kay}  
 ^{ b  }   e_{{\stt_  {\la  \cup\alpha}}}
 -
      \psi^{c}  
 _{ {\kay-1} }  
 ( e_{\stt_{\xi-\gamma }	}
\boxtimes 
e_{r-1,r,r}
\boxtimes
e_{\SSTT_{\la\cup\alpha}{\downarrow}_{>a}}
) 
 \psi_{c}  
 ^{ \kay }   +\mathscr{H}^{(\ppsucc\la) \cup \beta}
\end{align}
 where the first and third equalities follow  from the commuting case 2 of relation~\ref{rel4} and \cref{easy1}, and the second equality  follows from case~1 of relation~\ref{rel5}; the fourth equality follows from \ref{rel3}; the fifth is either trivial or follows from  case $(3)$ of  \ref{rel4} (in the latter case, the error term is zero by our  inductive assumption for rank $c-1<n$ of \cref{Claim 2}).  
For our continuing example, 
the righthand-side of \cref{gfgfgfgfgfgf2} is depicted in \cref{adjustmntexample787};
 the   box-configurations labelling the idempotents on the left and righthand-sides of \cref{gfgfgfgfgfgf} are depicted in \cref{reducedeg22244}.

 \begin{figure}[ht!!!] $$ \begin{minipage}{3.8cm}
\scalefont{0.9} \begin{tikzpicture}
  [xscale=0.5,yscale=  
  0.5]

\draw[thick,fill=white](-2,0.5) rectangle (5.5,-9.5);

\draw[thick ](5-3-1,-9+4) rectangle (6-4,-8+4);  
\draw(5.5-4,-8.5+4) node {$2$};

\draw[thick, fill=white] (0,0)--++(0:2)--++(-90:3)--++(180:1)--++(-90:6)--++(180:1)--++(90:9);

  \draw[thick, fill=gray!30] (0,0)--++(0:2)--++(-90:3)--++(180:1)--++(-90:1)--++(180:1)--++(90:4);

  \clip (0,0)--++(0:2)--++(-90:3)--++(180:1)--++(-90:6)--++(180:1)--++(90:9);
\draw(0.5,-0.5) node {$0$};
\draw(1.5,-0.5) node {$1$};
\draw(2.5,-0.6) node {$2$};
\draw(0.5,-1.5) node {$4$};
\draw(1.5,-1.5) node {$0$};
\draw(2.5,-1.5) node {$1$};
\draw(0.5,-2.5) node {$3$};
\draw(1.5,-2.5) node {$4$};
\draw(2.5,-2.5) node {$0$};

\draw(0.5,-3.5) node {$2$};
\draw(1.5,-3.5) node {$3$};

\draw(0.5,-4.5) node {$1$};
\draw(1.5,-4.5) node {$2$};

\draw(0.5,-5.5) node {$0$};
\draw(1.5,-5.5) node {$1$};

\draw(0.5,-6.5) node {$4$};\draw(0.5,-7.5) node {$3$};\draw(0.5,-8.5) node {$2$};
 
	\foreach \i in {0,1,2,...,20}
	{
		\path  (0,0)++(-90:1*\i cm)  coordinate (a\i);
		\path  (0.5,-0.6)++(0:1*\i cm)  coordinate (b\i);
		\path  (0.5,-1.5)++(0:1*\i cm)  coordinate (c\i);
		\path  (0.5,-2.5)++(0:1*\i cm)  coordinate (d\i);
		\draw[thick] (a\i)--++(0:6);
		\draw[thick] (1,0)--++(-90:6);
		\draw[thick] (2,0)--++(-90:5);
	}

  \end{tikzpicture} 
    \end{minipage}   
\;    \leftrightarrow \;\;
       \begin{minipage}{5.3cm}
\scalefont{0.9} \begin{tikzpicture}
  [xscale=0.5,yscale=  
  0.5]

\draw[thick,fill=white](-2,0.5) rectangle (5.5+3,-9.5); 

%

\draw[thick ](5-3-1+2+2,-9+4+1) rectangle (6-4+2+2,-8+4+1);  
\draw(5.5-4+2+2,-8.5+4+1 ) node {$2$};


\draw[thick, fill=white] (0,0)--++(0:2)--++(-90:3)--++(180:1)--++(-90:6)--++(180:1)--++(90:9);
  \clip (0,0)--++(0:2)--++(-90:3)--++(180:1)--++(-90:6)--++(180:1)--++(90:9);
\draw(0.5,-0.5) node {$0$};
\draw(1.5,-0.5) node {$1$};
\draw(2.5,-0.6) node {$2$};
\draw(0.5,-1.5) node {$4$};
\draw(1.5,-1.5) node {$0$};
\draw(2.5,-1.5) node {$1$};
\draw(0.5,-2.5) node {$3$};
\draw(1.5,-2.5) node {$4$};
\draw(2.5,-2.5) node {$0$};

\draw(0.5,-3.5) node {$2$};
\draw(1.5,-3.5) node {$3$};

\draw(0.5,-4.5) node {$1$};
\draw(1.5,-4.5) node {$2$};

\draw(0.5,-5.5) node {$0$};
\draw(1.5,-5.5) node {$1$};

\draw(0.5,-6.5) node {$4$};\draw(0.5,-7.5) node {$3$};\draw(0.5,-8.5) node {$2$};
 
	\foreach \i in {0,1,2,...,20}
	{
		\path  (0,0)++(-90:1*\i cm)  coordinate (a\i);
		\path  (0.5,-0.6)++(0:1*\i cm)  coordinate (b\i);
		\path  (0.5,-1.5)++(0:1*\i cm)  coordinate (c\i);
		\path  (0.5,-2.5)++(0:1*\i cm)  coordinate (d\i);
		\draw[thick] (a\i)--++(0:6);
		\draw[thick] (1,0)--++(-90:6);
		\draw[thick] (2,0)--++(-90:5);
	}

  \end{tikzpicture} 
    \end{minipage}  
    \;\;\;\;\;\;\;\; -  
     \begin{minipage}{5cm}
\scalefont{0.9} \begin{tikzpicture}
  [xscale=0.5,yscale=  
  0.5]

\draw[thick,fill=white](-2,0.5) rectangle (5.5+3,-9.5);

\draw[thick ](5-3-1,-9+4) rectangle (6-4,-8+4);  
\draw(5.5-4,-8.5+4) node {$2$};

\draw[thick ](5-3-1+2+3,-9+4) rectangle (6-4+2+3,-8+4);  
\draw(5.5-4+2+3,-8.5+4) node {$2$};

\draw[thick, fill=white] (0,0)--++(0:2)--++(-90:3)--++(180:1)--++(-90:6)--++(180:1)--++(90:9);

\draw[white,fill=white,thick](-0.1,-3) rectangle (1.01,-4);  
\draw[thick](0,-3) -- (1,-3);  
\draw[thick](0,-4) -- (1,-4);  
\draw[thick](0,-2)--++(-90:1) -- (1.1,-3);  
\draw[thick](0,-5)--++(90:1) -- (1.1,-4);

  \clip (0,0)--++(0:2)--++(-90:3)--++(180:1)--++(-90:6)--++(180:1)--++(90:9);

\draw(0.5,-0.5) node {$0$};
\draw(1.5,-0.5) node {$1$};
\draw(2.5,-0.6) node {$2$};
\draw(0.5,-1.5) node {$4$};
\draw(1.5,-1.5) node {$0$};
\draw(2.5,-1.5) node {$1$};
\draw(0.5,-2.5) node {$3$};
\draw(1.5,-2.5) node {$4$};
\draw(2.5,-2.5) node {$0$};

 \draw(1.5,-3.5) node {$3$};

\draw(0.5,-4.5) node {$1$};
\draw(1.5,-4.5) node {$2$};

\draw(0.5,-5.5) node {$0$};
\draw(1.5,-5.5) node {$1$};

\draw(0.5,-6.5) node {$4$};
\draw(0.5,-7.5) node {$3$};
\draw(0.5,-8.5) node {$2$};
 
	\foreach \i in {0,1,2,...,20}
	{
		\path  (0,0)++(-90:1*\i cm)  coordinate (a\i);
		\path  (0.5,-0.6)++(0:1*\i cm)  coordinate (b\i);
		\path  (0.5,-1.5)++(0:1*\i cm)  coordinate (c\i);
		\path  (0.5,-2.5)++(0:1*\i cm)  coordinate (d\i);
		\draw[thick] (a\i)--++(0:6);
		\draw[thick] (1,0)--++(-90:6);
		\draw[thick] (2,0)--++(-90:5);
	}

\draw[white,fill=white ](0,-3) rectangle (1,-4);  
\draw[thick](0,-2)--++(-90:1) -- (1.1,-3);  
\draw[thick](0,-5)--++(90:1) -- (1.1,-4);  

  \end{tikzpicture} 
    \end{minipage} 
  $$   \caption{
 	The lefthand-side is  $\la\cup\alpha_5$ as in   \cref{reducedeg2223} (with $\xi$ shaded grey).  The righthand-side labels the idempotents obtained from applying \cref{gfgfgfgfgfgf}.
 The middle box configuration 
 is $\SSTT_{\la\cup\alpha_6}$.   
 	}
 \label{reducedeg22244}
 \end{figure}
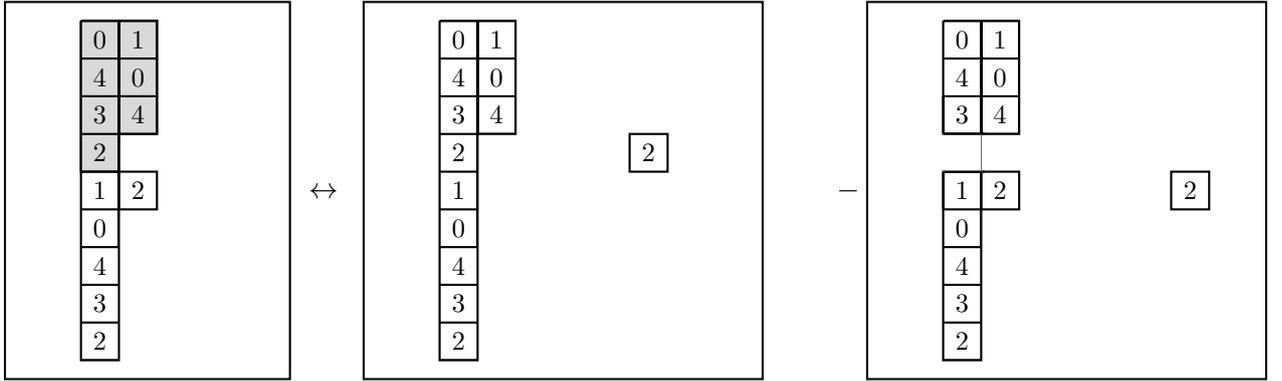

\begin{figure}[ht!]
$$
  \begin{minipage}{5.3cm}
\scalefont{0.9} \begin{tikzpicture}
  [xscale=0.5,yscale=  
  0.5]

\draw[thick,fill=white](-2,0.5) rectangle (5.5+3,-9.5);

\draw[thick ](5-3-1,-9+4) rectangle (6-4,-8+4);  
\draw(5.5-4,-8.5+4) node {$2$};

\draw[thick ](5-3-1+2+3,-9+4) rectangle (6-4+2+3,-8+4);  
\draw(5.5-4+2+3,-8.5+4) node {$2$};

\draw[thick, fill=white] (0,0)--++(0:2)--++(-90:3)--++(180:1)--++(-90:6)--++(180:1)--++(90:9);

\draw[white,fill=white,thick](0,-3) rectangle (1,-4);  
\draw[thick](0,-3) -- (1,-3);  
\draw[thick](0,-4) -- (1,-4);  
\draw[thick](0,-2)--++(-90:1) -- (1.1,-3);  
\draw[thick](0,-5)--++(90:1) -- (1.1,-4);

  \clip (0,0)--++(0:2)--++(-90:3)--++(180:1)--++(-90:6)--++(180:1)--++(90:9);

\draw(0.5,-0.5) node {$0$};
\draw(1.5,-0.5) node {$1$};
\draw(2.5,-0.6) node {$2$};
\draw(0.5,-1.5) node {$4$};
\draw(1.5,-1.5) node {$0$};
\draw(2.5,-1.5) node {$1$};
\draw(0.5,-2.5) node {$3$};
\draw(1.5,-2.5) node {$4$};
\draw(2.5,-2.5) node {$0$};

 \draw(1.5,-3.5) node {$3$};

\draw(0.5,-4.5) node {$1$};
\draw(1.5,-4.5) node {$2$};

\draw(0.5,-5.5) node {$0$};
\draw(1.5,-5.5) node {$1$};

\draw(0.5,-6.5) node {$4$};
\draw(0.5,-7.5) node {$3$};
\draw(0.5,-8.5) node {$2$};
 
	\foreach \i in {0,1,2,...,20}
	{
		\path  (0,0)++(-90:1*\i cm)  coordinate (a\i);
		\path  (0.5,-0.6)++(0:1*\i cm)  coordinate (b\i);
		\path  (0.5,-1.5)++(0:1*\i cm)  coordinate (c\i);
		\path  (0.5,-2.5)++(0:1*\i cm)  coordinate (d\i);
		\draw[thick] (a\i)--++(0:6);
		\draw[thick] (1,0)--++(-90:6);
		\draw[thick] (2,0)--++(-90:5);
	}

\draw[white,fill=white ](0,-3) rectangle (1,-4);  
\draw[thick](0,-2)--++(-90:1) -- (1.1,-3);  
\draw[thick](0,-5)--++(90:1) -- (1.1,-4);  

  \end{tikzpicture} 
    \end{minipage} 
 \;\leftrightarrow \;\;
 \begin{minipage}{5.3cm}
\scalefont{0.9} \begin{tikzpicture}
  [xscale=0.5,yscale=  
  0.5]

 \draw[thick](0,-3)--++(0:1)  --++(-90:1);  

\draw[thick,fill=white](-2,0.5) rectangle (5.5+3,-9.5); 
  
\draw(0.5,-3.5) node {$2$};


\draw[thick ](5-3-1+2+2,-9+4+1) rectangle (6-4+2+2,-8+4+1);  
\draw(5.5-4+2+2,-8.5+4+1 ) node {$2$};

\draw[thick ]  (2,-1) rectangle (3,-2);   
\draw(2.5,-1.5) node {$1$};

\draw(0.5,-3.5) node {$2$};

\draw[thick, fill=white] (0,0)--++(0:2)--++(-90:3)--++(180:1)--++(-90:6)--++(180:1)--++(90:9);

 \draw(0.5,-3.5) node {$2$};
\draw[white,fill=white,thick](0,-4) rectangle (1,-5);  
\draw[thick](0,-3-1) -- (1,-3-1);  
\draw[thick](0,-4-1) -- (1,-4-1);  
\draw[thick](0,-2-1)--++(-90:1) -- (1,-3-1);  
\draw[thick](0,-5-1)--++(90:1) -- (1,-4-1);  

 \draw[thick](0,-3)--++(0:1)  --++(-90:1);  

\draw(0.5,-3.5) node {$2$};
  \clip (0,0)--++(0:2)--++(-90:3)--++(180:1)--++(-90:6)--++(180:1)--++(90:9);

\draw(0.5,-0.5) node {$0$};
\draw(1.5,-0.5) node {$1$};
\draw(2.5,-0.6) node {$2$};
\draw(0.5,-1.5) node {$4$};
\draw(1.5,-1.5) node {$0$};
\draw(2.5,-1.5) node {$1$};
\draw(0.5,-2.5) node {$3$};
\draw(1.5,-2.5) node {$4$};
\draw(2.5,-2.5) node {$0$};

 \draw(1.5,-3.5) node {$3$};

\draw(0.5,-3.5) node {$2$};
\draw(1.5,-4.5) node {$2$};

\draw(0.5,-5.5) node {$0$};
\draw(1.5,-5.5) node {$1$};

\draw(0.5,-6.5) node {$4$};
\draw(0.5,-7.5) node {$3$};
\draw(0.5,-8.5) node {$2$};
 
	\foreach \i in {0,1,2,...,20}
	{
		\path  (0,0)++(-90:1*\i cm)  coordinate (a\i);
		\path  (0.5,-0.6)++(0:1*\i cm)  coordinate (b\i);
		\path  (0.5,-1.5)++(0:1*\i cm)  coordinate (c\i);
		\path  (0.5,-2.5)++(0:1*\i cm)  coordinate (d\i);
		\draw[thick] (a\i)--++(0:6);
		\draw[thick] (1,0)--++(-90:6);
		\draw[thick] (2,0)--++(-90:5);
	}

\draw[white,fill=white ](0,-3) rectangle (1.1,-4);
\draw[white,fill=white ](0,-4) rectangle (1.1,-5);  
\draw[thick](0,-3)--++(-90:1) -- (1,-4);  
\draw[thick](0,-6)--++(90:1) -- (1,-5);  
 \draw[thick](0,-3)--++(0:1)  --++(-90:1);  
 \draw(0.5,-3.5) node {$2$};
  \end{tikzpicture} 
    \end{minipage} 
    $$
    \caption{Rewriting the first term after the equality in \cref{gfgfgfgfgfgf}. We have moved the 1-box using case $(iii)$ and this leaves us free to move the 2-boxes up their corresponding diagonals.  
    The righthand-side is a box configuration which is strictly higher than $\la \cup \beta$ in the $\succ$-ordering.  
     }
    \label{khjlgfshjkldsfghjfgkdghdjlsfghdfjkls}
    \end{figure}
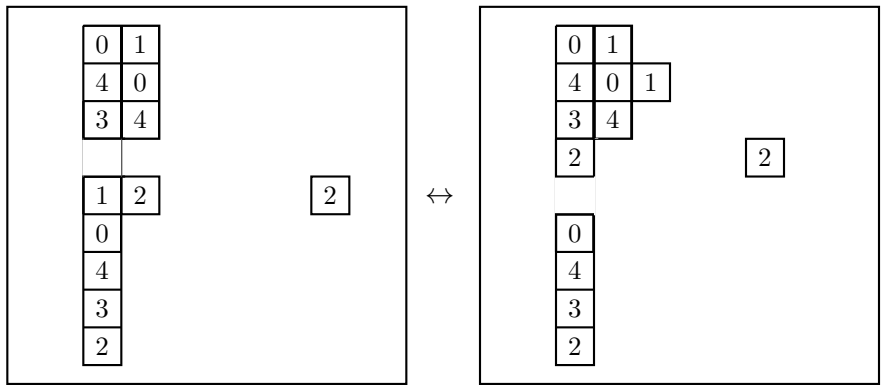

  We now consider the second term on the righthand-side of  \cref{gfgfgfgfgfgf}. 
  We have that 
 $$
 e_{\stt_{\xi-\gamma}	}
\boxtimes 
e_{r-1,r,r}
=e_{\SSTT_{ \xi-\gamma \cup[i,j-1,m] \cup\alpha \cup [i,j+e,m]}}$$
By our  inductive assumption for ranks $a-2,a-1<n$ for \cref{Claim 3}, we have that:  
$$  e_{\SSTT_{ \xi-\gamma\cup[i,j-1,m]}}	\in 	 \mathscr{H}_{a-2}^{\ppsucceq \rho			} \; 
\Longrightarrow \; 
 e_{\SSTT_{ \xi-\gamma\cup[i,j-1,m] \cup\alpha}}   \in \mathscr{H}_{a-1}^{(\ppsucceq \rho)	\cup \gamma	 }  
 \leq  \mathscr{H}_{a-1}^{ \ppsucceq (\rho 	\cup \gamma)	 }  $$
for $\rho= Y^{\ell+1}_{[i,j-1,m]}(\xi-\gamma\cup [i,j-1,m]) $.  
 Given $\pi\psucceq (\rho \cup \gamma)$, we can left justify $\pi\cup [i,j+e,m]$ to obtain  $\pi\cup\alpha$.  We note that 
 ${\sf Gar}_{\succ}(\alpha)\cap \pi$ contains no nodes of residue $r$ or $r\pm1$.   
 Therefore 
 $$
e_{\SSTT_{ \xi-\gamma\cup[i,j-1,m] \cup\alpha \cup[i,j+e,m]}}
 \in \mathscr{H}^{(\succeq \rho 	\cup \gamma)	\cup\beta	}_a	
 \; 
\Longrightarrow \; 
  e_{\stt_{\xi}	}
\boxtimes 
e_{r-1,r,r}
\boxtimes
e_{\SSTT_{\la\cup\alpha}{\downarrow}_{>a}}
 \in \mathscr{H}^{(\ppsucc  \la) 	\cup \beta	} _n
  $$
 as required.     (Note that $\gamma,\beta\not\in \pi 
  $ by   \cref{easy1}.)    
 See \cref{khjlgfshjkldsfghjfgkdghdjlsfghdfjkls} for an example.  
\end{itemize}

 \begin{figure}[ht!] 
  $$ 
 \begin{minipage}{7.5cm}\begin{tikzpicture}
  [xscale=0.5,yscale=  
  0.5]

 \draw (3,0) rectangle (16,3);
  
   \foreach \i in {3.5,4.5,...,15.5}
  {
   \fill[black](\i,0) circle(1.5pt); 
      \fill[black](\i,3)  circle(1.5pt);
       }

\draw[black] (3.5-3+3,3.5) node  {$0$};
\draw[black] (4.5-3+3,3.5) node  {$1$};
\draw[black] (5.5-3+3,3.5) node  {$4$};
\draw[black] (3.5+3,3.5) node  {$0$};
\draw[black] (4.5+3,3.5) node  {$3$};
\draw[black] (5.5+3,3.5) node  {$4$};
 \draw[black] (6.5+3,3.5) node  {$2$};
\draw[black] (8.5+2,3.5) node  {$1$};
\draw[black] (7.5+3+1,3.5) node  {$2$};
  \draw[black] (6.5+6,3.5) node  {$0$};
\draw[black] (7.5+6,3.5) node  {$4$};
\draw[black] (8.5+6,3.5) node  {$3$};
\draw[black] (9.5+6,3.5) node  {$2$}; 
 
 \draw[black] (3.5-3+3,-0.5) node  {$0$};
\draw[black] (4.5-3+3,-0.5) node  {$1$};
\draw[black] (5.5-3+3,-0.5) node  {$4$};
\draw[black] (3.5+3,-0.5) node  {$0$};
\draw[black] (4.5+3,-0.5) node  {$3$};
\draw[black] (5.5+3,-0.5) node  {$4$};
 \draw[black] (6.5+3,-0.5) node  {$2$};
\draw[black] (8.5+2,-0.5) node  {$1$};
\draw[black] (7.5+3+1,-0.5) node  {$2$};
  \draw[black] (6.5+6,-0.5) node  {$0$};
\draw[black] (7.5+6,-0.5) node  {$4$};
\draw[black] (8.5+6,-0.5) node  {$3$};
\draw[black] (9.5+6,-0.5) node  {$2$};

\draw(-1+1+14.5-4,3) to [out=-90,in=90]  
(2+-2+1+14.5-4,1.5) to [out=-90,in=90]  
(-1+1+14.5-4,0) ; 
 
\draw(-1+14.5-4+2,3) to [out=-90,in=90]  (-3+14.5-2,0) ; 
 
  \draw(2+-1+14.5-4,0) to [out=90,in=-90] (-2-4+15.5,1.5)  
  to [out=90,in=-90] (-3-2+14.5,3) ;

  
    \foreach \i in {3.5,4.5,...,8.5}
 {
\draw(\i,0)--++(90:1.5*2);
 }
 
 \foreach \i in {12.5,13.5,...,15.5}
 {
\draw(\i,0)--++(90:1.5*2);
 }

   \draw[black] (0.5+0.2,3.2) node {$\scalefont{0.8}  \stt_{\la \cup \alpha_5}$};

\draw[black] (0.5+0.2,-0.2) node {$\scalefont{0.8}  \stt_{\la \cup \alpha_5}$};

    \end{tikzpicture} 
    \end{minipage}
    \;\;\;\;\;\;\;\;\;    -\; 
      \begin{minipage}{7.7cm}\begin{tikzpicture}
  [xscale=0.5,yscale=  
  0.5]
\draw (3,0) rectangle (16,3);
  
   \foreach \i in {3.5,4.5,...,15.5}
  {
   \fill[black](\i,0) circle(1.5pt); 
      \fill[black](\i,3)  circle(1.5pt);
       }

 \draw[black] (3.5-3+3,-0.5) node  {$0$};
\draw[black] (4.5-3+3,-0.5) node  {$1$};
\draw[black] (5.5-3+3,-0.5) node  {$4$};
\draw[black] (3.5+3,-0.5) node  {$0$};
\draw[black] (4.5+3,-0.5) node  {$3$};
\draw[black] (5.5+3,-0.5) node  {$4$};
 \draw[black] (6.5+3,-0.5) node  {$2$};
\draw[black] (8.5+2,-0.5) node  {$1$};
\draw[black] (7.5+3+1,-0.5) node  {$2$};
  \draw[black] (6.5+6,-0.5) node  {$0$};
\draw[black] (7.5+6,-0.5) node  {$4$};
\draw[black] (8.5+6,-0.5) node  {$3$};
\draw[black] (9.5+6,-0.5) node  {$2$};

\draw(-1+1+14.5-4,3) to [out=-90,in=90]  
(-2+1+14.5-4,1.5) to [out=-90,in=90]  
(-1+1+14.5-4,0) ; 
 
\draw(-1+14.5-4,3) to [out=-90,in=90]  (-3+14.5,0) ; 
 
 \draw(-1+14.5-4,0) to [out=90,in=-90] (-4+15.5,1.5)  
 to [out=90,in=-90] (-3+14.5,3) ;

  
    \foreach \i in {3.5,4.5,...,8.5}
 {
\draw(\i,0)--++(90:1.5*2);
 }
 
 \foreach \i in {12.5,13.5,...,15.5}
 {
\draw(\i,0)--++(90:1.5*2);
 }

\draw[black] (3.5-3+3,3.5) node  {$0$};
\draw[black] (4.5-3+3,3.5) node  {$1$};
\draw[black] (5.5-3+3,3.5) node  {$4$};
\draw[black] (3.5+3,3.5) node  {$0$};
\draw[black] (4.5+3,3.5) node  {$3$};
\draw[black] (5.5+3,3.5) node  {$4$};
 \draw[black] (6.5+3,3.5) node  {$2$};
\draw[black] (8.5+2,3.5) node  {$1$};
\draw[black] (7.5+3+1,3.5) node  {$2$};
  \draw[black] (6.5+6,3.5) node  {$0$};
\draw[black] (7.5+6,3.5) node  {$4$};
\draw[black] (8.5+6,3.5) node  {$3$};
\draw[black] (9.5+6,3.5) node  {$2$};

%
%
%

\draw[black] (12.5+6-0.2,3.2)node {$\scalefont{0.8}  \stt_{\la \cup \alpha_5}$}; 
\draw[black] (12.5+6-0.2,-0.2) node {$\scalefont{0.8}  \stt_{\la \cup \alpha_5}$}; 

    \end{tikzpicture} 
    \end{minipage} 
  \quad\quad   $$  
    
     \caption{   The  righthand-side of \cref{gfgfgfgfgfgf} for $\la=(2^6,1^3)$ and $\alpha_5$ as in \cref{reducedeg222}.   
     The left diagram factors through the element $y^\succ_{\SSTT_{\la\cup\alpha_6}}$; we will further manipulate this in \cref{adjustmntexample78722} below.   }
\label{adjustmntexample787}    \end{figure}
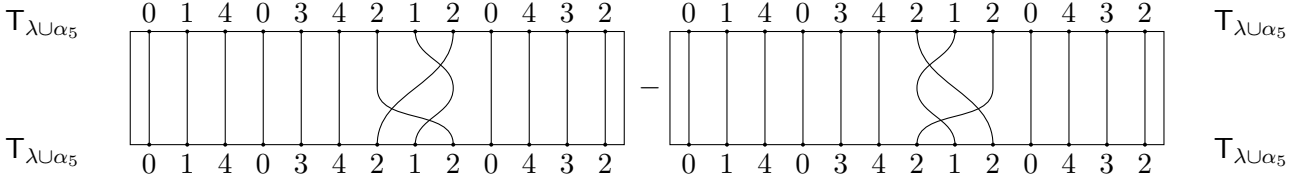

 \begin{figure}[ht!] 
  $$ 
 \begin{minipage}{7.5cm}\begin{tikzpicture}
  [xscale=0.5,yscale=  
  0.5]

 \draw (3,0) rectangle (16,3);
  
   \foreach \i in {3.5,4.5,...,15.5}
  {
   \fill[black](\i,0) circle(1.5pt); 
      \fill[black](\i,3)  circle(1.5pt);
       }

\draw[black] (3.5-3+3,3.5) node  {$0$};
\draw[black] (4.5-3+3,3.5) node  {$1$};
\draw[black] (5.5-3+3,3.5) node  {$4$};
\draw[black] (3.5+3,3.5) node  {$0$};
\draw[black] (4.5+3,3.5) node  {$3$};
\draw[black] (5.5+3,3.5) node  {$4$};
 \draw[black] (6.5+3,3.5) node  {$2$};
\draw[black] (8.5+2,3.5) node  {$1$};
\draw[black] (7.5+3+1,3.5) node  {$2$};
  \draw[black] (6.5+6,3.5) node  {$0$};
\draw[black] (7.5+6,3.5) node  {$4$};
\draw[black] (8.5+6,3.5) node  {$3$};
\draw[black] (9.5+6,3.5) node  {$2$}; 
 
 \draw[black] (3.5-3+3,-0.5) node  {$0$};
\draw[black] (4.5-3+3,-0.5) node  {$1$};
\draw[black] (5.5-3+3,-0.5) node  {$4$};
\draw[black] (3.5+3,-0.5) node  {$0$};
\draw[black] (4.5+3,-0.5) node  {$3$};
\draw[black] (5.5+3,-0.5) node  {$4$};
 \draw[black] (6.5+3,-0.5) node  {$2$};
\draw[black] (8.5+2,-0.5) node  {$1$};
\draw[black] (7.5+3+1,-0.5) node  {$2$};
  \draw[black] (6.5+6,-0.5) node  {$0$};
\draw[black] (7.5+6,-0.5) node  {$4$};
\draw[black] (8.5+6,-0.5) node  {$3$};
\draw[black] (9.5+6,-0.5) node  {$2$};

\draw(-1+1+14.5-4,3) to [out=-90,in=90]  
(2+-2+1+14.5-4,1.5) to [out=-90,in=90]  
(-1+1+14.5-4,0) ; 

\draw(-1+1+14.5-4-1,3) to [out=-90,in=90]  
(2+-2+1+14.5-4-1,1.5) to [out=-90,in=90]  
(-1+1+14.5-4-1,0) ;

\draw(-1+1+14.5-4+1,3) to [out=-90,in=90]  
(2+-2+1+14.5-4-2,1.5) to [out=-90,in=90]  
(-1+1+14.5-4+1,0) ; 

\draw[fill](2+-2+1+14.5-4-2,1.5)  circle (4pt); 
%
 

  
    \foreach \i in {3.5,4.5,...,8.5}
 {
\draw(\i,0)--++(90:1.5*2);
 }
 
 \foreach \i in {12.5,13.5,...,15.5}
 {
\draw(\i,0)--++(90:1.5*2);
 }

   \draw[black] (12.5+6-0.2,3.2) node {$\scalefont{0.8}  \stt_{\la \cup \alpha_5}$};

      \draw[black] (12.5+6-0.2,1.5) node {$\scalefont{0.8}  \stt_{\la \cup \alpha_7}$}; 

      \draw[black] (0.5+1.2,1.5) node {$-$};

\draw[black] (12.5+6-0.2,-0.2) node {$\scalefont{0.8}  \stt_{\la \cup \alpha_5}$};

    \end{tikzpicture} 
    \end{minipage}   $$  
    
     \caption{   The leftmost term in \cref{adjustmntexample787} rewritten in the form in \cref{1.19ffddfdf}.    This diagram factors through 
     $y^\succ_{\SSTT_{\la\cup\alpha_7}}$.   
In this case $\alpha=\alpha_5$ and $\beta=\alpha_7$  as in the first case of  \cref{L+!!!!!}.
      }
\label{adjustmntexample78722}    \end{figure}
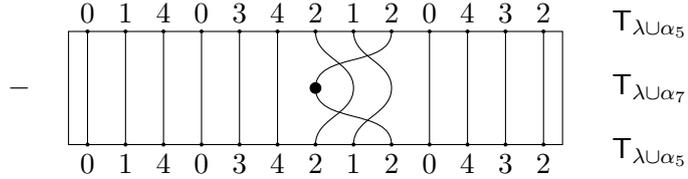

\smallskip

 \noindent {\bf Proof of  \cref{Claim 2} for a given $\la$ and $\alpha$}.   
We assume that \cref{Claim 2} holds for all $\la\in \mathscr{P}_{\underline{\tau}}(n-1)$.   
We set $\nu=\la\cup\alpha \in \mathscr{P}_{\underline{\tau}}(n)$ and we assume that  $\alpha$   is of residue $r\in \ZZ/e\ZZ$. 
We have that $e_{\SSTT_{\la\cup\alpha}}=\WHY^{\succ}_{\SSTT_{\la\cup\alpha}}$.  
   We have that 
\begin{align}\label{gsjfkhkdfhgdkfghd}
y_{a} \WHY_{\stt_{\la\cup\alpha}}
 &=
 \begin{cases}
  y_{a-1}\WHY_{\stt_{\la\cup\alpha}}
  + 
     \psi_{b}^a  \WHY_{\stt_{\la\cup\beta}}  \psi_{a}^b 
 &\text{ if $r-1\in {\rm res}( {\sf Adj\text{-}Gar}_{\succ}(\alpha))$;} \\
  y _{\stt_{\la\cup\beta}} 
&\text{ if  $r-1\not\in{\rm res}( {\sf Adj\text{-}Gar}_{\succ}(\alpha)) $;}				
\end{cases} 
\end{align}
 In the first case, this follows from case 3 of relation~\ref{rel4} and the commutativity relations, to see this note that   the $ (a-1)$th  
 strand has residue $r-1\in \ZZ/e\ZZ$.  
     In the second case,  the statement follows from \cref{easy1}   and the fact that  $b=a$.  
 Letting $\beta'$ be as in \cref{popopopopopopopo},  we note that 
 $y_{a-1}\WHY_{\stt_{\la }}\in  \mathscr{H}_{n-1}^{ \ppsucceq \la\cup\beta' -\alpha'   }
$
and so the first case of \cref{gsjfkhkdfhgdkfghd} is of the required form   by our assumption for $\la\cup\beta' -\alpha' =\mu \psucc \la$.

\bigskip

 \noindent {\bf Proof of  \cref{Claim 1}}.   
Let $\underline{j}= (j_1,\dots, j_{n-1},r)\in (\ZZ/e\ZZ)^{n }$.  
We can assume that  $\Shape({\sf J}_{\leq n-1}) = \la\in \mathscr{P}_\aatch(n-1) $ as otherwise 
$e_{\underline{j}} \boxtimes e_{j_n} = 0 \boxtimes e_{r} = 0$ by induction.  
Thus it remains to show that 
$$\tilde{\psi}_{\stt_{\la }}^{{\sf J}_{<n}} \tilde{\psi}^{\stt_{\la }}_{{\sf J}_{<n}} \boxtimes e_{r}
\in 
\pm \tilde{\psi}_{\stt_{\nu }}^{\sf J} \tilde{\psi}^{\stt_{\nu }}_{\sf J} + 
 \mathscr{H}^{ \ppsucc\nu  }_n
. $$
 We can associate this rightmost $r$-strand to the (unique) left-justified $r$-box $\alpha$ such that $\square \succ \alpha$    for all $\square \in \la$. (For example, $\square=[9,6,0]$ for 
 $\la =(2^3,1^6)$, see \cref{reducedeg2223}.)   
Thus every strand  in the diagram is labelled by a box.  
We  pull the strand labelled by $\alpha$ through the centre of the diagram
(which is equal to the idempotent  $e_{\SSTT_\nu}$)
one row at a time using \cref{Claim 3}.  
We can utilise \cref{Claim 3}  precisely when $\la \cup \alpha \not \in \mathscr{P}_\aatch(n)$.   
Therefore this process terminates when we reach the smallest addable $r$-box of $\la$ under the $\succ$-ordering,  namely 
 $ {\sf J}^{-1}(n )$.  Thus \cref{Claim 1} follows.  
       \end{proof}

%
%
%
%


 \begin{prop}[{\cite[Lemma 2.4 and Proposition 2.5]{bkw11}}]\label{move a dot down} 
We let $\w$, $\w'$ be any two choices of reduced expression for $w\in \mathfrak{S}_n $ and let $\vvv$ be any non-reduced expression for $w $. 
 We have that  
\begin{align}  \label{BK1}
 e_{\underline{i}}\psi_\w e_{\underline{j}}
    &= \textstyle  e_{\underline{i}}\psi_{\w'} e_{\underline{j}}+ \sum_{
 \begin{subarray}c
 \underline{x} < \w,\w'
  \end{subarray}
 }  e_{\underline{i}} \psi_{ \underline{x}}  f_{\x}(y)		e_{\underline{j}}   \\[3pt]  \label{BK2}
 e_{\underline{i}}\psi_\vvv e_{\underline{j}}  &=   \textstyle \sum_{
 \begin{subarray}c
\x < \vvv
 \end{subarray}
 }  e_{\underline{i}}   \psi_{\x} e_{\underline{j}} g_{\x} (y) 
 \\[3pt]  \label{BK3}
  y_{k}
e_{\underline{i}} \psi_\w e_{\underline{j}}  &= \textstyle  e_{\underline{j}} \psi_\w e_{\underline{i}} y_{w(k)} 
  +  \sum_{
 \begin{subarray}c
   \x < \w 
 \end{subarray}
 }  e_{\underline{i}} \psi_{\x} e_{\underline{j}}
 \end{align}
for some $ f_{\x}(y), g_{\x}(y)\in \mathcal{Y}_n$.  
   \end{prop} 

  \begin{prop} \label{spanner}
Let   $\Bbbk$  be an integral domain.      The $\Bbbk$-algebra  $   \mathscr{H}^\sigma_n$ has spanning set 
 \begin{align} 
  \{ \psi ^\SSTS_{\SSTT_\la}
  \psi _\SSTT^{\SSTT_\la}
  \mid
   \sts,\stt \in \Std (\lambda ), 
 \lambda  \in \mathscr{P}_ {\aatchpair}(n) \} .  
\end{align}
 \end{prop}

\begin{proof} 
 Let $d \in e_{\underline{i}}
\mathscr{H}^\sigma_n  e_{\underline{j}}$ for some $\imath,\jmath \in(\ZZ/e\ZZ)^n$. By 
  \cref{Claim 1}, we can 
 rewrite $e_{\underline{j}}$ (or equivalently $e_{\underline{i}}$) so that 
 $d=\sum_{x,y\in\mathfrak{S}_n} e_{\underline{i}} a_x e_{\stt_\la}a  _y e_{\underline{j}}$ for some  
 $a_x, a_y $ which are linear combinations  of KLR elements   tracing out some bijections 
  $x,y\in \mathfrak{S}_n$ respectively 
   (but are possibly decorated with dots and need not be reduced).  
  It remains to show that 
  $a_x, a_y\in \mathscr{H}_n^\sigma$ can be assumed to be reduced and undecorated.  
 We   establish  this  by induction  along the Bruhat order,   
by working modulo the span of elements   
\begin{equation}\label{workmod}
{\rm Span}_\Bbbk
 \{ \psi _{\underline{u}} e_{\stt_\la} \psi _{\underline{v}} 
  \mid u < x \text{  or  }
v  <  y
  \} +     \mathscr{H}^{\ppsucc \la}_{n}   .       
\end{equation} 
If $\x$ is not reduced
    $\psi_\x  e_{\stt_\la} a_y$ is zero modulo 
  (\ref{workmod}) 
 by  \cref{BK2}.  
Given two choices $\x,\x' $ of reduced expression for $x \in \mathfrak{S}_n$ we have that 
$ ( \psi_\x- \psi_{\x'}) e_{\SSTT_\la}a_\y$ belongs to  
 (\ref{workmod}) by \cref{BK1} followed by \cref{Claim 2}.
 Finally, if $a_x$ is obtained from 
  $\psi_\x$ by adding a linear combination  of dot decorations (at any  points within the expression $\psi_{\x}=\psi_{s_{i_1}}\dots \psi_{s_{i_k}}$) then 
    $\psi_\x  e_{\stt_\la} a_y$ is zero modulo 
  (\ref{workmod}) 
 by  \cref{BK3} followed by \cref{Claim 2}.  
    Thus $\mathscr{H}^{\ppsucceq \la}_n/\mathscr{H}^{\ppsucc  \la}_n$  is spanned by elements of the form 
  \begin{equation}\label{workmod2}  
 \{ \psi _{\underline{x}} e_{\stt_\la} \psi _{\underline{y}} 
  \mid  
\text{for  $\x,\y$   arbitrary choices of fixed reduced expressions of $x,y\in \mathfrak{S}_n$}
  \} +     \mathscr{H}^{\ppsucc \la}_{n}.    
\end{equation}
   It  remains to show that a spanning set is given by the  elements 
    $
    x= w^\sts_{\stt_\la} $, $ y= w  _\stt^{\stt_\la}		
    $ 
    for   $     \sts,\stt \in \Std (\lambda )  $.  
 
Given $\stt\in  \CStd  (\lambda )\setminus  \Std (\lambda )$, we have that $w_\stt^{\stt_\la}$ has a pair of crossing strands from 
$1\leq i<j\leq n$ to $1\leq w_\stt^{\stt_\la}(j)<w_\stt^{\stt_\la}(i)\leq n$ such that $\stt_\la^{-1}(i)=[r,c,m]$ and 
 $\stt_\la^{-1}(j)=[r,c+1,m]$  are in the same row and in particular so that $i=j-1$.  It suffices to show that $ \psi _{\underline{x}} e_{\stt_\la} \psi _{\underline{y}} $ belongs to the ideal $\mathscr{H}^{\ppsucc\la}_{n}$ for a preferred choice of $\y$; we choose $\y= s_{i} \w$ 
 (for some $w\in \mathfrak{S}_n$ such that 
  $s_i w  =y$).  Thus it remains to show that $e_{\stt_\la}   \psi_{s_{i} \w} 
   \in \mathscr{H}_{n}^{\ppsucc\la}$.  However, this   immediately follows from \cref{Claim 3}  
because   $e_{\stt_\la}\psi_{s_{i}  }
=\psi_{s_{i}  } e_{s_i(\stt_\la)}  $ and we have that 
 $e_{s_i(\stt_\la)}  
   \in    \mathscr{H}^{\ppsucceq \alpha} _n   $ for $\alpha= Y _{\SSTT_\la^{-1}(i+1)} (\la)\psucc \la $.    
 
Given any $\stt\in  \RStd  (\lambda )\setminus  \Std (\lambda )$ 
   we let $k$ be minimal such that 
   $\stt{\downarrow}_{<k}\in   \Std (\mu )$  for some $\mu\in \mathscr{P}_{\aatchpair}(k-1)$ 
   and 
      $\Shape(\stt{\downarrow}_{\leq k})=\nu$  for some $\nu\not \in \mathscr{P}_{\aatchpair}(k)$.     
   We
   have that 
$e_{\sf T}= e_{\SSTT_{\leq k}}\boxtimes e_{\SSTT_{> k}}$  
where $\Shape (e_{\SSTT_{\leq k}}) =\nu   \in \mathscr{C}_{\underline{\tau}}(k) \setminus 
 \mathscr{P}_{\underline{\tau}}(k)$   and so 
 $e_{\SSTT_{\leq k}} \in  \mathscr{H}_{k}^{\ppsucc\nu }$  
 by  
\cref{Claim 3} and so  $e_{\sf T}\in  \mathscr{H}_{n}^{\ppsucc \la}$ by concatenation and the definition of $\psucc$.  
 This implies that $e_{\SSTT_\la} \psi_\SSTT^{\SSTT_\la} =
  \psi_\SSTT^{\SSTT_\la} e_{\SSTT}  \in \mathscr{H}_{n}^{\ppsucc \la}$, as required.  
\end{proof}

 \begin{thm}  \label{cellularitybreedscontempt} 
Let   $\Bbbk$  be an integral domain.      The $\Bbbk$-algebra  $   \mathscr{H}^\sigma_n$ is       graded cellular   with basis     
 \begin{align}\label{required}
  \{ \psi ^\SSTS_{\SSTT_\la}
  \psi _\SSTT^{\SSTT_\la}
  \mid
   \sts,\stt \in \Std (\lambda ), 
 \lambda  \in \mathscr{P}_ {\aatchpair}(n) \} 
\end{align}
 anti-involution $\ast$ and the degree function ${\sf deg}:\Std \to\ZZ$.  
 For $\Bbbk$ a field,   $   \mathscr{H}^\sigma_n$ is quasi-hereditary.  
 \end{thm}

\begin{proof} 
We first prove that the spanning set of    \cref{spanner} is a $\Bbbk$-basis.  We will show that the  rank, as a $\Bbbk$-module,  of 
 $\mathcal{H}^\sigma_n {\sf y}_{\aatchpair}\mathcal{H}^\sigma_n $  is less than or equal to 
 $\sum_{\la\in \mathscr{P}(n) \setminus \mathscr{P}_{\aatchpair}(n)}|\Std(\la)|^2$.  
 The ($\psucc$)-ordering does not give us an easily constructible  basis of $   \mathcal{H}^\sigma_n$  (see \cite{manycell}).
 Hu--Mathas  \cite[Main Theorem]{hm10} have shown that the classical ordering $\rhd$ does give us an easily constructible  basis   $\{\psi^ \SSTS_{\SSTS_\la} 
y^{\rhd}_{\SSTS_\la}\psi_ \SSTT^{\SSTS_\la}  \mid \sts,\stt \in \Std(\la)\}$  
and presentations of all 
these cell-modules are given in \cite[Definition 5.9]{MR3004104}.  We claim that 
${\sf y}_{\aatchpair}$ annihilates any ($\rhd$)-cell-module  labelled by a $\la \in  \mathscr{P}_{\aatchpair}(n)$; in other words  
 ${\sf y}_{\aatchpair} \psi^ \SSTS_{\SSTS_\la} 
y^{\rhd}_{\SSTS_\la} \in \mathcal{H}^{\rhd \la}_n$ for  $\la \in  \mathscr{P}_{\aatchpair}(n)$.   Once we have proven the claim, we will deduce that the ideal is of the required rank, thus the spanning set is linearly independent (and hence a basis) as required. 
   The algebra is   then cellular (by its construction via idempotent ideals) with the stated basis (since we have a spanning set of the required rank).  
 Finally, we note that each layer of the cell chain contains an idempotent $e_{\stt_\la}$ and so the algebra is quasi-hereditary, as required.  
  We now turn to the proof of the claim.  If $m<\ell-1$ and $h_m<\sigma_{m+1}-\sigma_{m}$, or if $m=\ell-1$ and $h_{\ell-1}<e+\sigma_{0}-\sigma_{\ell-1}$    then the $m$th summand 
\begin{equation}\label{summer}
  y^{\succ}_{(\emptyset,\dots,\emptyset , (h_m+1) ,\emptyset, \dots,\emptyset)}\boxtimes 1_{\mathcal{H}_{n-h_m}^\sigma}
=
  e_{(\sigma_{m}, \sigma_{m}+1,\dots ,  \sigma_{m}+h_m) }
  \boxtimes 1_{\mathcal{H}_{n-h_m-1}^\sigma}
\end{equation} in \cref{idempotent} is an idempotent whose  residue sequence is not equal to that of any tableau $\SSTS \in \Std(\la)$  for $\la \in \mathscr{P}_{\aatchpair}(n)$ and therefore 
the claim is immediate.  If $h_m=\sigma_{m+1}-\sigma_{m}$, then the term in (\ref{summer}) is nilpotent and equal to 
\begin{equation}\label{summer2}
y_{h_m+1} e_{ \SSTT_{(\emptyset,\dots,\emptyset , {(h_m+1)} ,\emptyset, \dots,\emptyset)}}\boxtimes 1_{\mathcal{H}_{n-h_m-1}^\sigma}
=
y_{h_m+1} e_{\SSTT_{(\emptyset,\dots,\emptyset ,  {(h_m)} ,(1),\emptyset, \dots,\emptyset)}}\boxtimes 1_{\mathcal{H}_{n-h_m-1}^\sigma}.
\end{equation}
The idempotent in (\ref{summer2})
 annihilates  
   $\psi^ \SSTS_{\SSTS_\la} 
y^{\rhd}_\la  $
unless $\SSTS{\downarrow}_{\{1,\dots,h_m+1\}}$ is equal to  $\SSTT_{ (\emptyset,\dots,\emptyset ,(h_m), (1),  \emptyset,\dots,\emptyset  )} $.  
We now suppose that $\SSTS{\downarrow}_{\{1,\dots,h_m+1\}}$ is of this form and we set   $s_{m+1}\equiv\sigma_{m+1} \; {\rm mod }\; e$.  

 Since $\la \in \mathscr{P}_{\aatchpair}(n)$, we observe that  $[1,1,m+1]$ is the unique box in $\la$ of residue $s_{m+1}\in \ZZ/e\ZZ$ 
 in which we can place the integer $h_m+1$ (or any integer smaller than $h_m+1$) without violating the standard condition, by \cref{easy1}.  
 The presentation of the Specht module in \cite[Definition 5.9]{MR3004104} implies that 
$(i)$   $\psi_\w y^{\rhd}_\la   \in \mathcal{H}^{\rhd \la}_n$ for any $w \neq w^\SSTS_{\SSTS_\la}$ for some $\SSTS \in \Std(\la)$   with $\la \in \mathscr{P}_{\aatchpair}(n)$ (since 
every ($\rhd$)-Garnir belt has   fewer than  $e$ boxes) and $(ii)$
   $y_k y^{\rhd}_\la   \in \mathcal{H}^{\rhd \la}_n$ for any $1\leq k \leq n$.  

We  are now  ready to prove the claim.  
Using \cref{BK3}, we move the dot at the top of 
  $y_{h_m+1}\psi^ \SSTS_{\SSTS_\la} 
y^{\rhd}_\la  $ down the $(h_m+1)$th strand to obtain a linear combination
 of undecorated diagrams (in which we have undone some number of  crossings $s_m$-strands) and  $ \psi^ \SSTS_{\SSTS_\la} 
y_{\SSTS_\la (1,1,m+1) } y^{\rhd}_\la    $.  
By our above observation, all of these undecorated diagrams are labelled by non-standard $\la$-tableaux.  
Therefore all of these terms (and hence  $y_{h_m+1}\psi^ \SSTS_{\SSTS_\la} 
y^{\rhd}_\la  $) are zero,  by $(i)$ and $(ii)$.    
The claim and result follow.
 \end{proof}

Let   $\Bbbk$  be an integral domain.    
We define the   {\sf standard} or {\sf Specht} modules of  
$  \mathscr{H}^\sigma_n $   as follows, 
\begin{equation}  \label{identification}
  {\bf S}_\Bbbk(\lambda) = \{ \psi^\sts_{\stt_\la} + 
  \mathscr{H}^{\succ \la }
       \mid \sts \in \Std  (\lambda)\}
\end{equation}
 for $\la\in{\mathscr P}_{\aatchpair  }(n) $.   
We immediately deduce the following corollary of \cref{cellularitybreedscontempt}.
 \begin{cor}\label{GANRIENR}
 The   module $ {\bf S}_\Bbbk(\lambda) $ is the module generated by 
 $e_{\stt_\la}$ subject to the following relations:
 \begin{itemize}[leftmargin=*]
 \item $e_{\underline{i}}e_{\stt_\la}=\delta_{\underline{i}, {\rm res}({\stt_\la})}e_{\stt_\la}$ for $\underline{i}\in (\ZZ/e\ZZ)^n$;
\item $y_k  e_{\stt_\la}=0$ for $1\leq k \leq n$;
\item $\psi_k   e_{\stt_\la}=0$ for any $1\leq k< n$ such that $s_k(\stt_\la)$ is not row standard;
\item $\psi^{\sf S}_{\stt_\la}e_{\stt_\la}=0$ for ${\sf S}   \in  \RStd  (\lambda )\setminus  \Std (\lambda )$.  
 \end{itemize}
 \end{cor}
 
 \begin{proof}
We have already checked that all of these relations hold (and so one can define a homomorphism from the abstractly defined module with this presentation 
to $ {\bf S}_\Bbbk(\lambda)$) it only remains to check that these relations will suffice (i.e. the homomorphism is surjective).   
  We know that  $ {\bf S}_\Bbbk(\lambda)$ has a basis indexed by standard tableaux and so the result follows. 
  \end{proof}

 
   We now recall that the   cellular structure allows us to define bilinear forms, for each
 $\la \in \mathscr{P}_{\underline{\tau}}(n)$,  there  is a bilinear form  
  $ \langle\ ,\ \rangle ^{   \la}$ 
  on $ {\bf S}_\Bbbk(\la)  $,   which
is determined by
\begin{equation}\label{geoide}
   \psi ^{\stt_\la}_{\sts}\psi ^{\stt}_{\stt_\la}\equiv
  \langle     \psi ^{\sts}_{\stt_\la }, \psi ^{\stt}_{\stt_\la}
  \rangle ^{   \la} \;  
e_{\stt_\la}\pmod{\mathcal{H} ^{\succ \lambda}}
  \end{equation}
for any $\sts,\stt \in \Std(\lambda  )$.  
Let   $\Bbbk$  be a field of arbitrary characteristic.    
Factoring out by the radicals of these forms,  we obtain a complete set of non-isomorphic simple $ \mathscr{H}^\sigma_n$-modules 
   $$ 
  {\bf D}_\Bbbk(\lambda) =
 {\bf S}_\Bbbk(\lambda) /
  \rad( {\bf S}_\Bbbk(\lambda) ), \text{  $\la \in \mathscr{P} _{\aatchpair  }(n)$. } 
 $$

\begin{prop}\label{hfsaklhsalhskafhjksdlhjsadahlfdshjksadflhafskhsfajk} Let  $\lambda \in \mathscr{P}_{\aatchpair  } (n)$ and let $A_1\succ   A_2 \succ  \dots \succ  A_z$ denote the removable boxes of $\lambda$, totally ordered according to the $\succeq$-ordering.  The  restriction of ${\bf S}_\Bbbk (\la)$ has  an $\mathscr{H}_{n-1} ^\sigma $-module filtration 
\begin{align}\label{restriction}
0= {\bf S}^{z+1,\lambda}  \subset 
{\bf S}^{z,\lambda} 
\subset \dots 
\subset {\bf S}^{1,\lambda} =\Res_{\mathscr{H}_  {n-1}^\sigma} ( {\bf S}_\Bbbk (\la))
\end{align}
given by 
$${\bf S}^{x,\lambda} = \Bbbk\{\psi^\sts_{\stt_\la} \mid \Shape(\sts_{\leq n-1}) 
= \la-A_y \text{ for some }z\geq y \geq x			\}  .
$$
For each $1\leq r\leq z$, we have that  
\begin{align}\label{restriction2}
\varphi_r: {\bf S} (\la-A_r)\langle 	 \deg (A_r)	\rangle  
\cong
{\bf S}^{r,\lambda}  / {\bf S}^{r+1,\lambda} 
\qquad
:\psi_{\SSTS_{\leq n-1}} \mapsto \psi_{\SSTS_{\leq n-1}} \circ \psi^n_{\SSTT_\la (A_r) }
\end{align} 
\end{prop}
\begin{proof}
On the level of $\Bbbk$-modules, this is clear.  
Lifting this to $\mathscr{H}_{n-1}^\sigma$-modules is   a standard argument which  proceeds by checking the  relations of \cref{GANRIENR} in a routine manner.  
\end{proof}

\section{General light leaves bases for quiver Hecke algebras }\label{9999}

\newcommand{\Po}{\operatorname{Po}}

 The  principal idea of categorical Lie theory is to replace existing structures (combinatorics, bases, and presentations of Hecke algebras) with richer structures which keep track of more information.   
In this section, we replace the classical   tableaux combinatorics of symmetric groups (and quiver Hecke algebras) with that of paths in an alcove geometry.  This will allow us to construct ``light leaves" bases of these algebras, for which  $p$-Kazhdan--Lusztig is  baked-in  to the very definition.   
    The light leaves bases of ${\bf S}_\Bbbk(\la)$ are constructed in such a way as to keep track of not just the point $\la\in \mathbb{E}_{\aatchpair  }$ (or rather the single path,  $\SSTT_\la$, to the point $\la$)
  but of the many different ways    we can  get to the point $\la$ by a reduced path/word in the alcove geometry.  This extra  generality is essential when we wish to write bases in terms of ``2-generators" of the algebras of interest.

\subsection{The alcove geometry} 
\label{newsec3}
\color{black} 

 For ease of notation, we  set $H_m=h_0+\dots+h_m$ for $0\leq m <\ell$, 
and   $\aatch=\tau_0+\dots+\tau_{\ell-1}$.  
 For each $ 1\leq \I \leq n$ and $ 0\leq \M < \ell$ we let
  $\eps_{i,m}:=\eps_{(\tau_0+\dots+\tau_{m-1}) + i}$    denote a
formal symbol, and define an   $\aatch  $-dimensional real vector space 
\[
{\mathbb E}_{\aatchpair   }
=\bigoplus_{
	\begin{subarray}c 
	0 \leq m <  \ell   \\ 
		1\leq i \leq  \tau_m    
	\end{subarray}
} \mathbb{R}\varepsilon_{i,m}
\]
and $\overline{\mathbb E}_{\aatchpair  }$ to be the quotient of this space by the one-dimensional subspace spanned by 
\[\sum_{
	\begin{subarray}c 
	0 \leq m <  \ell   \\ 
		1\leq i \leq  \tau_m    
	\end{subarray}
} \varepsilon_{i,m}.\]
We have an inner product $\langle \; , \; \rangle$ on ${\mathbb E}_{\aatchpair  }$ given by extending
linearly the relations 
\[
\langle \varepsilon_{i,p} , \varepsilon_{j,q} \rangle= 
\delta_{\I,\J}\delta_{p,q}
\]
for all $1\leq \I, \J \leq n$ and $0 \leq p,q< \ell $, where
$\delta_{i,j}$ is the Kronecker delta.  
We identify $\lambda \in  {\mathcal C}_{\aatchpair  }(n)$ with an element of the integer lattice inside $\mathbb{E}_{\aatchpair  }  $ via the map
$$\lambda\longmapsto
 \sum_{\begin{subarray}c   0\leq {m}<\ell \\ 1\leq \I\leq \tau_m  \end{subarray}}
 (\lambda^{(m)})^T_\I \varepsilon_{i,m}$$
where $({-})^T$ is the transpose map. 
We let $\Phi$ denote the root system of type $A_{\aatch  -1}$ consisting of the roots 
$$\{\varepsilon_{i,p}-\varepsilon_{j,q}:  \ 0\leq p,q<\ell, \
1\leq i \leq \tau_p ,1\leq j \leq \tau_q, 
 \text{with}\ (i,p)\neq (j,q)\}$$
and $\Phi_0$ denote the root system of type $A_{\tau_0-1}\times\cdots\times A_{\tau_{\ell-1}-1}$ consisting of the roots 
$$\{\varepsilon_{i,m}-\varepsilon_{j,m}: 
  0\leq m<\ell, 1\leq i \neq j\leq \tau_m \}.$$
We choose $\Delta$ (respectively $\Delta_0$) to be the set of simple roots inside $\Phi$ (respectively  $\Phi_0$) of the form $\varepsilon_t-\varepsilon_{t+1}$ for some $t$.
Given $r\in\ZZ$ and $\alpha\in\Phi$ we define $s_{\alpha,re}$ to be the reflection which acts on ${\mathbb E}_{\aatchpair  }$ by
$$s_{\alpha,re}x=x-(\langle x,\alpha\rangle -re)\alpha$$
The group generated by the $s_{\alpha,0}$ with $\alpha\in\Phi$ (respectively $\alpha\in\Phi_0$) is isomorphic to the symmetric group $\mathfrak{S}_{\enn   }$ (respectively to $\mathfrak{S}_f:=\mathfrak{S}_{\tau_0}\times\cdots\times\mathfrak{S}_{\tau_{\ell-1}}$), while the group generated by the $s_{\alpha,re}$ with $\alpha \in\Phi$ and $r\in\ZZ$ is isomorphic to $\widehat{\mathfrak{S}}_{\enn   }$, the affine Weyl group of type $A_{\enn   -1}$. 
    We set $\alpha_0=\varepsilon_{\enn}-\varepsilon_1$ and $\Pi=\Delta\cup\{\alpha_0\}$.
The elements
 $S=
\{s_{\alpha,0}:\alpha\in\Delta\}\cup\{s_{\alpha_0,-e}\}
$ 
generate $\widehat{\mathfrak S}_{\enn}$. 

\begin{notn}
	We shall frequently find it convenient to refer to the generators in $S$ in terms of the elements of $\Pi$, and will abuse notation in two different ways. First, we will write $s_{\alpha}$ for $s_{\alpha,0}$ when $\alpha\in\Delta$ and $s_{\alpha_0}$ for $s_{\alpha_0,-e}$. This is unambiguous except in the case of the affine reflection $s_{\alpha_0,-e}$, where this notation has previously been used for the element $s_{\alpha,0}$. As the element $s_{\alpha_0,0}$ will not be referred to hereafter this should not cause confusion.
	Second, we will write $\alpha=\eps_i-\eps _{i+1}$ in all cases; if $i=\enn$ then all occurrences of $i+1$ should be interpreted modulo $\enn$ to refer to the index $1$.
\end{notn}

We shall consider a shifted action of the affine Weyl group $\widehat{\mathfrak{S}}_{\aatch  }$  on ${\mathbb E}_{h,l}$ 
by the element
$$
 \rho:= (\rho_{0}, \rho_{2}, \ldots, \rho_{\ell-1}) \in \ZZ^{\aatch  } \quad\text{where}\quad
\rho_m := (  \sigma_m+\aatch_m-1 ,  \sigma_m+\aatch_m-2,\dots,   \sigma_m ) \in \ZZ^{\tau_m},$$    
\noindent that is, given an element $w\in \widehat{\mathfrak{S}}_{\aatch  } $,  we set 
$
w\cdot x=w(x+\rho)-\rho.
$
 This shifted action induces a well-defined action on $\overline{\mathbb E}_{\aatchpair  }$; we will define various geometric objects in ${\mathbb E}_{\aatchpair  }$ in terms of this action, and denote the corresponding objects in the quotient with a bar without further comment.
We let ${\mathbb E} ({\alpha, re})$ denote the affine hyperplane
consisting of the points  
$${\mathbb E} ({\alpha, re}) = 
\{ x\in{\mathbb E}_{\aatchpair  } \mid  s_{\alpha,re} \cdot x = x\} .$$
Note that our assumption that $e>\tau_0+\dots+\tau_{\ell-1}$  implies that the origin does not lie on any hyperplane.   Given a
hyperplane ${\mathbb E} ( \alpha,re)$ we 
remove the hyperplane from ${\mathbb E}_{\aatchpair  }$ to obtain two
distinct subsets ${\mathbb 
	E}^{\great}(\alpha,re)$ and ${\mathbb E}^{\less}(\alpha,re)$
where the origin  lies in $  {\mathbb E}^{\less }(\alpha,re)$. The connected components of 
$$\overline{\mathbb E}_{\aatchpair  }  \setminus (\cup_{\alpha \in \Phi_0}\overline{\mathbb E}(\alpha,0))$$ are called chambers.
 The dominant chamber, denoted
$\overline{\mathbb E}_{\aatchpair   }^+ $, is defined to be 
$$\overline{\mathbb E}_{\aatchpair   }^+=\bigcap_{ \begin{subarray}c
	\alpha \in \Phi_0
	\end{subarray}
} \overline{\mathbb E}^{\less} (\alpha,0).$$
The connected components of $$\overline{\mathbb E}_{\aatchpair  }  \setminus (\cup_{\alpha \in \Phi,r\in \ZZ}\overline{\mathbb E}(\alpha,re))$$
are called alcoves, and any such alcove is a fundamental domain for the action of the group $  \widehat{\mathfrak{S}}_{\aatch  }$ on
the set $\Alc$ of all such alcoves. We define the {\sf fundamental alcove} $A_0$ to be the alcove containing the origin (which is inside the dominant chamber).
 We have a bijection from $\widehat{\mathfrak{S}}_{\aatch  }$ to $\Alc$ given by $w\longmapsto wA_0$. Under this identification $\Alc$ inherits a right action from the right action of $\widehat{\mathfrak{S}}_{\aatch  }$ on itself.
Consider the subgroup 
$$
\mathfrak{S}_f:=
\mathfrak{S}_{\tau_0}\times\cdots\times\mathfrak{S}_{\tau_{\ell-1}} \leq \widehat{\mathfrak{S}}_{\aatch  }.
$$
The dominant chamber is a fundamental domain for the action of $\mathfrak{S}_f$ on the set of chambers in $\overline{\mathbb E}_{\aatchpair  }$.  
We let $ \mathfrak{S}^f$ denote the set of minimal length representatives for right cosets $\mathfrak{S}_f \backslash \widehat{\mathfrak{S}}_{\aatch  }$.  So multiplication gives a bijection $\mathfrak{S}_f\times \mathfrak{S}^f \to  \widehat{\mathfrak{S}}_{\aatch  }$. 
 This induces a bijection between right cosets and the alcoves in our dominant chamber. 
  Under this identification, alcoves are partially ordered  by the Bruhat-ordering on $\widehat{ \mathfrak{S}}_h$ which is  a coarsening of   the {\em opposite}  of the order $\psucc  $.

 \color{black}
If the intersection of a hyperplane $\overline{\mathbb E}(\alpha,re)$ with the closure of an alcove $A$ is generically  of codimension one in $\overline{\mathbb E}_{\aatchpair  }$ then we call this intersection a {\sf wall} of $A$. The fundamental alcove $A_0$ has walls corresponding to $\overline{\mathbb E}(\alpha,0)$ with $\alpha\in\Delta$ together with an affine wall $\overline{\mathbb E}(\alpha_0,e)$. We will usually just write $\overline{\mathbb E}(\alpha)$ for the walls $\overline{\mathbb E}(\alpha,0)$ (when $\alpha\in\Delta$) and $\overline{\mathbb E}(\alpha,e)$ (when $\alpha=\alpha_0$). We regard each of these walls as being labelled by a distinct colour (and assign the same colour to the corresponding element of $S$). Under the action of $\widehat{\mathfrak{S}}_{\aatch  }$ each wall of a given alcove $A$ is in the orbit of a unique wall of $A_0$, and thus inherits a colour from that wall. We will sometimes use the right action of $\widehat{\mathfrak{S}}_{\aatch  }$ on $\Alc$. Given an alcove $A$ and an element $s\in S$, the alcove $As$ is obtained by reflecting $A$ in the wall of $A$ with colour corresponding to the colour of $s$. With this observation it is now easy to see that if $w=s_{1}\ldots s_{t}$ where the $s_i$ are in $S$ then $wA_0$ is the alcove obtained from $A_0$ by successively reflecting through the walls corresponding to $s_1$ up to $s_h$. 
  We will call a multipartition {\sf regular}  if its image in $\overline{\mathbb E}_{h,l}$ lies in some alcove; those multipartitions whose images lies on one or more walls will be called {\sf  singular}.

\subsection{Paths in the geometry}\color{black}
We now   develop a path combinatorics inside our geometry.  
Given   a map  $p: 
\{1,\mydots
, n\}\to \{1,\mydots ,   \aatch      \}$ we define points $\SSTP(k)\in
{\mathbb E}_{\aatchpair  }$ by
\[
\SSTP(k)=\sum_{1\leq i \leq k}\varepsilon_{p(i)} 
\]
for $1\leq i \leq n$. 
We define the associated path of length $n$    by $$\SSTP=\left(
\varnothing=\SSTP(0),\SSTP(1),\SSTP(2), \ldots, \SSTP(n) \right) $$ and we say that 
the path has shape $\pi= \SSTP(n) \in   {\mathbb E}_{\aatchpair  }$.   
 We also denote this path by 
$$\SSTP=(\varepsilon_{p(1)},\ldots,\varepsilon_{p(n)}).$$
Given $\lambda\in \mathscr{C}_{\tau}(n)$ we let $\Path(\lambda)$ denote the set of paths of length $n$ with shape $\lambda$. We define $\Path_{\aatchpair  }(\lambda)$ to be the subset of $\Path(\lambda)$ consisting of those paths lying entirely inside the dominant chamber; i.e. those $\SSTP$ such that $\SSTP(i)$ is dominant for all $0\leq i\leq n$. 
 We let $\Path_{\aatchpair  }(n) = \cup _{\la\in \mathscr{P}_{\underline{\tau}}(n)}\Path_{\aatchpair  }(\la)$.

Given a path $\mathsf{T}$ defined by such a map $p$ of length $n$ and shape $\lambda$ we can write each $p(j)$ uniquely in the form $\eps_{p(j)}=\eps_{m_j,c_j}$ where $0\leq m_j<\ell$ and $1\leq c_j\leq \tau_j$. We record these elements in a tableau of shape $\lambda^T$ by induction on $j$, where we place the positive integer $j$ in the first empty box in the $c_j$th column of component $m_j$.
By definition, such a tableau will have entries increasing down columns; if $\lambda$ is a multipartition then the entries also increase along rows if and only if the given path is in $\Path_{\aatchpair  }(\lambda)$, and hence there is a bijection between $\Path_{\aatchpair  }(\lambda)$ and $\Std(\lambda)$. 
For this reason we will sometimes refer to paths as tableaux, to emphasise that what we are doing is generalising the classical tableaux combinatorics for the symmetric group. 

  \begin{eg}
  Let    $\sigma=(0,3,6)\in  \ZZ ^3$ and $e=9$. 
  For $\la=((2,1),(2,1),(1^3))$, the  standard $\la$-tableaux  of \cref{oftttt} correspond to the paths 
$$\stt_\la= 
(\eps_1,\eps_2,\eps_3,\eps_4,\eps_5,\eps_1,\eps_3,\eps_{4},\eps_{5} ,\eps_{1},\eps_{3}, \eps_{5})
$$
$$
\sts= (\eps_1,\eps_1,\eps_3, \eps_5,\eps_4, \eps_2,\eps_3,\eps_{4},\eps_{5} ,\eps_{1},\eps_{3}, \eps_{5})   $$
\end{eg}

  \color{black}  
Given a path $\SSTP$ we define $$\res(\SSTP)=(\res_\SSTP(1),\ldots,\res_\SSTP(n))$$ where $\res_\SSTP(i)$ denotes the residue of the box labelled by $i$ in the tableau corresponding to $\SSTP$.

\begin{figure}[ht] 
	$$ 
	\scalefont{0.8}  \begin{minipage}{11cm} \begin{tikzpicture}[scale=1.1] 
	\fill[gray!30](0,0)--++(120:8)--++(180:0.2)--++(-60:8)--(0:0.2); 
	\fill[gray!30](0,0)--++(60:8)--++(0:0.2)--++(-120:8)--(180:0.2); 

	\begin{scope} \clip(0,0)--++(120:0.4*20)--++(0:0.4*20)--++(-120:0.4*20);  
		\foreach \i in {0,1,...,20}
		{    
			\path(0,0)--++(120:0.4*\i)  coordinate (a\i);
			\path(0,0)--++(60:0.4*\i)  coordinate (b\i);
			\path(0,0)--++(0:0.4*\i)  coordinate (c\i);
			\path(0,0)--++(120:0.4*20)--++(0:0.4*\i)  coordinate (d\i);   
			\path(0,0)--++(120:0.4*20)--++(0:0.4*20)--++(180:0.4*\i)  coordinate (e\i);   
			\draw[gray,densely dotted](a\i)--(b\i);
			\draw[gray,densely dotted](d\i)--(c\i);
			\draw[gray,densely dotted](a\i)--(e\i);
		}
		\path(0,0)--++(120:0.4)        --++(60:0.4) coordinate (hi);

		\path (0,0) coordinate (origin);
	\end{scope}
	
	\draw[->](-3,1.2)--++(120:0.4) node[anchor=south east]{$\varepsilon_1$};
	\draw[->](-3,1.2)--++(0:0.4) node[anchor=west]{$\varepsilon_2$};
	\draw[->](-3,1.2)--++(-120:0.4) node[anchor=north east]{$\varepsilon_3$};
	
	\draw[very thick, lime](0,0)--++(120:0.4*5)
	--++(60:0.4*5)--++(00:0.4*5); 

	\draw[very thick, lime](0,0)
	--++(120:0.4*5)--++(60:0.4*5)
	--++(120:0.4*5)
	--++(60:0.4*5)--++(00:0.4*5)--++(-60:0.4*5)--++(-120:0.4*5)--++(-180:0.4*5);
	
	\draw[very thick, lime](0,0)
	--++(120:0.4*5)--++(60:0.4*5)
	--++(120:0.4*5) --++(180:0.4*5)--++(120:0.4*5);  
	\draw[very thick, lime](0,0)
	--++(120:0.4*5)--++(60:0.4*5)
	--++(120:0.4*5)--++(-120:0.4*5); 
	
	\draw[very thick, cyan](0,0)--++(60:0.4*5)
	--++(120:0.4*5)--++(180:0.4*5);

	\draw[very thick, lime](0,0)
	--++(120:0.4*5)--++(60:0.4*5)
	--++(120:0.4*5)
	--++(60:0.4*5)--++(00:0.4*5)--++(-60:0.4*5)--++(-120:0.4*5)--++(-180:0.4*5);
	
	\path(0,0)--++(120:2) coordinate (hi);
	\draw[very thick, magenta](hi)--++(120:0.4*5)
	--++(60:0.4*5)--++(00:0.4*5) --++(-60:0.4*5)--++(-120:0.4*5)--++(-180:0.4*5);
	
	\path(0,0)--++(120:6)--++(0:2) coordinate (hi2);
	\draw[very thick, magenta](hi2)--++(120:0.4*5)
	--++(180:0.4*5);
	\path(0,0)--++(120:6)--++(0:4) coordinate (hi2);
	\draw[very thick, magenta](hi2)--++(60:0.4*5)
	--++(0:0.4*5);

	\path(0,0)--++(120:8) coordinate (hi2); 
	\draw[very thick, lime,densely dotted] (hi2)--++(60:0.4);
	\draw[very thick, magenta,densely dotted] (hi2)--++(120:0.4);
	\path(hi2) --++(0:2) coordinate (hi2); 
	\draw[very thick, magenta,densely dotted] (hi2)--++(60:0.4);
	\draw[very thick, cyan,densely dotted] (hi2)--++(120:0.4);
	\path(hi2) --++(0:2) coordinate (hi2); 
	\draw[very thick, cyan,densely dotted] (hi2)--++(60:0.4);
	\draw[very thick, lime,densely dotted] (hi2)--++(120:0.4);
	\path(hi2) --++(0:2) coordinate (hi2); 
	\draw[very thick, lime,densely dotted] (hi2)--++(60:0.4);
	\draw[very thick, magenta,densely dotted] (hi2)--++(120:0.4);
	\path(hi2) --++(0:2) coordinate (hi2); 
	\draw[very thick, magenta,densely dotted] (hi2)--++(60:0.4);
	\draw[very thick, cyan,densely dotted] (hi2)--++(120:0.4);
	\draw[very thick, cyan](0,0)
	--++(60:0.4*5)--++(120:0.4*5)
	--++(60:0.4*5)--++(0:0.4*5)       --++(60:0.4*5);
	
	\draw[very thick, cyan](0,0)
	--++(60:0.4*5)--++(120:0.4*5)       --++(60:0.4*5)
	--++(120:0.4*5)--++(180:0.4*5)       --++(-120:0.4*5)  --++(-60:0.4*5); 
	
	\path(0,0)--++(120:0.8)--++(0:0.4) coordinate (hi); 
	\draw(hi) node {\astrosun};
	\draw[very thick](hi)
	--++(120:0.4*1)
	--++(0:0.4*1)
	--++(120:0.4*1)
	--++(0:0.4*1)
	--++(120:0.4*1)
	--++(0:0.4*1)
	--++(120:0.4*3)
	--++(-120:0.4*1)
	--++(120:0.4*1)--++(120:0.4*1)
	--++(120:0.4*1)
	--++(0:0.4*1)--++(120:0.4*1)
	--++(120:0.4*1)
	--++(-120:0.4*1)
	--++(120:0.4*1)
	--++(-120:0.4*1)
	--++(120:0.4*1)
	--++(-120:0.4*1)
	--++(120:0.4*3)
	--++(0:0.4*1)--++(120:0.4*1)--++(120:0.4*1)
	--++(120:0.4*1)--++(-120:0.4*1)--++(120:0.4*1)
	;
	\path ++(60:0.4*4) coordinate (a);
	\path (a) ++(120:0.4*4) coordinate (b);
	\path (b) ++(120:0.4) coordinate (c); 
	\path (c) ++(180:0.4) coordinate (d);
	\path (d) ++(60:0.4) coordinate (e);
	\path (e) ++(120:0.4) coordinate (f);
	\path (f) ++(120:0.4*3) coordinate (g);
	\path (g) ++(180:0.4*3) coordinate (h);
	\path (h) ++(60:0.4) coordinate (i);
	\path (i) ++(120:0.4) coordinate (j);
	\path (j) ++(120:0.4) coordinate (k);
	\path (k) ++(180:0.4) coordinate (l);
	\node at (b)[circle,fill,inner sep=1.5pt]{}; 
	\node at (d)[circle,fill,inner sep=1.5pt]{}; 
	\node at (f)[circle,fill,inner sep=1.5pt]{}; 
	\node at (h)[circle,fill,inner sep=1.5pt]{}; 
	\node at (j)[circle,fill,inner sep=1.5pt]{}; 
	\node at (l)[circle,fill,inner sep=1.5pt]{}; 
\end{tikzpicture}  \end{minipage}
\begin{minipage}{2cm} \begin{tikzpicture}[scale=0.485]
	 
	\draw[thick](0,0)--++(-90:20)--++(0:1)--++(90:15)--++(0:2)--++(90:5)--++(180:3);
	
	\clip(0,0)--++(-90:20)--++(0:1)--++(90:15)--++(0:2)--++(90:5)--++(180:3);
	\foreach \i in {0,1,2,...,20}
	{
		\path  (0,0)++(-90:1*\i cm)  coordinate (a\i);
		\path  (0.5,-0.5)++(-90:1*\i cm)  coordinate (b\i);
		\path  (1.5,-0.5)++(-90:1*\i cm)  coordinate (c\i);
		\path  (2.5,-0.5)++(-90:1*\i cm)  coordinate (d\i);
		\draw[thick] (a\i)--++(0:3);
		\draw[thick] (1,0)--++(-90:5);
		\draw[thick] (2,0)--++(-90:5);
	}
\draw(b0) node {$1$};
	\draw(c0) node {$2$};
	
	\draw(b1) node {$3$};
	\draw(c1) node {$4$};
	
	\draw(b2) node {$5$};
	\draw(c2) node {$6$};
	
	\draw(b3) node {$7$};
	\draw(b4) node {$8$};
	\draw(b5) node {$9$};
	\draw(d0) node {$10$};   
	\draw(b6) node {$11$};
	\draw(b7) node {$12$};
	\draw(b8) node {$13$};
	\draw(c3) node {$14$};
	\draw(b9) node {$15$};
	\draw(b10) node {$16$};
	\draw(d1) node {$17$};\draw(b11) node {$18$};
	\draw(d2) node {$19$};\draw(b12) node {$20$};
	\draw(d3) node {$21$};\draw(b13) node {$22$};
	
	\draw(b14) node {$23$};
	\draw(b15) node {$24$};
	\draw(b16) node {$26$};
	\draw(c4) node {$25$};
	\draw(b17) node {$27$};  \draw(b18) node {$28$};
	\draw(d4) node {$29$};
	\draw(b19) node {$30$};
	\end{tikzpicture} \end{minipage}
	$$
  \caption{An alcove path in $\overline{\mathbb E}_{3,1}^+$ in $\Path_{\aatchpair  }(3^5,1^{15})$ and the corresponding tableau in $\Std(3^5,1^{15})$. The black vertices denote vertices on the path in the orbit of the origin.}
\label{diag2}
\end{figure}
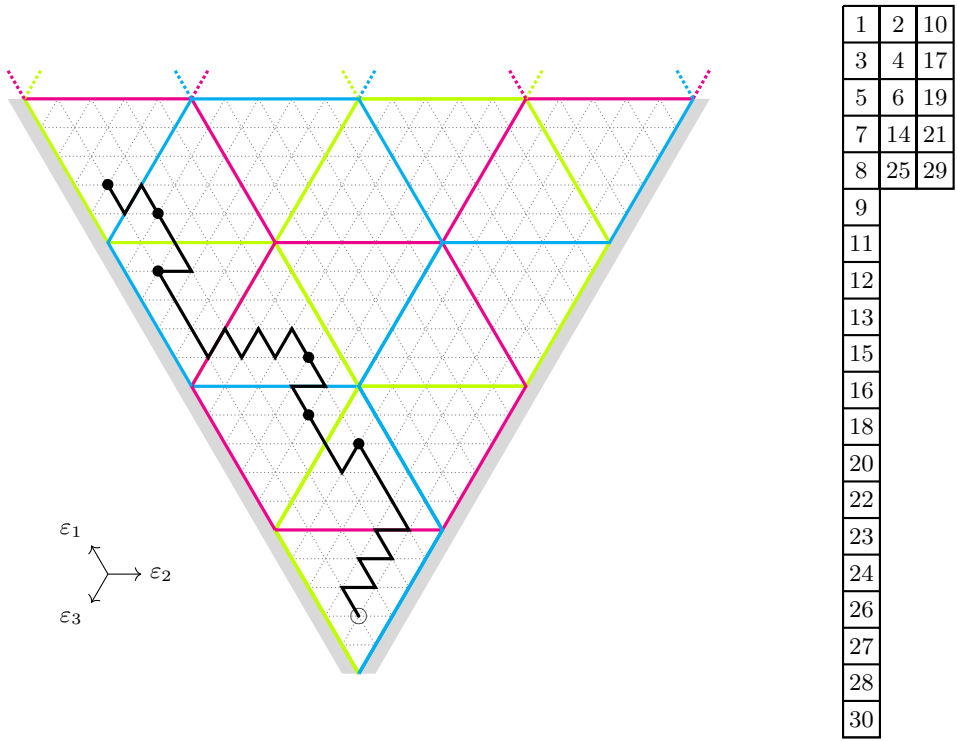

Given paths $\SSTP=(\varepsilon_{p(1)},\ldots,\varepsilon_{p(n)})$ and $\SSTQ=(\varepsilon_{q(1)},\ldots,\varepsilon_{q(n)})$ we say that $\SSTP\sim\SSTQ$ if there exists an $\alpha=\varepsilon_{i,p}-\varepsilon_{j,q}\in\Phi$ and $r\in\ZZ$ and $s\leq n$ such that
$$\SSTP(s)\in{\mathbb E}(\alpha,re)\qquad \text{ 
and  }
\qquad \varepsilon_{q(t)}=\left\{\begin{array}{ll}
\varepsilon_{p(t)}& \ \text{for}\ 1\leq t\leq s\\
s_{\alpha}\varepsilon_{p(t)}& \ \text{for}\ s+1\leq t\leq n.\end{array}\right.$$ 

In other words the paths $\SSTP$ and $\SSTQ$ agree up to some point $\SSTP(s)=\SSTQ(s)$ which lies on ${\mathbb E}(\alpha,re)$, after which each $\SSTQ(t)$ is obtained from $\SSTP(t)$ by reflection in ${\mathbb E}(\alpha,re)$. We extend $\sim$ by transitivity to give an equivalence relation on paths, and say that two paths in the same equivalence class are related by a series of {\sf wall reflections of paths} and given $\SSTS \in \Path_{\aatchpair  }(n)$ we set $[\SSTS]  = \{ \SSTT \in \Path_{\aatchpair  }(n) \mid \SSTS\sim\SSTT\}$.

We recast the degree of a tableau in the path-theoretic setting as follows.

\begin{defn}\label{Soergeldegreee}
Given a path $\sts=(\sts(0),\sts(1),\sts(2), \ldots, \sts(n))$  we set
$\deg(\sts(0))=0$ and define  
 \[
 \deg(\sts ) = \sum_{1\leq k \leq n} d (\sts(k),\sts(k-1)), 
 \]
 where $d(\sts(k),\sts(k-1))$ is defined as follows. 
For $\alpha\in\Phi$ we set $d_{\alpha}(\sts(k),\sts(k-1))$ to be
\begin{itemize}
\item $+1$ if $\sts(k-1) \in 
   {\mathbb E}(\alpha,re)$ and 
   $\sts(k) \in 
   {\mathbb E}^{\less}(\alpha,re)$;
   
\item $-1$ if $\sts(k-1) \in 
   {\mathbb E}^{\great}(\alpha,re)$ and 
   $\sts(k) \in 
   {\mathbb E}(\alpha,re)$;
\item $0$ otherwise.  
   \end{itemize}
We let 
$$ 
{\rm deg}(\SSTS)= \sum _{1\leq k \leq n }\sum_{\alpha \in \Phi}d_\alpha(\sts(k-1),\sts(k)).$$  
%
%
%
%
   We say that $\SSTP=(\varepsilon_{p(1)},\ldots,\varepsilon_{p(n)})$  is a {\sf reduced path} if 
$d_\alpha(\SSTP(k-1),\SSTP(k ))=0 $ for  $1\leq k \leq n$ and $\alpha \in \Pi$.  
 \end{defn}
   

This  definition of a reduced path is easily seen to be equivalent to that of \cite[Section 2.3]{cell4us2}.

\begin{figure}[ht!]
 
$$
   \begin{minipage}{2.6cm}\begin{tikzpicture}[scale=0.75] 
    
    \path(0,0) coordinate (origin);
          \foreach \i in {0,1,2}
  {
    \path  (origin)++(60:1*\i cm)  coordinate (a\i);
    \path (origin)++(0:1*\i cm)  coordinate (b\i);
     \path (origin)++(-60:1*\i cm)  coordinate (c\i);
    }
  
      \path(3,0) coordinate (origin);
          \foreach \i in {0,1,2}
  {
    \path  (origin)++(120:1*\i cm)  coordinate (d\i);
    \path (origin)++(180:1*\i cm)  coordinate (e\i);
     \path (origin)++(-120:1*\i cm)  coordinate (f\i);
    }

  \foreach \i in {0,1,2}
  {
    \draw[gray, densely dotted] (a\i)--(b\i);
        \draw[gray, densely dotted] (c\i)--(b\i);
    \draw[gray, densely dotted] (d\i)--(e\i);
        \draw[gray, densely dotted] (f\i)--(e\i);
            \draw[gray, densely dotted] (a\i)--(d\i);
                \draw[gray, densely dotted] (c\i)--(f\i);
 
     }
  \draw[very thick, magenta] (0,0)--++(0:3) ;
    \draw[very thick, cyan] (3,0)--++(120:3) coordinate (hi);
        \draw[very thick, darkgreen] (hi)--++(-120:3) coordinate  (hi);

        \draw[very thick, darkgreen] (hi)--++(-60:3) coordinate  (hi);

    \draw[very thick, cyan] (hi)--++(60:3) coordinate (hi);

\path                 (hi)  --++(180:0.6*4) 
 coordinate (hiyer);

      \path(0,0)--++(-60:5*0.6)--++(120:2*0.6)--++(0:0.6) coordinate (hi);
     \path(0,0)--++(0:4*0.6)--++(120:3*0.6)           coordinate(hi2) ; 
     \path(0,0)--++(0:4*0.6)          coordinate(hi3) ;

     
          \path(0,0)  --++(0:1)
            coordinate (hi);

      \draw[very thick,<-] (hi) --++(60:1);

 \end{tikzpicture}\end{minipage}
 \qquad
 \quad
   \begin{minipage}{2.6cm}\begin{tikzpicture}[scale=0.75] 
    
    \path(0,0) coordinate (origin);
          \foreach \i in {0,1,2}
  {
    \path  (origin)++(60:1*\i cm)  coordinate (a\i);
    \path (origin)++(0:1*\i cm)  coordinate (b\i);
     \path (origin)++(-60:1*\i cm)  coordinate (c\i);
    }
  
      \path(3,0) coordinate (origin);
          \foreach \i in {0,1,2}
  {
    \path  (origin)++(120:1*\i cm)  coordinate (d\i);
    \path (origin)++(180:1*\i cm)  coordinate (e\i);
     \path (origin)++(-120:1*\i cm)  coordinate (f\i);
    }

  \foreach \i in {0,1,2}
  {
    \draw[gray, densely dotted] (a\i)--(b\i);
        \draw[gray, densely dotted] (c\i)--(b\i);
    \draw[gray, densely dotted] (d\i)--(e\i);
        \draw[gray, densely dotted] (f\i)--(e\i);
            \draw[gray, densely dotted] (a\i)--(d\i);
                \draw[gray, densely dotted] (c\i)--(f\i);
 
     }
  \draw[very thick, magenta] (0,0)--++(0:3) ;
    \draw[very thick, cyan] (3,0)--++(120:3) coordinate (hi);
        \draw[very thick, darkgreen] (hi)--++(-120:3) coordinate  (hi);

        \draw[very thick, darkgreen] (hi)--++(-60:3) coordinate  (hi);

    \draw[very thick, cyan] (hi)--++(60:3) coordinate (hi);

\path                 (hi)  --++(180:0.6*4) 
 coordinate (hiyer);

      \path(0,0)--++(-60:5*0.6)--++(120:2*0.6)--++(0:0.6) coordinate (hi);
     \path(0,0)--++(0:4*0.6)--++(120:3*0.6)           coordinate(hi2) ; 
     \path(0,0)--++(0:4*0.6)          coordinate(hi3) ;

     
          \path(0,0)  --++(0:2)
            coordinate (hi);

      \draw[very thick,->] (hi) --++(-120:1);

 \end{tikzpicture}\end{minipage} 
\qquad
\begin{minipage}{2.6cm}\begin{tikzpicture}[xscale=0.75,yscale=-0.75] 
    
    \path(0,0) coordinate (origin);
          \foreach \i in {0,1,2}
  {
    \path  (origin)++(60:1*\i cm)  coordinate (a\i);
    \path (origin)++(0:1*\i cm)  coordinate (b\i);
     \path (origin)++(-60:1*\i cm)  coordinate (c\i);
    }
  
      \path(3,0) coordinate (origin);
          \foreach \i in {0,1,2}
  {
    \path  (origin)++(120:1*\i cm)  coordinate (d\i);
    \path (origin)++(180:1*\i cm)  coordinate (e\i);
     \path (origin)++(-120:1*\i cm)  coordinate (f\i);
    }

  \foreach \i in {0,1,2}
  {
    \draw[gray, densely dotted] (a\i)--(b\i);
        \draw[gray, densely dotted] (c\i)--(b\i);
    \draw[gray, densely dotted] (d\i)--(e\i);
        \draw[gray, densely dotted] (f\i)--(e\i);
            \draw[gray, densely dotted] (a\i)--(d\i);
                \draw[gray, densely dotted] (c\i)--(f\i);
 
     }
  \draw[very thick, magenta] (0,0)--++(0:3) ;
    \draw[very thick, cyan] (3,0)--++(120:3) coordinate (hi);
        \draw[very thick, darkgreen] (hi)--++(-120:3) coordinate  (hi);

        \draw[very thick, darkgreen] (hi)--++(-60:3) coordinate  (hi);

    \draw[very thick, cyan] (hi)--++(60:3) coordinate (hi);

\path                 (hi)  --++(180:0.6*4) 
 coordinate (hiyer);

      \path(0,0)--++(-60:5*0.6)--++(120:2*0.6)--++(0:0.6) coordinate (hi);
     \path(0,0)--++(0:4*0.6)--++(120:3*0.6)           coordinate(hi2) ; 
     \path(0,0)--++(0:4*0.6)          coordinate(hi3) ;

     
          \path(0,0)  --++(0:1)
            coordinate (hi);

      \draw[very thick,<-] (hi) --++(60:1);

 \end{tikzpicture}\end{minipage}
 \qquad
 \quad
   \begin{minipage}{2.6cm}\begin{tikzpicture}[xscale=0.75,yscale=-0.75] 
    
    \path(0,0) coordinate (origin);
          \foreach \i in {0,1,2}
  {
    \path  (origin)++(60:1*\i cm)  coordinate (a\i);
    \path (origin)++(0:1*\i cm)  coordinate (b\i);
     \path (origin)++(-60:1*\i cm)  coordinate (c\i);
    }
  
      \path(3,0) coordinate (origin);
          \foreach \i in {0,1,2}
  {
    \path  (origin)++(120:1*\i cm)  coordinate (d\i);
    \path (origin)++(180:1*\i cm)  coordinate (e\i);
     \path (origin)++(-120:1*\i cm)  coordinate (f\i);
    }

  \foreach \i in {0,1,2}
  {
    \draw[gray, densely dotted] (a\i)--(b\i);
        \draw[gray, densely dotted] (c\i)--(b\i);
    \draw[gray, densely dotted] (d\i)--(e\i);
        \draw[gray, densely dotted] (f\i)--(e\i);
            \draw[gray, densely dotted] (a\i)--(d\i);
                \draw[gray, densely dotted] (c\i)--(f\i);
 
     }
  \draw[very thick, magenta] (0,0)--++(0:3) ;
    \draw[very thick, cyan] (3,0)--++(120:3) coordinate (hi);
        \draw[very thick, darkgreen] (hi)--++(-120:3) coordinate  (hi);

        \draw[very thick, darkgreen] (hi)--++(-60:3) coordinate  (hi);

    \draw[very thick, cyan] (hi)--++(60:3) coordinate (hi);

\path                 (hi)  --++(180:0.6*4) 
 coordinate (hiyer);

      \path(0,0)--++(-60:5*0.6)--++(120:2*0.6)--++(0:0.6) coordinate (hi);
     \path(0,0)--++(0:4*0.6)--++(120:3*0.6)           coordinate(hi2) ; 
     \path(0,0)--++(0:4*0.6)          coordinate(hi3) ;

     
          \path(0,0)  --++(0:2)
            coordinate (hi);

      \draw[very thick,->] (hi) --++(-120:1);

 \end{tikzpicture}\end{minipage} 
 $$
\caption{The  first and second paths have degrees $-1$ and $+1$ respectively.  The third and fourth paths have degree 0.  Here we take the convention that the origin is below the pink hyperplane.  
}
 \label{upanddown22}
\end{figure}
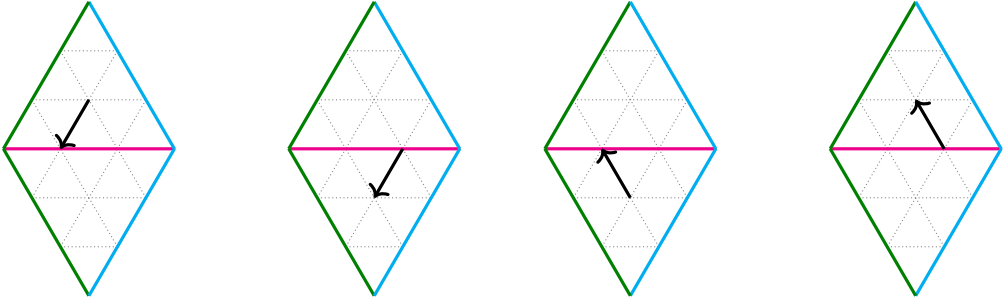

There exist a unique reduced path in each $\sim$-equivalence class (and, of course, each reduced path belongs to some 
$\sim$-equivalence class and so $\sim$-classes and reduced paths are in bijection).  
 We remark that $\SSTT_\mu$, the maximal path   in the reverse cylindric  ordering $\succ$,  is an example of a  reduced path.  
Given $\SSTS \in \Path_{\aatchpair  }(n)$, we let $\min [\SSTS]$ denote the minimal path in the $\sim$-equivalence class containing $\SSTS$.  
 Given a reduced path $\SSTP_\la \in \Path_{\aatchpair  } (\la)$, we have that 
$$ \mathscr{H}^\sigma_n e_{\SSTP_\la } =  {\bf P}(\la)\oplus \bigoplus_{\mu\succ\la} k^\mu_{\SSTP_\la}{\bf P}(\mu)$$
decomposes (in a unitriangular fashion) as a sum of projective indecomposable modules for some generalised $p$-Kostka coefficients  $k^\mu_{\SSTP_\la}\in \Bbbk$.   
   In general, we have 
 $$
  \mathscr{H}^\sigma_n e_{\SSTP_\la } \not \cong   \mathscr{H}^\sigma_n e_{\SSTQ_\la } 
 $$
 for reduced paths $ \SSTP_\la  , \SSTQ_\la \in \Path_{\aatchpair  } (\la)$ and so the choice of reduced path does  matter.   
(This is not surprising,  the auxiliary steps in  Soergel's algorithm 
for calculating Kazhdan--Lusztig polynomials produces a different pattern depending on the choice of reduced expression.)  
 However, they do agree modulo  higher terms under $\psucc$, as we shall soon see (and indeed, after the cancellations in Soergel's algorithm one obtains that the Kazhdan--Lusztig polynomials are independent of  choices of reduced expressions).

\begin{lem}\label{alemmmerd}
Given  $\la\in \mathscr{P}_ {\aatchpair}(n)$, let $\SSTP_\la $, $\SSTQ_\la$, 
 $\SSTS_\la$ be  any triple of   reduced paths in $\Path_{\aatchpair  }(\la)$.    
The element $e_{\SSTP_\la}$ generates $\mathcal{H}^{\succeq \la}/\mathcal{H}^{\succ  \la}$ and     
moreover $$\psi^{\SSTP_\la}_{\SSTQ_\la} 
\psi^{\SSTQ_\la}_{\SSTS_\la} =
k\psi^{\SSTP_\la}_{\SSTS_\la} 
+ \mathcal{H}^{\succ  \la} $$for some $k\in \Bbbk \setminus \{0\}$.   
\end{lem}

\begin{proof}
Let    ${{\sf R}_\la }  $ be any reduced path in $ \Path_{\aatchpair  }(\la)$.  
Two paths have the same residue sequence if and only if they belong to the same $\sim$-class.  
If $\SSTS \sim {{\sf R}_\la }$ then either $\SSTS={{\sf R}_\la }$ or $\SSTS$ terminates at a point $\mu \succ \la$.    Thus, we have that  
  \begin{equation}\label{eqnhg}
e_{{\sf R}_\la}  {\bf S}_\Bbbk(\nu) \neq 0 \text{ implies }\nu
\psucc  \la \text{ or }\nu=\la
\text{ and } e_{{\sf R}_\la}    {\bf S} (\la)=e_{{\sf R}_\la}   {\bf D}_\Bbbk (\la)= \psi^{{\sf R}_\la }_{\SSTT_\la}.  \end{equation} 
 This implies that $e_{{\sf R}_\la}\in \mathcal{H}^{\succeq \la}/\mathcal{H}^{\succ  \la}$ and therefore generates $ \mathcal{H}^{\succeq \la}/\mathcal{H}^{\succ  \la}$ and belongs to the simple head of the Specht module; the result follows.  
    \end{proof}

\begin{defn} Given two paths 
	$$\SSTP=(\eps_{i_1},\eps_{i_2},\dots, \eps_{i_p}) \in \Path(\mu)
	\quad\text{and}\quad
	\SSTQ=(\eps_{j_1},\eps_{j_2},\dots, \eps_{j_q}) 
	\in \Path(\nu)$$  we define the {\sf naive concatenated path} 
	$$\SSTP\boxtimes \SSTQ =
	(\eps_{i_1},\eps_{i_2},\dots, \eps_{i_p}, \eps_{j_1},\eps_{j_2},\dots, \eps_{j_q}) 
	\in \Path (\mu+\nu).$$ 
\end{defn}

\subsection{Branching coefficients} 
We now discuss how one can  think of a permutation as a  morphism between pairs of paths in the alcove geometries of \cref{newsec3}.  
 Let $\la \in \mathcal{C}_{\aatchpair   }(n)$.
 Given a pair of paths $\SSTS,\SSTT\in \Path(\la)$ we write the steps in 
 $\SSTS$ and $\SSTT$ in sequence along the    top     and    bottom        edges of a frame, respectively.
 We can now reinterpret the element 
    $w^{\SSTS}_{\SSTT}\in \mathfrak{S}_n$ (of \cref{sec3}) as the 
 unique step-preserving permutation   with the minimal number of crossings.

In the following (running) example we label our paths by $\SSTP_\emp(=\SSTT_{(3^3)})$ and $\SSTP_\al^\flat$. 
For this section, we do not need to know what inspires this notation; 
however,    all will become clear in  \cref{geners}.

\begin{eg}\label{concanetaion}
We consider $\Bbbk\mathfrak{S}_9$ for   $\Bbbk$    a field of characteristic $ 5$; 
 the characteristic  is unimportant now, but inspires the notation and   will be needed when we 
 refer back to this example later.
We set $\al=\eps_3-\eps_1\in \Pi $.   
Here we have    $$\SSTP_\emp= (\eps_1,\eps_2,\eps_3,\eps_1,\eps_2,\eps_3,\eps_1,\eps_2,\eps_3)
\quad  \text{  
and }
\quad \reflectpath  =(\eps_1,\eps_2,\eps_1,\eps_2,\eps_1,\eps_2,\eps_3,\eps_3,\eps_3) $$
are two examples of paths of shape $(3^3$).  
The unique step-preserving permutation of minimal length  is given by 
\begin{equation} 
w_{\SSTP_\emp} 
^{\reflectpath }
= \begin{minipage}{6cm}
 \begin{tikzpicture}
  [xscale=0.6,yscale=  
 0.6]

  \draw  (15.55-2 ,4.5)  node  {$ \SSTP_{{\al}}^\flat$};

     \draw  (15.55-2 ,3) node {$  \SSTP_\emp  $};
     \draw(3,3) rectangle (12,4.5);
        \foreach \i in {3.5,4.5,...,11.5}
  {
   \fill(\i,3) circle(1.5pt) coordinate (a\i);
          \fill(\i,4.5)circle(1.5pt)  coordinate (d\i);
    }

\draw(5.5,3) --(9.5,4.5);
\draw (4+4.5,3) --(10.5,4.5);

\draw(3.5,3) --(3.5,4.5);
\draw(4.5,3) --(4.5,4.5);
\draw(6.5,3)  --(5.5,4.5);

\draw(3+4.5,3) --(6.5,4.5);

\draw(11.5,3) --(11.5,4.5);

 \draw (4+5.5,3)  --(4+3.5,4.5);
\draw(4+6.5,3) --(4+4.5,4.5);

 \scalefont{0.9}
\draw (3.5,4.5) node[above] {$\eps_1$};
\draw (4.5,4.5) node[above] {$\eps_2$};
\draw (5.5,4.5) node[above] {$\eps_1$};
\draw (6.5,4.5) node[above] {$\eps_2$};
\draw (7.5,4.5) node[above] {$\eps_1$};
\draw (8.5,4.5) node[above] {$\eps_2$};
\draw (9.5,4.5) node[above] {$\eps_3$};
\draw (10.5,4.5) node[above] {$\eps_3$};
\draw (11.5,4.5) node[above] {$\eps_3$};

\draw (3.5,2.5) node  {$\eps_1$};
\draw (4.5,2.5) node  {$\eps_2$};
\draw (5.5,2.5) node  {$\eps_3$};
\draw (6.5,2.5) node  {$\eps_1$};
\draw (7.5,2.5) node  {$\eps_2$};
\draw (8.5,2.5) node  {$\eps_3$};
\draw (9.5,2.5) node  {$\eps_1$};
\draw (10.5,2.5) node  {$\eps_2$};
\draw (11.5,2.5) node  {$\eps_3$};

  \end{tikzpicture} 
    \end{minipage} \end{equation}
  Notice that if two strands have the same step-label, then they do not cross.  This is, of course, exactly what it means for a step-preserving permutation to be  of minimal length. \end{eg}

\begin{defn} \label{choices!!}\label{firstupsil} \color{black}    Fix $(\SSTS,\SSTT)$ an ordered pair of paths which both terminate at some point $\la\in \mathscr{P}_{\underline{\tau}}(n)$.  
We now inductively construct a reduced expression for $w^\SSTS_\SSTT$.  
We define the branching coefficients 
$$
d_p({\SSTS},\SSTT)
=w^p_{q}  \qquad \text{ where } q=|\{1\leq i \leq p \mid w^\SSTS_
\SSTT(i) \leq  w^\SSTS_
\SSTT(p)\}| 
  $$ 
 and 
  $$
\Upsilon_p({\SSTS},\SSTT)= 
(-1)^{\sharp\{  p< k \leq q \mid  i_k=i_p \}}
 \psi_{d_p({\SSTS},\SSTT)} 
$$for $1\leq p \leq n$.  
These allow us to fix a distinguished reduced expression, $\w^\SSTS_\SSTT$,   for $w^\SSTS_\SSTT $ as follows, 
$$
\w^\SSTS_\SSTT = d_1(\SSTS,\SSTT)\dots d_n(\SSTS,\SSTT).  
$$ 
   and we set 
$$
   \Upsilon ^\SSTS_\SSTT=e_\SSTS  \Upsilon_{ \w ^\SSTS_\SSTT}e_\SSTT
=e_\SSTS  \Upsilon_1({\SSTS},\SSTT) \Upsilon_2({\SSTS},\SSTT)
\cdots 
 \Upsilon_n({\SSTS},\SSTT)e_\SSTT.
 $$ 
 \end{defn}

\begin{eg}\label{concanetaion2}
We continue with the assumptions of  \cref{concanetaion}.  
 We have that 
 $$
1_{\mathfrak{S}_9}= d_p(\SSTP_\emp,\SSTP^\flat_\al) 
   $$
 for each $p=1,2,3,4,5,6,9$  because $w^{\SSTP_\emp}_{ \SSTP^\flat_\al}(p) \geq i$ for all $1\leq i \leq p$.   
 We have that 
  $$
 d_7(\SSTP_\emp,\SSTP^\flat_\al)=w^7_3
\quad
  d_8(\SSTP_\emp,\SSTP^\flat_\al)=w^8_6
  $$
 and so our reduced word is depicted in \cref{reducedeg}.
 \end{eg}

 We can think of the branching coefficients  as ``one step morphisms"   which allow us to mutate the path
  $\SSTS$ into $\SSTT$ via a series of $n$ steps (as each branching coefficient moves the position of one step in the path) and so  this mutation proceeds via $n+1$ paths  
 $$
 \SSTS=\SSTS_0,\SSTS_1,\SSTS_2,\dots, \SSTS_n=\SSTT
 $$
 see   \cref{reducedeg} for an example.  
  We now lift these branching coefficients to the KLR algebra.

 \begin{figure}[ht!]
   \begin{minipage}{6cm}
 \begin{tikzpicture}
  [xscale=0.6,yscale=  
 0.6]

      \foreach \i in {0,1.5,3,...,12}
 {\draw(0,\i)--++(0:9);}
      \draw(0,0) rectangle (9,9*1.5);
        \foreach \i in {0.5,1.5,...,8.5}
  {
   \fill(\i,0) circle(1.5pt) coordinate (a\i);
          \fill(\i,1.5)circle(1.5pt)  coordinate (d\i);
   \fill(\i,3) circle(1.5pt) coordinate (a\i);
          \fill(\i,4.5)circle(1.5pt)  coordinate (d\i);
   \fill(\i,6) circle(1.5pt) coordinate (a\i);
          \fill(\i,7.5)circle(1.5pt)  coordinate (d\i);
   \fill(\i,9) circle(1.5pt) coordinate (a\i);
          \fill(\i,10.5)circle(1.5pt)  coordinate (d\i);
   \fill(\i,12) circle(1.5pt) coordinate (a\i);
          \fill(\i,13.5)circle(1.5pt)  coordinate (d\i);
      } 
     \foreach \i in {0.5,1.5,...,8.5}
{\draw(\i,9*1.5)--++(-90:6*1.5); }

    \foreach \i in {0.5,1.5,...,8.5}
{\draw(\i,1.5)--++(-90:1*1.5); }

  \foreach \i in {4.5}
  {\draw(0.5,\i)--++(-90:1*1.5); 
  \draw(1.5,\i)--++(-90:1*1.5); 
  \draw(2.5,\i)-- (3.5,\i-1.5); 
    \draw(2.5+1,\i)-- (4.5,\i-1.5); 
      \draw(2.5+1+1,\i)-- (5.5,\i-1.5); 
        \draw(2.5+1+1+1,\i)-- (6.5,\i-1.5); 
                \draw(6.5,\i)-- (2.5,\i-1.5); 
\draw(7.5,\i)--++(-90:1*1.5);                 
\draw(8.5,\i)--++(-90:1*1.5); 
  }

    \foreach \i in {3}
  {\draw(0.5,\i)--++(-90:1*1.5); 
  \draw(1.5,\i)--++(-90:1*1.5); 
  \draw(2.5,\i)--++(-90:1*1.5);  
   \draw(3.5,\i)--++(-90:1*1.5); 
   \draw(4.5,\i)--++(-90:1*1.5); 
         \draw(2.5+1+1+1,\i)-- (6.5,\i-1.5); 
               \draw(2.5+1+1+1+1,\i)-- (6.5+1,\i-1.5); 
                \draw(7.5,\i)-- (5.5,\i-1.5); 
\draw(8.5,\i)--++(-90:1*1.5); 
  }

    \scalefont{0.9}
\draw (-3+3.5,9*1.5) node[above] {$\eps_1$};
\draw (-3+4.5,9*1.5) node[above] {$\eps_2$};
\draw (-3+5.5,9*1.5) node[above] {$\eps_1$};
\draw (-3+6.5,9*1.5) node[above] {$\eps_2$};
\draw (-3+7.5,9*1.5) node[above] {$\eps_1$};
\draw (-3+8.5,9*1.5) node[above] {$\eps_2$};
\draw (-3+9.5,9*1.5) node[above] {$\eps_3$};
\draw (-3+10.5,9*1.5) node[above] {$\eps_3$};
\draw (-3+11.5,9*1.5) node[above] {$\eps_3$};

\draw (-3+3.5,0) node[below] {$\eps_1$};
\draw (-3+4.5,0) node[below] {$\eps_2$};
\draw (-3+5.5,0) node[below] {$\eps_3$};
\draw (-3+6.5,0) node[below] {$\eps_1$};
\draw (-3+7.5,0) node[below] {$\eps_2$};
\draw (-3+8.5,0) node[below] {$\eps_3$};
\draw (-3+9.5,0) node[below] {$\eps_1$};
\draw (-3+10.5,0) node[below] {$\eps_2$};
\draw (-3+11.5,0) node[below] {$\eps_3$};

\draw (10.5,0) node  {$ \SSTS_9 = \SSTP_\emp $};
\draw (10.5,1.5) node  {$\SSTS_8\phantom{\SSTP_\emp=}$};
\draw (10.5,3) node  {$\SSTS_7\phantom{\SSTP_\emp=}$};
\draw (10.5,4.5) node  {$\SSTS_6\phantom{\SSTP_\emp=}$};
\draw (10.5,6) node  {$\SSTS_5\phantom{\SSTP_\emp=}$};
\draw (10.5,7.5) node  {$\SSTS_4\phantom{\SSTP_\emp=}$};
\draw (10.5,9) node  {$\SSTS_3\phantom{\SSTP_\emp=}$};
\draw (10.5,10.5) node  {$\SSTS_2 \phantom{\SSTP_\emp=}$};
\draw (10.5,12) node  {$\SSTS_1\phantom{\SSTP_\emp=}$};
\draw (10.5,13.5) node  {$\SSTS_0={\SSTP_\al^\flat}$};

  \end{tikzpicture} 
    \end{minipage}  
  \caption{The reduced word, $\w_{\SSTP_\emp}^
{\reflectpath } $ (see also \cref{concanetaion,concanetaion2}).  }
 \label{reducedeg}
 \end{figure}


 
 \begin{rmk}
 The  ``sign twist" in \cref{firstupsil} is of no consequence in this paper as we are mostly concerned with constructing generators and bases of quiver Hecke algebras and their truncations.  However, in order to match-up our {\em relations} with those of Elias--Williamson, this sign twist will be necessary and so we introduce it here for the purposes of consistency with \cite{cell4us2}.  
\end{rmk}

Now, let's momentarily restrict our attention to pairs of paths of the form 
$(\SSTS,\SSTT_\la)$.  In this case, the branching coefficients actually come from the ``branching  rule" for restriction along the tower   $\dots \subset\mathscr{H}_{n-1}^\sigma\subset \mathscr{H}_{n}^\sigma\subset\dots $.  To see this, we note that 
$$
\w^\SSTS_{\SSTT_\la} =   w_1({\SSTS},{\SSTT_\la})w_2({\SSTS},{\SSTT_\la})
\dots 
w_n({\SSTS},{\SSTT_\la}) 
$$
where $w_n({\SSTS},{\SSTT_\la})= w^n_{\SSTT_\la (\square)}$ for some removable box $\square \in \Rem(\la)$ and where 
\begin{align}\label{popopo}
 w_1({\SSTS},{\SSTT_\la})w_2({\SSTS},{\SSTT_\la})
\dots 
w_{n-1}({\SSTS},{\SSTT_\la}) 
=
 \w ({\SSTS_{\leq n-1}},{\SSTT_{\la-\square}}) \in \mathfrak{S}_{n-1}\leq 
  \mathfrak{S}_{n}.  
\end{align} 
By \cref{hfsaklhsalhskafhjksdlhjsadahlfdshjksadflhafskhsfajk}, we have that  
 \begin{align} \label{popopo2}
{\bf S}_\Bbbk (\la- \square )\langle 	 \deg (A_r)	\rangle  
\cong  \Bbbk\{ \Upsilon^\SSTS_{\SSTT_\la} \mid \Shape(\SSTS_{\leq n-1}) 
= \la- \square 	\}  .
\end{align} 
Thus the branching coefficients above provide a factorisation of 
the cellular  basis of  
\cref{cellularitybreedscontempt} which is compatible with the restriction rule.  
 

\begin{eg}\label{concanetaion3}

Continuing with   \cref{concanetaion,concanetaion2},
the  lift of the path-morphism to the  KLR  algebra is as follows, 
  $$
     \Upsilon _{\SSTP_\emp}^
{\reflectpath }
=
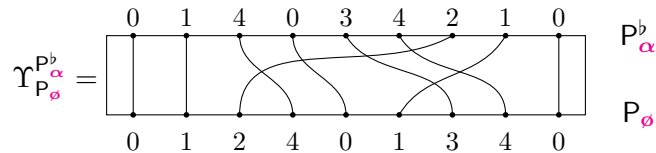
\begin{minipage}{7.6cm} \begin{tikzpicture}
  [xscale=0.7,yscale=  
  0.7]

  \draw  (0.2+14.8-2  ,4.5)  node  {$ \SSTP_{{\al}} ^\flat $};

     \draw  (0.2+14.8-2 ,3) node {$  \SSTP_\emp  $};
     \draw(3,3) rectangle (12,4.5);
        \foreach \i in {3.5,4.5,...,11.5}
  {
   \fill(\i,3) circle(1.5pt) coordinate (a\i);
          \fill(\i,4.5)circle(1.5pt)  coordinate (d\i);
    }

\draw(5.5,3)  to [out=90,in=-150] (9.5,4.5);
\draw (8.5,3)  to [out=60,in=-120] (10.5,4.5);

\draw(3.5,3)  to [out=90,in=-90](3.5,4.5);
\draw(4.5,3)  to [out=90,in=-90](4.5,4.5);
\draw(6.5,3)  to [out=90,in=-90](5.5,4.5);

\draw(3+4.5,3) to [out=90,in=-90] (6.5,4.5);

\draw(11.5,3) to [out=90,in=-90](11.5,4.5);

 \draw (4+5.5,3)  to [out=90,in=-60](4+3.5,4.5);
\draw(4+6.5,3) to [out=90,in=-60](4+4.5,4.5);

 \scalefont{0.9} 

\draw (3.5,4.5) node[above] {$ 0$};
\draw (4.5,4.5) node[above] {$ 1$};
\draw (5.5,4.5) node[above] {$ 4$};
\draw (6.5,4.5) node[above] {$ 0$};
\draw (7.5,4.5) node[above] {$ 3$};
\draw (8.5,4.5) node[above] {$ 4$};
\draw (9.5,4.5) node[above] {$ 2$};
\draw (10.5,4.5) node[above] {$ 1$};
\draw (11.5,4.5) node[above] {$ 0$};

\draw (3.5,2.5) node  {$ 0$};
\draw (4.5,2.5) node  {$ 1$};
\draw (5.5,2.5) node  {$ 2$};
\draw (6.5,2.5) node  {$ 4$};
\draw (7.5,2.5) node  {$ 0$};
\draw (8.5,2.5) node  {$ 1$};
\draw (9.5,2.5) node  {$ 3$};
\draw (10.5,2.5) node  {$ 4$};
\draw (11.5,2.5) node  {$ 0$};

  \end{tikzpicture} 
\end{minipage}.  $$
At each step in the restriction along the tower, there is precisely one removable box of any given residue and so the restriction is, in fact, a direct sum of Specht modules.   
\end{eg}


%
 We   wish to   modify the branching coefficients above so that we can consider more general (families of) reduced paths $\SSTP_\la$ in place of the path $\SSTT_\la$.   
 Given $\SSTS \in   \Path_{\aatchpair  } (\lambda ), $ 
  we can choose a {\sf reduced path vector} as follows  
  $$
 \underline{  \SSTP}_\SSTS
=
({  \SSTP}_{\SSTS,0},{  \SSTP}_{\SSTS,1},\dots ,{  \SSTP}_{\SSTS,n})
  $$
  such that $\Shape({  \SSTP}_{\SSTS,k})=\Shape({  \SSTS} _{\leq k})$ for each $0\leq k \leq n$. 
      In other words, 
  we choose a reduced path  ${  \SSTP}_{\SSTS,p}$ for each and every
  point in the path $\SSTS$.  
For $0\leq p \leq n$ and 
$\Shape(\SSTS_{< p}) + \eps_{i_p} =
    \Shape(\SSTS_{\leq p })$,  we define the {\sf modified branching coefficient}, 
$$
  d_p(\SSTS, \underline{\SSTP}_\SSTS)
  =
  \Upsilon^{\SSTP_{\SSTS,p-1}\boxtimes \SSTP_{i_p}}
  _{\SSTP_{\SSTS,p}}  
    $$ 
    and we hence define  $$
     \Upsilon^\SSTS_ {\underline{\SSTP}_{\SSTS }} = \prod_{1\leq p \leq n}d_p(\SSTS, \underline{\SSTP}_\SSTS).
    $$
Here we have freely identified elements of algebras of different sizes using the usual embedding $\mathscr{H}_{n-1}^\sigma \hookrightarrow \mathscr{H}_{n}^\sigma$ given by $d\mapsto d \boxtimes (\sum_{i\in \ZZ/e\ZZ}  e(i)		)$.  
 We set $\Upsilon_\SSTS^ {\underline{\SSTP}_{\SSTS }}:=(\Upsilon^\SSTS_ {\underline{\SSTP}_{\SSTS }})^\ast$.  

\begin{rmk}
For symmetric groups there is a canonical choice of reduced path vector coming from the coset-like combinatorics which has historically been used for studying   these groups.  
For the light leaves construction of Bott--Samelson endomorphism algebras,  Libedinsky and Elias--Williamson   require very different families of reduced path vectors 
 whose origin can be seen as coming from a  basis which can be written in terms of their  2-generators \cite{MR2441994,MR3555156}.  
\end{rmk}

\begin{eg}
Continuing with \cref{concanetaion3} we have already noted that $\SSTP_\emp= \SSTT_{(3^3)}$.  
We choose to take  the sequence $ \SSTT_\mu$ for $\mu=\Shape((\SSTP^\flat_\al)_{\leq k})$ for $k\geq 0$ as our  reduced path vector $ {\underline{\SSTP}_{\SSTS }} $.   
Having made this choice, we have  
 that $     \Upsilon ^{\SSTP_\al ^\flat}_
 {\underline{\SSTP}_{\SSTS }}  =      \Upsilon _{\SSTP_\emp}^
{\reflectpath }$ (this holds more generally, see the \cref{cellularitybreedscontempt22222233232328787} and the discussion immediately prior).   
We record this in tableaux format to help the reader transition between the old and new ways of thinking.  
$$
\Bigg(\Yvcentermath1 
\varnothing\; , \;
\gyoung(1) \; , \; \gyoung(1;2) \; , \; \gyoung(1;2,3) \; , \; \gyoung(1;2,3;4) \; , \; \gyoung(1;2,3;4,5) \; , \; \gyoung(1;2,3;4,5;6) \; , \; \gyoung(1;2;3,4;5,6;7)
\; , \; \gyoung(1;2;3,4;5;6,7;8)  
\; , \; \gyoung(1;2;3,4;5;6,7;8;9) \;\;
\Bigg)  
$$
\end{eg}

The light leaves basis will be given in terms of products 
 $    \Upsilon^\SSTS_ {\underline{  \SSTP}_\SSTS}
            \Upsilon_\SSTT^ {\underline{  \SSTP}_\SSTT}
         $ 
         for 
        $  \SSTS,\SSTT \in \Path_{\aatchpair  } (\lambda )$ and ``compatible choices" of 
           ${\underline{  \SSTP}_\SSTS}$ and      ${\underline{  \SSTP}_\SSTT}$.    
  Here the only condition for compatibility is that $\SSTP_{\SSTS,n}=\SSTQ_\la=\SSTP_{\SSTT,n}$ for some fixed choice of reduced path $\SSTQ_\la \in \Path_{\aatchpair  } (\lambda )$, in other words the final choices of reduced path for each of $\SSTS$ and $\SSTT$ coincide. 
  We remark that if           $\SSTP_{\SSTS,n}\neq \SSTP_{\SSTT,n}$ then the product is clearly equal to zero (by idempotent considerations) and so this is the only sensible choice to make for such a  product.  
  In light of the above, we
  let $\SSTQ_\la$ be a reduced path and we 
   say that  a reduced path vector  ${\underline{  \SSTP}_\SSTS}$ {\sf terminates} at  $\SSTQ_\la$ if   ${\underline{  \SSTP}_{\SSTS,n}}=\SSTQ_\la$.

   \begin{thm}[The light leaves basis] \label{cellularitybreedscontempt2222223323232} 
   Let $\Bbbk$ be an integral domain and    $\underline{\tau}= (\tau_0,\dots,\tau_{\ell-1})  \in \NN^\ell$
      be such that  $\tau_m\leq   \sigma_{m+1}-\sigma_{m} $ for $0\leq m< \ell-1$ and  $\tau_{\ell-1}<  e+\sigma_0-\sigma_{\ell-1}  $.  
For each $\la\in \mathscr{P}_{\underline{\tau}}(n)$ we fix a reduced path $\SSTQ_\la\in \Path_{\aatchpair  }(\la)$ and 
for each $\SSTS \in \Path_{\aatchpair  }(\la)$, 
 	we fix an associated
 reduced path vector 	 ${\underline{  \SSTP}_\SSTS}$ terminating with $\SSTQ_\la$.  
  The $\Bbbk$-algebra  $   \mathscr{H}^\sigma_n$ is a     graded cellular algebra with basis  
  $$
     \{    \Upsilon^\SSTS_ {\underline{  \SSTP}_\SSTS}
           \Upsilon_\SSTT^ {\underline{  \SSTP}_\SSTT}
          \mid 
           \SSTS,\SSTT \in \Path_{\aatchpair  } (\lambda ), 
 \lambda  \in \mathscr{P}_ {\aatchpair}(n) \} 
          $$
           anti-involution $\ast$ and the degree function ${\sf deg}:\Path_{\aatchpair  } \to\Bbbk$.  
  \end{thm}
\begin{proof}By \cref{cellularitybreedscontempt}, we have that 
$$
\{ \Upsilon^\SSTS_{\SSTT_\la}     \Upsilon_\SSTT^{\SSTT_\la}	  \mid \SSTS,\SSTT \in \Path_{\aatchpair  }(\la)\}
$$
provides a  $\Bbbk$-basis of $\mathscr{H}^{\succeq \la}/\mathscr{H}^{\succ  \la}$.  By \cref{alemmmerd}, we have that $ \Upsilon^{\SSTT_\la} _{\SSTQ_\la} e_{\SSTQ_\la } \Upsilon_{\SSTT_\la} ^{\SSTQ_\la} = k e_{\SSTT_\la }$ 
for  some $k\in \Bbbk \setminus \{0\}$ modulo   higher terms under $\succ$ and so 
$$
\{ \Upsilon^\SSTS_{\SSTT_\la}( \Upsilon^{\SSTT_\la} _{\SSTQ_\la} 
  \Upsilon_{\SSTT_\la} ^{\SSTQ_\la}  	 ) \Upsilon_\SSTT^{\SSTT_\la}	  \mid \SSTS,\SSTT \in \Path_{\aatchpair  }(\la)\}
$$provides a $\Bbbk$-basis of $\mathscr{H}^{\succeq \la}/\mathscr{H}^{\succ  \la}$.   By \cref{popopo,popopo2}, we have that 
\begin{equation}\label{bas1}
  \{
 \Upsilon^{\sf s}_{\SSTT_{\la-{\eps_i}} }  \Upsilon^{\SSTT_{\la-\eps_i}\boxtimes\SSTP_i }_{\SSTT_\la } 
  ( \Upsilon^{\SSTT_\la} _{\SSTQ_\la} 
  \Upsilon_{\SSTT_\la} ^{\SSTQ_\la}  	 ) \Upsilon_\SSTT^{\SSTT_\la}	 
  \mid 
 {\sf s} \in \Std(\la-{\eps_i}),
  {\eps_i} \in \Rem(\la),\SSTT \in \Path_{\aatchpair  }(\la)\}   
\end{equation}
provides a $\Bbbk$-basis of $\mathscr{H}^{\succeq \la}/\mathscr{H}^{\succ  \la}$.  
By \cref{hfsaklhsalhskafhjksdlhjsadahlfdshjksadflhafskhsfajk}, we have that 
 $$    \Upsilon^{\SSTT_{\la-\eps_i}\boxtimes\SSTP_i }_{\SSTT_\la } 
  e_{\SSTT_\la 	}  $$ generates a left subquotient  of $\mathscr{H}^{\succeq \la}/\mathscr{H}^{\succ  \la}$ which is isomorphic to 
  ${\bf S}_\Bbbk(\la-{\eps_i})$.  
   Now, for each pair 
   $\eps_i \in {\rm Rem}(\la)$  and ${\sf s} \in \Std(\la-\eps_i)$, we fix a corresponding choice of reduced path ${\sf P}_{\SSTS,n-1} \in \Path_{\aatchpair  }(\la-{\eps_i})$.  
 By \cref{alemmmerd,hfsaklhsalhskafhjksdlhjsadahlfdshjksadflhafskhsfajk}, we have that the set of all 
\begin{equation*} 
 \Upsilon^{{\sf s } }_{\SSTT_{\la-\eps_i} } 
( \Upsilon^ {\SSTT_{\la-\eps_i}\boxtimes \SSTP_i} _ {{{\sf P}_{\SSTS,n-1}\boxtimes \SSTP_i}} 
 \Upsilon_ {\SSTT_{\la-\eps_i}\boxtimes \SSTP_i} ^ {{{\sf P}_{\SSTS,n-1}\boxtimes \SSTP_i}} )
 \Upsilon^{\SSTT_{\la-\eps_i}\boxtimes\SSTP_i }_{\SSTT_\la } 
  ( \Upsilon^{\SSTT_\la} _{\SSTQ_\la}  
  \Upsilon_{\SSTT_\la} ^{\SSTQ_\la}  	 ) \Upsilon_\SSTT^{\SSTT_\la} 
\end{equation*}
as we vary over all 
$
 {\sf s} \in \Path_{\aatchpair  }(\la-{\eps_i}),
  {\eps_i} \in \Rem(\la),
  $ and $ \SSTT \in \Path_{\aatchpair  }(\la) $
provides a $\Bbbk$-basis of   $\mathscr{H}^{\succeq \la}/\mathscr{H}^{\succ  \la}$.   Re-bracketing the above, we have that the set of all 
\begin{equation*} 
( \Upsilon^{{\sf s} }_{\SSTT_{\la-\eps_i} } 
  \Upsilon^ {\SSTT_{\la-\eps_i} } _ {{{\sf P}_{\SSTS,n-1} }} )(
 \Upsilon_ {\SSTT_{\la-\eps_i}\boxtimes \SSTP_i} ^ {{{\sf P}_{\SSTS,n-1}\boxtimes \SSTP_i}}  
 \Upsilon^{\SSTT_{\la-\eps_i}\boxtimes\SSTP_i }_{\SSTT_\la } 
    \Upsilon^{\SSTT_\la} _{\SSTQ_\la} ) 
 ( \Upsilon_{\SSTT_\la} ^{\SSTQ_\la}  	   \Upsilon_\SSTT^{\SSTT_\la} )
\end{equation*}as we vary over all 
$ {{\sf s}} \in \Path_{\aatchpair  }(\la-{\eps_i}),
  {\eps_i} \in \Rem(\la),$ and $ 
  \SSTT \in \Path_{\aatchpair  }(\la)  $ provides a $\Bbbk$-basis of   $\mathscr{H}^{\succeq \la}/\mathscr{H}^{\succ  \la}$. Finally, simplifying using \cref{move a dot down} 
 we obtain that 
\begin{equation*} 
 \{
  \Upsilon^{{\sf s} } 
 _ {{{\sf P}_{\SSTS,n-1} }}  
 \Upsilon ^ {{{\sf P}_{\SSTS,n-1}\boxtimes \SSTP_i}}  
  _{\SSTQ_\la} 
   \Upsilon ^{\SSTQ_\la}  	 _\SSTT  \mid 
 {{\sf s}\boxtimes \SSTP_i} \in \Path_{\aatchpair  }(\la-{\eps_i}),
  {\eps_i} \in \Rem(\la),
  \SSTT \in \Path_{\aatchpair  }(\la)\}   
\end{equation*}is a $\Bbbk$-basis of   $\mathscr{H}^{\succeq \la}/\mathscr{H}^{\succ  \la}$ where we note that the middle term in the KLR-product is our modified branching coefficient. 
 Repeating $n$ times, we have that 
\begin{equation*} 
 \{
 \Upsilon ^ {{{\sf P}_{\SSTS,0}\boxtimes \SSTP_{i_1}}}  _{{\sf P}_{\SSTS,1}}
\dots 
 \Upsilon ^ {{{\sf P}_{\SSTS,n-2}\boxtimes \SSTP_{i_{n-1}}}}  _{{\sf P}_{\SSTS,n-1}}
 \Upsilon ^ {{{\sf P}_{\SSTS,n-1}\boxtimes \SSTP_{i_n}		}}  
  _{\SSTQ_\la} 
   \Upsilon ^{\SSTQ_\la}  	 _\SSTT  \mid 
 \SSTS , 
  \SSTT \in \Path_{\aatchpair  }(\la)\}  
\end{equation*}is a $\Bbbk$-basis of   $\mathscr{H}^{\succeq \la}/\mathscr{H}^{\succ  \la}$; repeating the above for the righthand-side, the result follows.  
\end{proof} 
 
 In particular, we can set $ \underline{  \SSTP}_\SSTS = (\SSTQ_\la{\downarrow}_{\leq k})_{k\geq 0}$ and obtain the following corollary, which specialises to \cref{cellularitybreedscontempt} for $\SSTQ_\la=\SSTT_\la$.  
 
   \begin{cor}  \label{cellularitybreedscontempt22222233232328787} 
For each $\la\in \mathscr{P}_{\underline{\tau}}(n)$ we fix a reduced path $\SSTQ_\la\in \Path_{\aatchpair  }(\la)$.    The $\Bbbk$-algebra  $   \mathscr{H}^\sigma_n$ is a     graded cellular algebra with basis  
  $$
     \{    \Upsilon^\SSTS_ {\SSTQ_\la}
           \Upsilon_\SSTT^ {\SSTQ_\la}
          \mid 
           \SSTS,\SSTT \in \Path_{\aatchpair  } (\lambda ), 
 \lambda  \in \mathscr{P}_ {\aatchpair}(n) \} 
          $$
           anti-involution $\ast$ and the degree function ${\sf deg}:\Path_{\aatchpair  } \to\ZZ$.  
  \end{cor}
     
\section{Light leaf generators for the principal block}\label{geners}

We now restrict our attention to the principal block and illustrate how the constructions of previous sections specialise to be familiar ideas from Soergel diagrammatics.  
In particular, we provide an exact analogue of
 Libedinsky's and Elias--Williamson's algorithmic construction of a light leaves   basis  for such blocks.  
In order to do this, we provide a short list of path-morphisms 
 which we will show generate the algebra ${\sf f}_{n,\sigma }(\mathcal{H}^\sigma_n /\mathcal{H}^\sigma_n {\sf y}_{\aatchpair}\mathcal{H}^\sigma_n) {\sf f}_{n,\sigma } $ (thus proving Theorem B).  

\subsection{Alcove paths }
When passing from multicompositions to our geometry $\overline{\mathbb E}_{h,l}$, many non-trivial elements map to the origin. One such element is $\delta=((h_1),\mydots,(h_\ell))\in{\mathscr P}_{\aatchpair  } (\aatch  )$.
 (Recall our transpose convention for embedding multipartitions into our geometry.) We will sometimes refer to this as the {\sf determinant} as (for the symmetric group) it corresponds to the determinant representation of the associated general linear group. We will also need to consider elements corresponding to powers of the determinant, namely $\delta_k=((h_1^k),\mydots,(h_\ell^k))\in{\mathscr P}_{\aatchpair  }(k\aatch  )$. We now restrict our attention to paths between points in the principal linkage class, in other words to paths between points in $\Shl\cdot 0$. Such points can be represented by multicompositions $\mu$ in $\Shl\cdot\delta_k$ for some choice of $k$.

\begin{defn}\label{alphabet} We will associate alcove paths to certain words in the {\sf alphabet}
	$$
	S\cup\{1\}=\{ s_\alpha \mid \alpha \in \Pi \cup \{\emptyset\}\}
	$$
	where $s_\emptyset =1$.  That is, we will consider words in the generators of the affine Weyl group, but enriched with explicit occurrences of the identity 
	in these expressions.  We refer to the number of elements in such an expression (including the occurrences of the identity) as the {\sf degree} of this expression. 
We say that an enriched word is reduced if, upon forgetting occurrences of the identity in the expression, the resulting  word   is reduced.  	
	 \end{defn}

Given a path $\SSTP$ between points in the principal linkage class, the end point lies in the interior of an alcove of the form $wA_0$ for some $w\in\widehat{\mathfrak S}_{\aatch  }$. If we write $w$ as a word in our alphabet, and then replace each element $s_{\alpha}$ by the corresponding non-affine reflection $s_{\alpha}$ in ${\mathfrak S}_{\aatch  }$ to form the element $\overline{w}\in{\mathfrak S}_{\aatch  }$ then the basis vectors $\varepsilon_i$ are permuted by the corresponding action of $\overline{w}$ to give $\varepsilon_{\overline{w}(i)}$, and there is an isomorphism from $\overline{\mathbb E}_{h,l}$ to itself which maps $A_0$ to $wA_0$ such that $0$ maps to $w\cdot 0$, coloured walls map to walls of the same colour, and each basis element $\varepsilon_i$ map to $\varepsilon_{\overline{w}(i)}$. Under this map we can transform a path $\SSTQ$ starting at the origin to a path starting at $w\cdot 0$ which passes through the same sequence of coloured walls as $\SSTQ$ does.

\begin{defn}
	Given two paths 
	 $\SSTP=(\eps_{i_1},\eps_{i_2},\dots, \eps_{i_p}) \in \Path(\mu)$
and $	\SSTQ=(\eps_{j_1},\eps_{j_2},\dots, \eps_{j_q}) 
	\in \Path(\nu)$   with the endpoint of $\SSTP$ lying in the closure of some alcove $wA_0$
	we define the {\sf contextualised concatenated path} 
	$$\SSTP\otimes_w \SSTQ =
	(\eps_{i_1},\eps_{i_2},\dots, \eps_{i_p})\boxtimes  
	(\eps_{\overline{w}(j_1)},\eps_{\overline{w}(j_2) },\dots, \eps_{\overline{w}(j_q) }) 
	\in \Path (\mu+(w \cdot \nu)).$$ 
	If there is a unique such $w$ then we may simply write $\SSTP\otimes \SSTQ$.
	If $w=s_\al $ we will simply write $\SSTP\otimes_\al\SSTQ$.    \end{defn}
%

We now define the building blocks from which all of our distinguished paths will be constructed.  We begin by defining certain integers that describe the position of the origin in our fundamental alcove.

\begin{defn} Given ${\al}\in \Pi$ we define 
	$\exx$ to be the distance from the origin to the wall corresponding to 
	${\al}$, and let $b_\emptyset =1$.    
	Given our earlier conventions this corresponds to setting
	$$
	b_{\varepsilon_{H_m + i }-\varepsilon_{H_m+i+1 }} = 1
	$$
	for $1\leq i <h_{m+1}$ and $0 \leq m <\ell$ and that 
	$$
	b_{\varepsilon_{H_m }-\varepsilon_{ H_m+1 }} = \sigma_{m+1}-\sigma_m -h_m+1 
	\qquad
	b_{\varepsilon_{\aatch  }-\varepsilon_{1}} = e+ \sigma_{0}-\sigma_{\ell-1} -h_{\ell-1}+1 
	$$
	for $0\leq m<\ell-1$.  
 	We sometimes write $\delta_{\al}$ for the element $\delta_{\exx}$. 
	Given ${\al},\bet\in \Pi$  we set $b_{\al\bet}=b_\al + b_\bet$.   
\end{defn}

\begin{eg}\label{diag22point4}
	Let $e=5$, $h=3$ and $\ell=1$ as in \cref{diag2}.  
	Then $b_{\color{darkgreen}\varepsilon_2-\varepsilon_3}$ and 
	$b_{\color{cyan}\varepsilon_1-\varepsilon_2}$ both equal $1$, while $b_{\color{magenta}\varepsilon_3-\varepsilon_1}=3$ and $b_{\emptyset}=1$.  
\end{eg}

\begin{eg}
	Let $e=7$, $h=2$ and $\ell=2$ and $\sigma=(0,3)\in\ZZ^2$.  
	Then  
	$b_{\color{cyan}\varepsilon_1-\varepsilon_2}$ and $b_{\color{magenta}\varepsilon_3-\varepsilon_4}$ both equal $1$, while 
	$b_{\color{darkgreen}\varepsilon_4-\varepsilon_1}=3$, $b_{\color{orange}\varepsilon_2-\varepsilon_3}=2$, and $b_{\emptyset}=1$.  
\end{eg}

We can now define our basic building blocks for paths.

\begin{defn} \label{base} Given ${\al}  =\varepsilon_i -\varepsilon_{i+1}\in \Pi$, we consider the multicomposition 
	$ 
	s_\al\cdot \delta_{\al} 
	$ 
	with all columns of length $\exx$, with the exception of the $i$th and $(i+1)$st columns, which are of length  $0$ and $2\exx$, respectively.  
	We set 
	$$
	\REMOVETHESE {i} {} =  (\varepsilon_{1 }, \mydots, \varepsilon_{i-1 }
	, \widehat{\varepsilon_{ i  }},   \varepsilon_{ i+1  },\mydots, \varepsilon_{\aatch   })
	\quad \text{and}\quad 
	\ADDTHIS {i} {} =(+\eps_i)$$
	where $\widehat{.}$ denotes omission of a coordinate.
	Then our distinguished path corresponding to $s_\al$ is given by
	$$ \SSTP_\al=
 	\REMOVETHESE{i} {\exx} \boxtimes \ADDTHIS {i+1}{\exx} \in \Path(s_\al\cdot\delta _{\al}).
	$$
	The distinguished path corresponding to $\emptyset$ is given by
	$$
	\SSTP_{\emptyset}=  (\varepsilon_{1 }, \mydots, \varepsilon_{i-1 }
	,\varepsilon_{ i  } ,   \varepsilon_{ i+1  },\mydots, \varepsilon_{\aatch   })
	\in \Path(\delta)=\Path(s_\emptyset\cdot\delta) 
	$$ 
and set $\SSTP_\emp=(\SSTP_\emptyset)^{\exx}$.   
\end{defn}

Given all of the above, we can finally define our distinguished paths for general words in our alphabet. There will be one such path for each word in our alphabet, and they will be defined by induction on the degree of the word, as follows.

\begin{defn}\label{thepathwewatn}
	We now define a {\sf distinguished path} $\SSTP_{\w}$ for each word $\w$ in our alphabet $S\cup\{1\}$ by induction on the degree of $\w$. If $\w$ is $s_\emptyset$ or a simple reflection $s_\alpha$ we have already defined the distinguished path in Definition \ref{base}. Otherwise if $\w=s_\alpha\w'$ then we define  
	$$
	\SSTP_{\underline{w}}:= \SSTP_{\al }
	\;\otimes_{\al}  \; \SSTP_{\underline{w}'}	.$$ 
  \end{defn}

	If $\w$ is a reduced word in $\widehat{\mathfrak S}_{hl}$,  then   the corresponding  
	 path $\SSTP_{\w}$ is a  reduced path.

\begin{rmk}
	Contextualised concatenation is not associative (if we wish to decorate the tensor products with the corresponding elements $w$). As we will typically be constructing paths as in Definition \ref{thepathwewatn} we will adopt the convention that an unbracketed concatenation of $n$ terms corresponds to bracketing from the right: 
	$$\SSTQ_1\otimes\SSTQ_2\otimes \SSTQ_3\otimes\cdots \SSTQ_n=\SSTQ_1\otimes(\SSTQ_2\otimes(\SSTQ_3\otimes(\cdots \otimes \SSTQ_n)\cdots)).$$
\end{rmk}

We will also need certain reflections of our distinguished paths corresponding to elements of $\Pi$. 

\begin{defn}
	{ \renewcommand{\SSTU}{{\sf P}}\renewcommand{\mu}{{\pi}} 
		Given ${\al} \in \Pi$ we set 
		$$
		\reflectpath =
		\REMOVETHESE {i} {b_\al} \boxtimes \ADDTHIS {i} {\exx} =
		\REMOVETHESE {i} {b_\al} \otimes_\al \ADDTHIS {i+1} {\exx} =
		(+\varepsilon_{1 }, \mydots, +\varepsilon_{i-1 }
		, \widehat{+\varepsilon_{ i  }},   +\varepsilon_{ i+1  },\mydots, +\varepsilon_{\aatch   })^\exx 
		\boxtimes  (\eps_i)^{\exx}
		$$
		the path obtained by reflecting the second part of $\SSTP_\al$ in the wall through which it passes.  
	} 
	
\end{defn}

\begin{eg}
We illustrate these various constructions in a series of examples. In the first two diagrams of Figure \ref{figure1}, we illustrate the basic path $\SSTP_{\al}$ and the path $\reflectpath$ 
 and in the rightmost diagram of Figure \ref{figure1},  we illustrate the path $\SSTP_\emptyset$. A more complicated example is illustrated in Figure \ref{diag2}, where we show the distinguished path $\SSTP_{\underline{w}}$ for $\underline{w}=s_{\color{magenta}\varepsilon_3-\varepsilon_1}s_{\color{darkgreen}\varepsilon_2-\varepsilon_1}
	s_{\color{cyan}\varepsilon_3-\varepsilon_2}s_{\color{magenta}\varepsilon_3-\varepsilon_1}
	s_{\color{darkgreen}\varepsilon_2-\varepsilon_1}s_{\color{cyan}\varepsilon_3-\varepsilon_2}$ as in   \cref{diag2}. The components of the path between consecutive black nodes correspond to individual $\SSTP_{\al}$s.   \end{eg}

%
   
\begin{figure}[ht!]

$$
  \begin{minipage}{2.6cm}\begin{tikzpicture}[scale=0.75] 
    
    \path(0,0) coordinate (origin);
          \foreach \i in {0,1,2,3,4,5}
  {
    \path  (origin)++(60:0.6*\i cm)  coordinate (a\i);
    \path (origin)++(0:0.6*\i cm)  coordinate (b\i);
     \path (origin)++(-60:0.6*\i cm)  coordinate (c\i);
    }
  
      \path(3,0) coordinate (origin);
          \foreach \i in {0,1,2,3,4,5}
  {
    \path  (origin)++(120:0.6*\i cm)  coordinate (d\i);
    \path (origin)++(180:0.6*\i cm)  coordinate (e\i);
     \path (origin)++(-120:0.6*\i cm)  coordinate (f\i);
    }

  \foreach \i in {0,1,2,3,4,5}
  {
    \draw[gray, densely dotted] (a\i)--(b\i);
        \draw[gray, densely dotted] (c\i)--(b\i);
    \draw[gray, densely dotted] (d\i)--(e\i);
        \draw[gray, densely dotted] (f\i)--(e\i);
            \draw[gray, densely dotted] (a\i)--(d\i);
                \draw[gray, densely dotted] (c\i)--(f\i);
 
     }
  \draw[very thick, magenta] (0,0)--++(0:3) ;
    \draw[very thick, cyan] (3,0)--++(120:3) coordinate (hi);
        \draw[very thick, darkgreen] (hi)--++(-120:3) coordinate  (hi);

        \draw[very thick, darkgreen] (hi)--++(-60:3) coordinate  (hi);

    \draw[very thick, cyan] (hi)--++(60:3) coordinate (hi);

      \path(0,0)--++(-60:5*0.6)--++(120:2*0.6)--++(0:0.6) coordinate (hi);
     \path(0,0)--++(0:4*0.6)--++(120:3*0.6)           coordinate(hi2) ; 
     \path(0,0)--++(0:4*0.6)          coordinate(hi3) ;

          \path(hi)  --++(120:0.6)
           coordinate(step1) 
          --++(0:0.6)
                     coordinate(step2) 
      --++(120:0.6)
                 coordinate(step3) 
                 --++(0:0.6)           coordinate(step4) 
      --++(120:0.6) 
                 coordinate(step5)  --++(0:0.6) 
                 coordinate(step6)   ;    
     
                \path(0,0)--++(0:4*0.6)--++(120:2*0.6)           coordinate(hin) ; 

                     \path(0,0)--++(0:4*0.6)--++(120:2.85*0.6)           coordinate(hi4) ; 

        \draw[ thick,->]    (hi)  to    (step1) 
          to [out=0,in=-180]
                              (step2) 
 to   
            (step3) 
                     to [out=0,in=-180]       (step4) 
  to   
             (step5)       to [out=0,in=-180]   (step6) 
              (step6) 
              to [out=100,in=-35] (hi4)  ;

 \end{tikzpicture}\end{minipage}
\qquad
   \begin{minipage}{2.6cm}\begin{tikzpicture}[scale=0.75] 
    
    \path(0,0) coordinate (origin);
          \foreach \i in {0,1,2,3,4,5}
  {
    \path  (origin)++(60:0.6*\i cm)  coordinate (a\i);
    \path (origin)++(0:0.6*\i cm)  coordinate (b\i);
     \path (origin)++(-60:0.6*\i cm)  coordinate (c\i);
    }
  
      \path(3,0) coordinate (origin);
          \foreach \i in {0,1,2,3,4,5}
  {
    \path  (origin)++(120:0.6*\i cm)  coordinate (d\i);
    \path (origin)++(180:0.6*\i cm)  coordinate (e\i);
     \path (origin)++(-120:0.6*\i cm)  coordinate (f\i);
    }

  \foreach \i in {0,1,2,3,4,5}
  {
    \draw[gray, densely dotted] (a\i)--(b\i);
        \draw[gray, densely dotted] (c\i)--(b\i);
    \draw[gray, densely dotted] (d\i)--(e\i);
        \draw[gray, densely dotted] (f\i)--(e\i);
            \draw[gray, densely dotted] (a\i)--(d\i);
                \draw[gray, densely dotted] (c\i)--(f\i);
 
     }
  \draw[very thick, magenta] (0,0)--++(0:3) ;
    \draw[very thick, cyan] (3,0)--++(120:3) coordinate (hi);
        \draw[very thick, darkgreen] (hi)--++(-120:3) coordinate  (hi);

        \draw[very thick, darkgreen] (hi)--++(-60:3) coordinate  (hi);

    \draw[very thick, cyan] (hi)--++(60:3) coordinate (hi);

      \path(0,0)--++(-60:5*0.6)--++(120:2*0.6)--++(0:0.6) coordinate (hi);
     \path(0,0)--++(0:4*0.6)--++(120:3*0.6)           coordinate(hi2) ; 
     \path(0,0)--++(0:4*0.6)          coordinate(hi3) ;

          \path(hi)  --++(120:0.6)
           coordinate(step1) 
          --++(0:0.6)
                     coordinate(step2) 
      --++(120:0.6)
                 coordinate(step3) 
                 --++(0:0.6)           coordinate(step4) 
      --++(120:0.6) 
                 coordinate(step5)  --++(0:0.6) 
                 coordinate(step6)   ;

           \path(0,0)--++(0:4*0.6)--++(-120:3*0.6) --++(25:0.1)          coordinate(hin) ;

        \draw[ thick,->]    (hi)  to    (step1) 
          to [out=0,in=-180]
                              (step2) 
 to   
            (step3) 
                     to [out=0,in=-180]       (step4) 
  to   
             (step5)       to [out=0,in=-180]   (step6) to [out=-90,in=45] (hin)  ;

 \end{tikzpicture}\end{minipage} \qquad 
   \begin{minipage}{2.6cm}\begin{tikzpicture}[scale=0.75] 
    
    \path(0,0) coordinate (origin);
          \foreach \i in {0,1,2,3,4,5}
  {
    \path  (origin)++(60:0.6*\i cm)  coordinate (a\i);
    \path (origin)++(0:0.6*\i cm)  coordinate (b\i);
     \path (origin)++(-60:0.6*\i cm)  coordinate (c\i);
    }
  
      \path(3,0) coordinate (origin);
          \foreach \i in {0,1,2,3,4,5}
  {
    \path  (origin)++(120:0.6*\i cm)  coordinate (d\i);
    \path (origin)++(180:0.6*\i cm)  coordinate (e\i);
     \path (origin)++(-120:0.6*\i cm)  coordinate (f\i);
    }

  \foreach \i in {0,1,2,3,4,5}
  {
    \draw[gray, densely dotted] (a\i)--(b\i);
        \draw[gray, densely dotted] (c\i)--(b\i);
    \draw[gray, densely dotted] (d\i)--(e\i);
        \draw[gray, densely dotted] (f\i)--(e\i);
            \draw[gray, densely dotted] (a\i)--(d\i);
                \draw[gray, densely dotted] (c\i)--(f\i);
 
     }
  \draw[very thick, magenta] (0,0)--++(0:3) ;
    \draw[very thick, cyan] (3,0)--++(120:3) coordinate (hi);
        \draw[very thick, darkgreen] (hi)--++(-120:3) coordinate  (hi);

        \draw[very thick, darkgreen] (hi)--++(-60:3) coordinate  (hi);

    \draw[very thick, cyan] (hi)--++(60:3) coordinate (hi);

      \path(0,0)--++(-60:5*0.6)--++(120:2*0.6)--++(0:0.6) coordinate (hi);
     \path(0,0)--++(0:4*0.6)--++(120:3*0.6)           coordinate(hi2) ; 
     \path(0,0)--++(0:4*0.6)          coordinate(hi3) ;

          \path(hi)  --++(120:0.6)
           coordinate(step1) 
          --++(0:0.6)
                     coordinate(step2) 
      --++(120:0.6)
                 coordinate(step3) 
                 --++(0:0.6)           coordinate(step4) 
      --++(120:0.6) 
                 coordinate(step5)  --++(0:0.6) 
                 coordinate(step6)   ;

           \path(0,0)--++(0:4*0.6)--++(-120:3*0.6) --++(25:0.1)          coordinate(hin) ;

        \draw[ thick,->]    (hi)  to    (step1) 
          to [out=0,in=-180]
                              (step2) 
 to [out=-90,in=45] (hin)  ;

 \end{tikzpicture}\end{minipage}  
 $$
\caption{The leftmost two diagrams picture the path $\SSTP_\al$  walking through an ${\al}$-hyperplane in $\overline{\mathbb E}_{1,3}^+$, and the path $\reflectpath $  which reflects this path  through the same ${\al}$-hyperplane.
The rightmost diagram pictures the path $\SSTP_\emptyset$ in $\overline{\mathbb E}_{1,3}^+$.  
We have bent the paths slightly to make them clearer. } \label{figure1}
\end{figure}
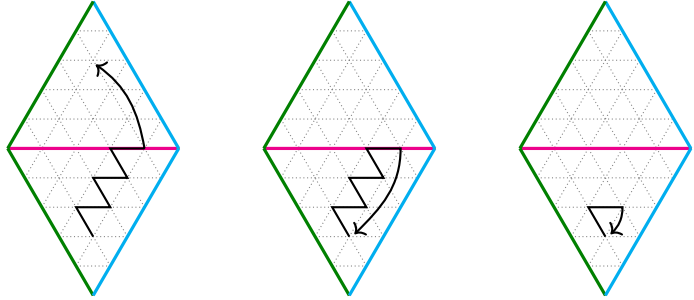

 \!\!\!

 \subsection{The principal block of $\algebra$}    
    We now restrict our attention to regular blocks of $\algebra$.   In order to do this, 
we first recall that we consider an element  of the quiver Hecke algebra  to   be a morphism between paths.  
The easiest elements to construct are the idempotents 
corresponding to the trivial morphism from a  path to itself.      
 Given $\al$ a simple reflection  or $\al=\emptyset$, 
 we have an 
 associated  path 
     $\SSTP_\al $, a trivial bijection $w^{\SSTP_\al} _{\SSTP_\al} =1\in \mathfrak{S}_{b_\al \aatch  }$,  and  an idempotent element of the quiver Hecke algebra
 $$ 
e_{\SSTP_\al }   := e_{\res(\SSTP_\al)}\in \mathcal{H}_{b_\al \aatch  }^\sigma .  
 $$
 More generally, given any $\w=s_{ \alpha^{(1)}}s_{ \alpha^{(2)}}\dots  s_{ \alpha^{(k)} }$, we have an 
 associated  path 
     $\SSTP_\w $, and  an element of the quiver Hecke algebra
 $$ 
e_{\SSTP_\w }   := e_{\res(\SSTP_\w)}=
e_{\SSTP_{\alpha^{(1)}}}
\otimes
e_{\SSTP_{\alpha^{(2)}}}	
\otimes \dots\otimes e_{\SSTP_{\alpha^{(k)}}}	\in \mathscr{H}_{ \aatch  b_{\alpha^{(1)}}  +\dots+ \aatch b_{\alpha^{(k)}}    }^\sigma 
 $$    
 We let $\Std_{n,\sigma} (\la)$ be the set    of all standard  $\la$-tableaux which can be obtained by contextualised concatenation of paths from the set 
$$\{ \SSTP_\al \mid \al \in \Pi  \}\cup\{ \SSTP_\al^\flat  \mid \al \in \Pi  \}\cup \{ \SSTP_\emptyset   \} .$$
We let $\mathscr{P}_{\underline{\tau}}(n,\sigma)=\{ \la\in    \mathscr{P}_{\underline{\tau}}(n) \mid \Std_{n,\sigma} (\la)\neq \emptyset\}$.  
We let $\Std_{n,\sigma} = \cup_{\la\in  \mathscr{P}_{\underline{\tau}}(n ,\sigma)}\Std_{n,\sigma} (\la)$.   
For example, the path in \cref{diag2}   is equal to  $\SSTP_\al \otimes \SSTP_\gam \otimes \SSTP_\bet \otimes 
\SSTP_\al \otimes \SSTP_\gam \otimes \SSTP_\bet.$ 
 We define 
 \begin{equation}\label{redue}
    {\sf f}_{n,\sigma}=\sum_{\begin{subarray}c
   \sts\in \Std_{n,\sigma}(\la)
   \\
   \la  \in \mathscr{P}_{\underline{\tau}}(n,\sigma )
   \end{subarray}} { e}_{\SSTS}
\end{equation}
and the remainder of this paper will be dedicated to understanding the algebra
$$
     {\sf f}_{n,\sigma} (\mathcal{H}^\sigma_n/ \mathcal{H}^\sigma_n {\sf y}_{\underline{\tau} }\mathcal{H}^\sigma_n )      {\sf f}_{n,\sigma}.
  $$  
In fact, we will provide a concise list of generators for this truncated algebra (in the spirit of \cite{MR3555156}) and rewrite the basis of 
\cref{cellularitybreedscontempt2222223323232} in terms of these generators. 

 In this section, we use our concrete   branching coefficients to define the ``Soergel 2-generators" of $ {\sf f}_{n,\sigma} (\mathcal{H}^\sigma_n/ \mathcal{H}^\sigma_n {\sf y}_{\underline{\tau} }\mathcal{H}^\sigma_n )      {\sf f}_{n,\sigma}$  explicitly.  In the companion paper \cite{cell4us2}, we will show that these generators are actually independent of these choices of reduced expressions (however, this won't be needed here --- we simply make a note, again, for purposes of consistency with \cite{cell4us2}).

\subsection{Generator morphisms in degree zero}  
 We first discuss how to pass between   paths $ {\SSTP_\w } $ and $ {\SSTP_{\w'} } $  which are in {\em different linkage classes} but for which   $\w$ and $\w'$  have the same underlying permutation.    
  Fix two such paths  
 $$ 
 {\SSTP_\w }   =
 {\SSTP_{\alpha^{(1)}}}
\otimes
 {\SSTP_{\alpha^{(2)}}}	
\otimes \dots\otimes  {\SSTP_{\alpha^{(k)}}}	\in \Path_{\aatchpair  }(\la) \qquad 
 {\SSTP_{\w'} }   =
 {\SSTP_{\beta^{(1)}}}
\otimes
 {\SSTP_{\beta^{(2)}}}	
\otimes \dots\otimes  {\SSTP_{\beta^{(k)}}}	\in \Path_{\aatchpair  }(\la) $$    
with 
$\alpha^{(1)} ,\dots , \alpha^{(k)} ,
\beta^{(1)} ,\dots , \beta^{(k)} \in \Pi \cup \{\emptyset\}$. 
  We suppose, only for the purposes of this motivational discussion, that both paths are reduced.   
 In which case, we have that $w \in \Shl$ and so the   expressions $\w$ and $\w'$ differ only by applying Coxeter relations in 
of $   \Shl$ and the trivial ``adjustment" relation $s_i 1 = 1 s_i$ (made necessary by our augmentation of the Coxeter presentation).  Moreover,   $\w$ and $\w'$  are both {\em reduced} expressions and so we need only apply the ``hexagon" relation  $s_i s_{i+1} s_i= 
 s_{i+1} s_is_{i+1}$
  and the ``commutation" relation $s_i s_j = s_j s_i$ for $|i-j|>1$.  
  The remainder of this subsection  will be dedicated to lifting  these path-morphisms to the level of generators of  the KLR algebra.  
 We stress that one can apply these adjustment/hexagon/commutator  path-morphisms to any paths (not just reduced paths) but the reduced paths provide the motivation.

\subsubsection{Adjustment generator}
We will refer to the passage between alcove paths which differ only 
     by occurrences of $s_\emptyset=1$  (and their associated idempotents) as ``adjustment". 
We  define the {\sf KLR}-adjustment generator to be the element 
 $$
 {\sf adj}^{\al\emptyset}_{{\emptyset\al}}
 :=
\Upsilon^{   \SSTP_{\al\emptyset} }_{\SSTP_{\emptyset\al}}.
 $$  
Examples of the paths
${   \SSTP_{\al\emptyset} }$, ${\SSTP_{\emptyset\al}}$, and  adjustment generators 
 are given in \cref{adjustmntexamplehidihi2,adjustmntexamplehidihi}.

 \begin{figure}[ht!] 
  $$  
 \begin{tikzpicture}
  [xscale=0.5,yscale=  
  0.5]

\draw[black] (3.5-3,11*1.5) node[above] {$0$};
\draw[black] (4.5-3,11*1.5) node[above] {$1$};
\draw[black] (5.5-3,11*1.5) node[above] {$2$};
\draw[magenta] (3.5,11*1.5) node[above] {$4$};
\draw[magenta] (4.5,11*1.5) node[above] {$0$};
\draw[magenta] (5.5,11*1.5) node[above] {$3$};
 \draw[magenta] (6.5,11*1.5) node[above] {$4$};
\draw[magenta] (7.5,11*1.5) node[above] {$2$};
\draw[magenta] (8.5,11*1.5) node[above] {$3$};
\draw[magenta] (6.5+3,11*1.5) node[above] {$1$};
\draw[magenta] (7.5+3,11*1.5) node[above] {$0$};
\draw[magenta] (8.5+3,11*1.5) node[above] {$4$};

\draw[black] (3.5-3,11*1.5+1) node[above] {$\eps_1$};
\draw[black] (4.5-3,11*1.5+1) node[above] {$\eps_2$};
\draw[black] (5.5-3,11*1.5+1) node[above] {$\eps_3$};
\draw[magenta] (3.5,11*1.5+1) node[above] {$\eps_1$};
\draw[magenta] (4.5,11*1.5+1) node[above] {$\eps_2$};
\draw[magenta] (5.5,11*1.5+1) node[above] {$\eps_1$};
 \draw[magenta] (6.5,11*1.5+1) node[above] {$\eps_2$};
\draw[magenta] (7.5,11*1.5+1) node[above] {$\eps_1$};
\draw[magenta] (8.5,11*1.5+1) node[above] {$\eps_2$};
\draw[magenta] (6.5+3,11*1.5+1) node[above] {$\eps_1$};
\draw[magenta] (7.5+3,11*1.5+1) node[above] {$\eps_1$};
\draw[magenta] (8.5+3,11*1.5+1) node[above] {$\eps_1$};

\draw[magenta] (3.5-3,0+3) node[below] {$0$};
\draw[magenta] (4.5-3,0+3) node[below] {$1$};
\draw[magenta] (5.5-3,0+3) node[below] {$4$};
\draw[magenta] (3.5,0+3) node[below] {$0$};
\draw[magenta] (4.5,0+3) node[below] {$3$};
\draw[magenta] (5.5,0+3) node[below] {$4$};
 \draw[magenta] (6.5,0+3) node[below] {$2$};
\draw[magenta] (7.5,0+3) node[below] {$1$};
\draw[magenta] (8.5,0+3) node[below] {$0$};
\draw[black] (6.5+3,0+3) node[below] {$2$};
\draw[black] (7.5+3,0+3) node[below] {$3$};
\draw[black] (8.5+3,0+3) node[below] {$4$};

\draw[magenta] (3.5-3,-1+3) node[below] {$\eps_1$};
\draw[magenta] (4.5-3,-1+3) node[below] {$\eps_2$};
\draw[magenta] (5.5-3,-1+3) node[below] {$\eps_1$};
\draw[magenta] (3.5,-1+3) node[below] {$\eps_2$};
\draw[magenta] (4.5,-1+3) node[below] {$\eps_ 1$};
\draw[magenta] (5.5,-1+3) node[below] {$\eps_2$};
 \draw[magenta] (6.5,-1+3) node[below] {$\eps_1$};
\draw[magenta] (7.5,-1+3) node[below] {$\eps_1$};
\draw[magenta] (8.5,-1+3) node[below] {$\eps_1$};
\draw[black] (6.5+3,-1+3) node[below] {$\eps_3$};
\draw[black] (7.5+3,-1+3) node[below] {$\eps_2$};
\draw[black] (8.5+3,-1+3) node[below] {$\eps_1$};
 
        \foreach \i in {0.5,1.5,...,11.5}
  {
    \fill(\i,3) circle(1.5pt) coordinate (a\i);
          \fill(\i,4.5)circle(1.5pt)  coordinate (d\i);
   \fill(\i,6) circle(1.5pt) coordinate (a\i);
          \fill(\i,7.5)circle(1.5pt)  coordinate (d\i);
   \fill(\i,9) circle(1.5pt) coordinate (a\i);
          \fill(\i,10.5)circle(1.5pt)  coordinate (d\i);
   \fill(\i,12) circle(1.5pt) coordinate (a\i);
          \fill(\i,13.5)circle(1.5pt)  coordinate (d\i);
                \fill(\i,12+1.5)circle(1.5pt)  coordinate (d\i);
   \fill(\i,12+3) circle(1.5pt) coordinate (a\i);
          \fill(\i,12+4.5)circle(1.5pt)  coordinate (d\i);
      } 

\draw(0,3) rectangle (12,11*1.5);
\clip(0,3) rectangle (12,11*1.5);

     \foreach \i in {0,1.5,3,...,18}
 {\draw(0,\i)--++(0:12);}
      \draw(0,0) rectangle (12,18);

         \foreach \i in {0.5,1.5,2.5,3.5,4.5,5.5,6.5,7.5,8.5,9.5,10.5,11.5}
  {
  \draw(\i,12*1.5)--++(-90:3);
        }

        \foreach \i in {0.5,1.5,4.5,5.5,6.5,7.5,8.5,9.5,10.5,11.5}
  {
  \draw(\i,10*1.5)--++(-90:1.5);
        } 

        \foreach \i in {0.5,1.5,2.5,5.5,6.5,7.5,8.5,9.5,10.5,11.5}
  {
  \draw(\i,9*1.5)--++(-90:1.5);
        }

     \foreach \i in {0.5,1.5,2.5,3.5,6.5,7.5,8.5,9.5,10.5,11.5}
  {
  \draw(\i,8*1.5)--++(-90:1.5);
        }

     \foreach \i in {0.5,1.5,2.5,3.5,4.5,7.5,8.5,9.5,10.5,11.5}
  {
  \draw(\i,7*1.5)--++(-90:1.5);
        }

     \foreach \i in {0.5,1.5,2.5,3.5,4.5,5.5,8.5,9.5,10.5,11.5}
  {
  \draw(\i,6*1.5)--++(-90:1.5);
        } 

    \foreach \i in {0.5,1.5,2.5,3.5,4.5,5.5,6.5} 
  {
  \draw(\i,5*1.5)--++(-90:5*1.5);
        } 
 
     \foreach \i in {0.5,1.5,2.5,3.5,4.5,...,7.5,8.5,9.5,10.5,11.5}
  {
  \draw(\i,1*1.5)--++(-90:1.5);
        }

 \draw(2.5,10*1.5)--(11.5-2,3*1.5)--++(-90:3); 
 \draw(3.5,10*1.5)--(2.5,9*1.5); 
  \draw(3.5+1,9*1.5)--(2.5+1,8*1.5); 
   \draw(3.5+1+1,8*1.5)--(2.5+1+1,7*1.5); 
     \draw(3.5+1+1+1,7*1.5)--(2.5+1+1+1,6*1.5); 
          \draw(3.5+1+1+1+1,6*1.5)--(2.5+1+1+1+1,5*1.5); 
          
                    \draw(3.5+1+1+1+1+2,5*1.5)--(2.5+1+1+1+1+1,4*1.5)--++(-90:6) ;
                    
                    \draw(3.5+1+1+1+1+1,5*1.5)--(3.5+1+1+1+1+1+2,3*1.5)--++(-90:3) ;                     
  
             \draw(3.5+1+1+1+1+2+1,5*1.5)--++(90:-1.5)--(2.5+1+1+1+1+1+1,3*1.5)--++(90:-3); 
          \draw(3.5+1+1+1+1+2+1+1,5*1.5)--++(90:-6);

    \end{tikzpicture} $$  
    
    \!\!\!
    \caption{  We let $h=3$, $\ell=1$, $e=5$ and  $\al=\varepsilon_3-\varepsilon_1$.   
    The adjustment term $  {\sf adj}_{{\al \emptyset}}^{{ \emptyset \al}}
  $ is illustrated.  
      }
\label{adjustmntexamplehidihi}    \end{figure}

\begin{figure}[ht!]
$$  
  \begin{minipage}{2.6cm}\begin{tikzpicture}[scale=0.65] 
    
    \path(0,0) coordinate (origin);
          \foreach \i in {0,1,2,3,4,5}
  {
    \path  (origin)++(60:0.6*\i cm)  coordinate (a\i);
    \path (origin)++(0:0.6*\i cm)  coordinate (b\i);
     \path (origin)++(-60:0.6*\i cm)  coordinate (c\i);
    }
  
      \path(3,0) coordinate (origin);
          \foreach \i in {0,1,2,3,4,5}
  {
    \path  (origin)++(120:0.6*\i cm)  coordinate (d\i);
    \path (origin)++(180:0.6*\i cm)  coordinate (e\i);
     \path (origin)++(-120:0.6*\i cm)  coordinate (f\i);
    }

  \foreach \i in {0,1,2,3,4,5}
  {
    \draw[gray, densely dotted] (a\i)--(b\i);
        \draw[gray, densely dotted] (c\i)--(b\i);
    \draw[gray, densely dotted] (d\i)--(e\i);
        \draw[gray, densely dotted] (f\i)--(e\i);
            \draw[gray, densely dotted] (a\i)--(d\i);
                \draw[gray, densely dotted] (c\i)--(f\i);
 
     }
  \draw[very thick, magenta] (0,0)--++(0:3) ;
    \draw[very thick, cyan] (3,0)--++(120:3) coordinate (hi);
        \draw[very thick, darkgreen] (hi)--++(-120:3) coordinate  (hi);

        \draw[very thick, darkgreen] (hi)--++(-60:3) coordinate  (hi);

    \draw[very thick, cyan] (hi)--++(60:3) coordinate (hi);

      \path(0,0)--++(-60:5*0.6)--++(120:2*0.6)--++(0:0.6) coordinate (hi);
     \path(0,0)--++(0:4*0.6)--++(120:3*0.6)           coordinate(hi2) ; 
     \path(0,0)--++(0:4*0.6)          coordinate(hi3) ;

          \path(hi)  --++(120:0.6)
           coordinate(step1) 
          --++(0:0.6)
                     coordinate(step2) 
      --++(120:0.6)
                 coordinate(step3) 
                 --++(0:0.6)           coordinate(step4) 
      --++(120:0.6) 
                 coordinate(step5)  --++(0:0.6) 
                 coordinate(step6)   ;

           \path(0,0)--++(0:4*0.6)--++(120:3*0.6)           coordinate(hin) ; 

                     \path(0,0)--++(0:4*0.6)--++(120:2.85*0.6)           coordinate(hi4) ; 

        \draw[ thick,->]    (hi)  to [out=90,in=-30]  (step1) 
          to [out=-30,in=-150]
                              (step2) 
 to [out=90,in=-30] 
            (step3) 
                     to [out=-30,in=-150]       (step4) 
  to [out=90,in=-30] 
             (step5)       to [out=-30,in=-150]   (step6) 
              (step6) 
              to [out=90,in=-12.5] (hin)
         to [out=-155,in=100] ++(-120:0.6)  
                  to [out=-30,in=-150] ++(0:0.6)     
                           to [out=100,in=-45] (hi4)   ;

    \end{tikzpicture}\end{minipage}
\quad
  \begin{minipage}{2.6cm}\begin{tikzpicture}[scale=0.65] 
    
    \path(0,0) coordinate (origin);
          \foreach \i in {0,1,2,3,4,5}
  {
    \path  (origin)++(60:0.6*\i cm)  coordinate (a\i);
    \path (origin)++(0:0.6*\i cm)  coordinate (b\i);
     \path (origin)++(-60:0.6*\i cm)  coordinate (c\i);
    }
  
      \path(3,0) coordinate (origin);
          \foreach \i in {0,1,2,3,4,5}
  {
    \path  (origin)++(120:0.6*\i cm)  coordinate (d\i);
    \path (origin)++(180:0.6*\i cm)  coordinate (e\i);
     \path (origin)++(-120:0.6*\i cm)  coordinate (f\i);
    }

  \foreach \i in {0,1,2,3,4,5}
  {
    \draw[gray, densely dotted] (a\i)--(b\i);
        \draw[gray, densely dotted] (c\i)--(b\i);
    \draw[gray, densely dotted] (d\i)--(e\i);
        \draw[gray, densely dotted] (f\i)--(e\i);
            \draw[gray, densely dotted] (a\i)--(d\i);
                \draw[gray, densely dotted] (c\i)--(f\i);
 
     }
  \draw[very thick, magenta] (0,0)--++(0:3) ;
    \draw[very thick, cyan] (3,0)--++(120:3) coordinate (hi);
        \draw[very thick, darkgreen] (hi)--++(-120:3) coordinate  (hi);

        \draw[very thick, darkgreen] (hi)--++(-60:3) coordinate  (hi);

    \draw[very thick, cyan] (hi)--++(60:3) coordinate (hi);

      \path(0,0)--++(-60:5*0.6)--++(120:2*0.6)--++(0:0.6) coordinate (hi);
     \path(0,0)--++(0:4*0.6)--++(120:3*0.6)           coordinate(hi2) ; 
     \path(0,0)--++(0:4*0.6)          coordinate(hi3) ;

          \path(hi)  --++(120:0.6)
           coordinate(step1) 
          --++(0:0.6)
                     coordinate(step2) 
      --++(120:0.6)
                 coordinate(step3) 
                 --++(0:0.6)           coordinate(step4) 
      --++(120:0.6) 
                 coordinate(step5)  --++(0:0.6) 
                 coordinate(step6)   ;

           \path(0,0)--++(0:4*0.6)--++(120:3*0.6)           coordinate(hin) ; 

                     \path(0,0)--++(0:4*0.6)--++(120:2.85*0.6)           coordinate(hi4) ;

                      \path(hi) --++(-90:0.05) coordinate (boo);
    \path(step1) --++(135:0.05) coordinate (boo2);
        \path(step2) --++(135:0.05) coordinate (boo3);
        \draw[ thick  ]    
        (hi)  to [out=90,in=-30]  (step1) 
          to [out=-30,in=-150]  (step2) 
         to  [out=-90,in=30]  (boo)
             to  [out=135,in=-90]  (boo2)
                          to  [out=35,in=155]  (boo3); 
          
\draw[thick,->] (boo3) to [out=90,in=-30] 
            (step3) 
                     to [out=-30,in=-150]       (step4) 
  to [out=90,in=-30] 
             (step5)       to [out=-30,in=-150]   (step6) 
              (step6) 
              to [out=90,in=-12.5] (hin)
       ;

    \end{tikzpicture}\end{minipage}
$$

\caption{We let $h=3$, $\ell=1$, $e=5$ and  $\al=\varepsilon_3-\varepsilon_1$.    We picture the paths 
 $ \SSTP_{\al\emptyset} $  and $\SSTP_{\emptyset\al}.$}
 \label{adjustmntexamplehidihi2}
\end{figure}
\!\!\!

   \subsubsection{The KLR hexagon diagram} \label{adjustgen}
 We wish to pass between the two distinct paths around a 
 vertex in our alcove geometry  which lies at the intersection of two hyperplanes labelled by non-commuting reflections.   
 To this end, we let $\al , \bet  \in \Pi  $  label a pair of non-commuting reflections.    Of course, one path around the vertex may be longer than the other.   
 Thus, we have two cases to consider: if $b_\al \geq b_\bet$ then we must pass between the paths   
$ 
  \SSTP_{\al \bet \al}$
 and $ 
 \SSTP_{\emp-\empb}\otimes \SSTP_{ \bet \al \bet}$
  and if 
 $b_\al \leq b_\bet$ then we pass between the paths   
$ 
 \SSTP_{\empb-\emp}\otimes  \SSTP_{\al \bet \al}$
and 
$ 
  \SSTP_{ \bet \al \bet}$, where here 
  $\emp -\empb:=\emptyset^{b_\al-b_\bet}$.

 \begin{figure}[ht!]
 $$  \begin{minipage}{4.05cm}\begin{tikzpicture}[scale=0.6]

  \draw[very thick, magenta] (0,0)--++(0:3) ;
    \draw[very thick, cyan] (3,0)--++(120:3) coordinate (hi);
      \draw[very thick, cyan] (3,0)--++(0:3)  ;
            \draw[very thick, magenta] (3,0)--++(60:3)  ;
                        \draw[very thick, magenta] (3,0)--++(-60:3)  ;
        \draw[very thick, darkgreen] (hi)--++(-120:3) coordinate  (hi);

        \draw[very thick, darkgreen] (hi)--++(-60:3) coordinate  (hi);

 \draw[very thick, darkgreen] (0,0)--++(60:3)--++(0:3)--++(-60:3)--++(-120:3)--++(180:3)--++(120:3);
    \draw[very thick, cyan] (hi)--++(60:3) coordinate (hi);

\clip(0,0)--++(60:3)--++(0:3)--++(-60:3)--++(-120:3)--++(180:3)--++(120:3);

    \path(0,0) coordinate (origin);
          \foreach \i in {0,1,2,3,4,5}
  {
    \path  (origin)++(60:0.6*\i cm)  coordinate (a\i);
    \path (origin)++(0:0.6*\i cm)  coordinate (b\i);
     \path (origin)++(-60:0.6*\i cm)  coordinate (c\i);
    }
  
      \path(3,0) coordinate (origin);
          \foreach \i in {0,1,2,3,4,5}
  {
    \path  (origin)++(120:0.6*\i cm)  coordinate (d\i);
    \path (origin)++(180:0.6*\i cm)  coordinate (e\i);
     \path (origin)++(-120:0.6*\i cm)  coordinate (f\i);
    }

  \foreach \i in {0,1,2,3,4,5}
  {
    \draw[gray, densely dotted] (a\i)--(b\i);
        \draw[gray, densely dotted] (c\i)--(b\i);
    \draw[gray, densely dotted] (d\i)--(e\i);
        \draw[gray, densely dotted] (f\i)--(e\i);
            \draw[gray, densely dotted] (a\i)--(d\i);
                \draw[gray, densely dotted] (c\i)--(f\i);
 
     }

   \path(0,0)--++(0:3) coordinate (origin22);
          \foreach \i in {0,1,2,3,4,5}
  {
    \path  (origin22)++(60:0.6*\i cm)  coordinate (a\i);
    \path (origin22)++(0:0.6*\i cm)  coordinate (b\i);
     \path (origin22)++(-60:0.6*\i cm)  coordinate (c\i);
    }
  
      \path(3,0) coordinate (origin22);
          \foreach \i in {0,1,2,3,4,5}
  {
    \path  (origin22)++(120:0.6*\i cm)  coordinate (d\i);
    \path (origin22)++(180:0.6*\i cm)  coordinate (e\i);
     \path (origin22)++(-120:0.6*\i cm)  coordinate (f\i);
    }

  \foreach \i in {0,1,2,3,4,5}
  {
    \draw[gray, densely dotted] (a\i)--(b\i);
        \draw[gray, densely dotted] (c\i)--(b\i);
    \draw[gray, densely dotted] (d\i)--(e\i);
        \draw[gray, densely dotted] (f\i)--(e\i);
            \draw[gray, densely dotted] (a\i)--(d\i);
                \draw[gray, densely dotted] (c\i)--(f\i);
 
     }

   \path(0,0)--++(60:3) coordinate (origin33);
          \foreach \i in {0,1,2,3,4,5}
  {
    \path  (origin33)++(-60+60:0.6*\i cm)  coordinate (a\i);
    \path (origin33)++(-60+0:0.6*\i cm)  coordinate (b\i);
     \path (origin33)++(-60+-60:0.6*\i cm)  coordinate (c\i);
    }
  
      \path(3,0) coordinate (origin33);
          \foreach \i in {0,1,2,3,4,5}
  {
    \path  (origin33)++(-60+120:0.6*\i cm)  coordinate (d\i);
    \path (origin33)++(-60+180:0.6*\i cm)  coordinate (e\i);
     \path (origin33)++(-60+-120:0.6*\i cm)  coordinate (f\i);
    }

  \foreach \i in {0,1,2,3,4,5}
  {
    \draw[gray, densely dotted] (a\i)--(b\i);
        \draw[gray, densely dotted] (c\i)--(b\i);
    \draw[gray, densely dotted] (d\i)--(e\i);
        \draw[gray, densely dotted] (f\i)--(e\i);
            \draw[gray, densely dotted] (a\i)--(d\i);
                \draw[gray, densely dotted] (c\i)--(f\i);
 
     }

   \path(0,0)--++(-60:3) coordinate (origin44);
          \foreach \i in {0,1,2,3,4,5}
  {
    \path  (origin44)--++(+120:0.6*\i cm)  coordinate (a\i);
    \path (origin44)--++(+60:0.6*\i cm)  coordinate (b\i);
     \path (origin44)--++(0:0.6*\i cm)  coordinate (c\i);
    }
  
      \path(origin44)--++(60:3) coordinate (origin55);
          \foreach \i in {0,1,2,3,4,5}
  {
    \path  (origin55)--++(+60+120:0.6*\i cm)  coordinate (d\i);
    \path (origin55)--++(+60+180:0.6*\i cm)  coordinate (e\i);
     \path (origin55)--++(+60-120:0.6*\i cm)  coordinate (f\i);
    }

  \foreach \i in {0,1,2,3,4,5}
  {
    \draw[gray, densely dotted] (a\i)--(b\i);
        \draw[gray, densely dotted] (c\i)--(b\i);
    \draw[gray, densely dotted] (d\i)--(e\i);
        \draw[gray, densely dotted] (f\i)--(e\i);
            \draw[gray, densely dotted] (a\i)--(d\i);
                \draw[gray, densely dotted] (c\i)--(f\i);
 
     }

        \path(origin44)--++(0:3) coordinate (origin44);
          \foreach \i in {0,1,2,3,4,5}
  {
    \path  (origin44)--++(+120:0.6*\i cm)  coordinate (a\i);
    \path (origin44)--++(+60:0.6*\i cm)  coordinate (b\i);
     \path (origin44)--++(0:0.6*\i cm)  coordinate (c\i);
    }
  
      \path(origin44)--++(60:3) coordinate (origin55);
          \foreach \i in {0,1,2,3,4,5}
  {
    \path  (origin55)--++(+60+120:0.6*\i cm)  coordinate (d\i);
    \path (origin55)--++(+60+180:0.6*\i cm)  coordinate (e\i);
     \path (origin55)--++(+60-120:0.6*\i cm)  coordinate (f\i);
    }

  \foreach \i in {0,1,2,3,4,5}
  {
    \draw[gray, densely dotted] (a\i)--(b\i);
        \draw[gray, densely dotted] (c\i)--(b\i);
    \draw[gray, densely dotted] (d\i)--(e\i);
        \draw[gray, densely dotted] (f\i)--(e\i);
            \draw[gray, densely dotted] (a\i)--(d\i);
                \draw[gray, densely dotted] (c\i)--(f\i);}

                     \path(origin44)--++(120:3) coordinate (origin44);
          \foreach \i in {0,1,2,3,4,5}
  {
    \path  (origin44)--++(+120:0.6*\i cm)  coordinate (a\i);
    \path (origin44)--++(+60:0.6*\i cm)  coordinate (b\i);
     \path (origin44)--++(0:0.6*\i cm)  coordinate (c\i);
    }
  
      \path(origin44)--++(60:3) coordinate (origin55);
          \foreach \i in {0,1,2,3,4,5}
  {
    \path  (origin55)--++(+60+120:0.6*\i cm)  coordinate (d\i);
    \path (origin55)--++(+60+180:0.6*\i cm)  coordinate (e\i);
     \path (origin55)--++(+60-120:0.6*\i cm)  coordinate (f\i);
    }

  \foreach \i in {0,1,2,3,4,5}
  {
    \draw[gray, densely dotted] (a\i)--(b\i);
        \draw[gray, densely dotted] (c\i)--(b\i);
    \draw[gray, densely dotted] (d\i)--(e\i);
        \draw[gray, densely dotted] (f\i)--(e\i);
            \draw[gray, densely dotted] (a\i)--(d\i);
                \draw[gray, densely dotted] (c\i)--(f\i);}

      \path(0,0)--++(-60:5*0.6)--++(120:2*0.6)--++(0:0.6) coordinate (hi);
     \path(0,0)--++(0:4*0.6)--++(120:3*0.6)           coordinate(hi2) ; 
     \path(0,0)--++(0:4*0.6)          coordinate(hi3) ;

          \path(hi)  --++(120:0.6)
           coordinate(step1) 
          --++(0:0.6)
                     coordinate(step2) 
      --++(120:0.6)
                 coordinate(step3) 
                 --++(0:0.6)           coordinate(step4) 
      --++(120:0.54) 
                 coordinate(step5)  --++(0:0.62) 
                 coordinate(step6)   ;

   \draw[very thick,->]    (hi)  --++(120:0.6)--++(0:0.6)
   --++(120:0.6)--++(0:0.6)
   --++(120:0.6)--++(0:0.6)
   --++(120:0.6*3) 
   --++(0:0.6)
   --++(120:0.6)
   --++(0:0.6)
      --++(0:0.6)
      --++(-120:0.6)  --++(0:0.6)
            --++(-120:0.6)  --++(0:0.6)
                  --++(-120:0.6)  --++(0:0.6*3)
         ;   
   
 \end{tikzpicture}\end{minipage}
\quad   \begin{minipage}{4.05cm}\begin{tikzpicture}[scale=0.6]

  \draw[very thick, magenta] (0,0)--++(0:3) ;
    \draw[very thick, cyan] (3,0)--++(120:3) coordinate (hi);
      \draw[very thick, cyan] (3,0)--++(0:3)  ;
            \draw[very thick, magenta] (3,0)--++(60:3)  ;
                        \draw[very thick, magenta] (3,0)--++(-60:3)  ;
        \draw[very thick, darkgreen] (hi)--++(-120:3) coordinate  (hi);

        \draw[very thick, darkgreen] (hi)--++(-60:3) coordinate  (hi);

 \draw[very thick, darkgreen] (0,0)--++(60:3)--++(0:3)--++(-60:3)--++(-120:3)--++(180:3)--++(120:3);
    \draw[very thick, cyan] (hi)--++(60:3) coordinate (hi);

\clip(0,0)--++(60:3)--++(0:3)--++(-60:3)--++(-120:3)--++(180:3)--++(120:3);

    \path(0,0) coordinate (origin);
          \foreach \i in {0,1,2,3,4,5}
  {
    \path  (origin)++(60:0.6*\i cm)  coordinate (a\i);
    \path (origin)++(0:0.6*\i cm)  coordinate (b\i);
     \path (origin)++(-60:0.6*\i cm)  coordinate (c\i);
    }
  
      \path(3,0) coordinate (origin);
          \foreach \i in {0,1,2,3,4,5}
  {
    \path  (origin)++(120:0.6*\i cm)  coordinate (d\i);
    \path (origin)++(180:0.6*\i cm)  coordinate (e\i);
     \path (origin)++(-120:0.6*\i cm)  coordinate (f\i);
    }

  \foreach \i in {0,1,2,3,4,5}
  {
    \draw[gray, densely dotted] (a\i)--(b\i);
        \draw[gray, densely dotted] (c\i)--(b\i);
    \draw[gray, densely dotted] (d\i)--(e\i);
        \draw[gray, densely dotted] (f\i)--(e\i);
            \draw[gray, densely dotted] (a\i)--(d\i);
                \draw[gray, densely dotted] (c\i)--(f\i);
 
     }

   \path(0,0)--++(0:3) coordinate (origin22);
          \foreach \i in {0,1,2,3,4,5}
  {
    \path  (origin22)++(60:0.6*\i cm)  coordinate (a\i);
    \path (origin22)++(0:0.6*\i cm)  coordinate (b\i);
     \path (origin22)++(-60:0.6*\i cm)  coordinate (c\i);
    }
  
      \path(3,0) coordinate (origin22);
          \foreach \i in {0,1,2,3,4,5}
  {
    \path  (origin22)++(120:0.6*\i cm)  coordinate (d\i);
    \path (origin22)++(180:0.6*\i cm)  coordinate (e\i);
     \path (origin22)++(-120:0.6*\i cm)  coordinate (f\i);
    }

  \foreach \i in {0,1,2,3,4,5}
  {
    \draw[gray, densely dotted] (a\i)--(b\i);
        \draw[gray, densely dotted] (c\i)--(b\i);
    \draw[gray, densely dotted] (d\i)--(e\i);
        \draw[gray, densely dotted] (f\i)--(e\i);
            \draw[gray, densely dotted] (a\i)--(d\i);
                \draw[gray, densely dotted] (c\i)--(f\i);
 
     }

   \path(0,0)--++(60:3) coordinate (origin33);
          \foreach \i in {0,1,2,3,4,5}
  {
    \path  (origin33)++(-60+60:0.6*\i cm)  coordinate (a\i);
    \path (origin33)++(-60+0:0.6*\i cm)  coordinate (b\i);
     \path (origin33)++(-60+-60:0.6*\i cm)  coordinate (c\i);
    }
  
      \path(3,0) coordinate (origin33);
          \foreach \i in {0,1,2,3,4,5}
  {
    \path  (origin33)++(-60+120:0.6*\i cm)  coordinate (d\i);
    \path (origin33)++(-60+180:0.6*\i cm)  coordinate (e\i);
     \path (origin33)++(-60+-120:0.6*\i cm)  coordinate (f\i);
    }

  \foreach \i in {0,1,2,3,4,5}
  {
    \draw[gray, densely dotted] (a\i)--(b\i);
        \draw[gray, densely dotted] (c\i)--(b\i);
    \draw[gray, densely dotted] (d\i)--(e\i);
        \draw[gray, densely dotted] (f\i)--(e\i);
            \draw[gray, densely dotted] (a\i)--(d\i);
                \draw[gray, densely dotted] (c\i)--(f\i);
 
     }

   \path(0,0)--++(-60:3) coordinate (origin44);
          \foreach \i in {0,1,2,3,4,5}
  {
    \path  (origin44)--++(+120:0.6*\i cm)  coordinate (a\i);
    \path (origin44)--++(+60:0.6*\i cm)  coordinate (b\i);
     \path (origin44)--++(0:0.6*\i cm)  coordinate (c\i);
    }
  
      \path(origin44)--++(60:3) coordinate (origin55);
          \foreach \i in {0,1,2,3,4,5}
  {
    \path  (origin55)--++(+60+120:0.6*\i cm)  coordinate (d\i);
    \path (origin55)--++(+60+180:0.6*\i cm)  coordinate (e\i);
     \path (origin55)--++(+60-120:0.6*\i cm)  coordinate (f\i);
    }

  \foreach \i in {0,1,2,3,4,5}
  {
    \draw[gray, densely dotted] (a\i)--(b\i);
        \draw[gray, densely dotted] (c\i)--(b\i);
    \draw[gray, densely dotted] (d\i)--(e\i);
        \draw[gray, densely dotted] (f\i)--(e\i);
            \draw[gray, densely dotted] (a\i)--(d\i);
                \draw[gray, densely dotted] (c\i)--(f\i);
 
     }

        \path(origin44)--++(0:3) coordinate (origin44);
          \foreach \i in {0,1,2,3,4,5}
  {
    \path  (origin44)--++(+120:0.6*\i cm)  coordinate (a\i);
    \path (origin44)--++(+60:0.6*\i cm)  coordinate (b\i);
     \path (origin44)--++(0:0.6*\i cm)  coordinate (c\i);
    }
  
      \path(origin44)--++(60:3) coordinate (origin55);
          \foreach \i in {0,1,2,3,4,5}
  {
    \path  (origin55)--++(+60+120:0.6*\i cm)  coordinate (d\i);
    \path (origin55)--++(+60+180:0.6*\i cm)  coordinate (e\i);
     \path (origin55)--++(+60-120:0.6*\i cm)  coordinate (f\i);
    }

  \foreach \i in {0,1,2,3,4,5}
  {
    \draw[gray, densely dotted] (a\i)--(b\i);
        \draw[gray, densely dotted] (c\i)--(b\i);
    \draw[gray, densely dotted] (d\i)--(e\i);
        \draw[gray, densely dotted] (f\i)--(e\i);
            \draw[gray, densely dotted] (a\i)--(d\i);
                \draw[gray, densely dotted] (c\i)--(f\i);}

                     \path(origin44)--++(120:3) coordinate (origin44);
          \foreach \i in {0,1,2,3,4,5}
  {
    \path  (origin44)--++(+120:0.6*\i cm)  coordinate (a\i);
    \path (origin44)--++(+60:0.6*\i cm)  coordinate (b\i);
     \path (origin44)--++(0:0.6*\i cm)  coordinate (c\i);
    }
  
      \path(origin44)--++(60:3) coordinate (origin55);
          \foreach \i in {0,1,2,3,4,5}
  {
    \path  (origin55)--++(+60+120:0.6*\i cm)  coordinate (d\i);
    \path (origin55)--++(+60+180:0.6*\i cm)  coordinate (e\i);
     \path (origin55)--++(+60-120:0.6*\i cm)  coordinate (f\i);
    }

  \foreach \i in {0,1,2,3,4,5}
  {
    \draw[gray, densely dotted] (a\i)--(b\i);
        \draw[gray, densely dotted] (c\i)--(b\i);
    \draw[gray, densely dotted] (d\i)--(e\i);
        \draw[gray, densely dotted] (f\i)--(e\i);
            \draw[gray, densely dotted] (a\i)--(d\i);
                \draw[gray, densely dotted] (c\i)--(f\i);}

      \path(0,0)--++(-60:5*0.6)--++(120:2*0.6)--++(0:0.6) coordinate (hi);
     \path(0,0)--++(0:4*0.6)--++(120:3*0.6)           coordinate(hi2) ; 
     \path(0,0)--++(0:4*0.6)          coordinate(hi3) ;

          \path(hi)  --++(120:0.6)
           coordinate(step1) 
          --++(0:0.6)
                     coordinate(step2) 
      --++(120:0.6)
                 coordinate(step3) 
                 --++(0:0.6)           coordinate(step4) 
      --++(120:0.54) 
                 coordinate(step5)  --++(0:0.62) 
                 coordinate(step6)   ;

   \draw[very thick,->]    (hi)

   --++(0:0.6)   --++(-120:0.6)        
 --++(0:0.6*2)
   --++(120:0.6)--++(0:0.6)   --++(120:0.6)--++(0:0.6) --++(120:0.6)--++(0:0.6) 
    --++(0:0.6*2)
    --++(120:0.6)--++(0:0.6)
        --++(120:0.6*1) 
    ;   
   
 \end{tikzpicture}\end{minipage}
 \quad   \begin{minipage}{4.05cm}\begin{tikzpicture}[scale=0.6]

  \draw[very thick, cyan] (0,0)--++(0:3) ;
    \draw[very thick, darkgreen] (3,0)--++(120:3) coordinate (hi);
      \draw[very thick, darkgreen] (3,0)--++(0:3)  ;
            \draw[very thick, cyan] (3,0)--++(60:3)  ;
                        \draw[very thick, cyan] (3,0)--++(-60:3)  ;
        \draw[very thick, magenta] (hi)--++(-120:3) coordinate  (hi);

        \draw[very thick, magenta] (hi)--++(-60:3) coordinate  (hi);

 \draw[very thick, magenta] (0,0)--++(60:3)--++(0:3)--++(-60:3)--++(-120:3)--++(180:3)--++(120:3);
    \draw[very thick, darkgreen] (hi)--++(60:3) coordinate (hi);

\clip(0,0)--++(60:3)--++(0:3)--++(-60:3)--++(-120:3)--++(180:3)--++(120:3);

    \path(0,0) coordinate (origin);
          \foreach \i in {0,1,2,3,4,5}
  {
    \path  (origin)++(60:0.6*\i cm)  coordinate (a\i);
    \path (origin)++(0:0.6*\i cm)  coordinate (b\i);
     \path (origin)++(-60:0.6*\i cm)  coordinate (c\i);
    }
  
      \path(3,0) coordinate (origin);
          \foreach \i in {0,1,2,3,4,5}
  {
    \path  (origin)++(120:0.6*\i cm)  coordinate (d\i);
    \path (origin)++(180:0.6*\i cm)  coordinate (e\i);
     \path (origin)++(-120:0.6*\i cm)  coordinate (f\i);
    }

  \foreach \i in {0,1,2,3,4,5}
  {
    \draw[gray, densely dotted] (a\i)--(b\i);
        \draw[gray, densely dotted] (c\i)--(b\i);
    \draw[gray, densely dotted] (d\i)--(e\i);
        \draw[gray, densely dotted] (f\i)--(e\i);
            \draw[gray, densely dotted] (a\i)--(d\i);
                \draw[gray, densely dotted] (c\i)--(f\i);
 
     }

   \path(0,0)--++(0:3) coordinate (origin22);
          \foreach \i in {0,1,2,3,4,5}
  {
    \path  (origin22)++(60:0.6*\i cm)  coordinate (a\i);
    \path (origin22)++(0:0.6*\i cm)  coordinate (b\i);
     \path (origin22)++(-60:0.6*\i cm)  coordinate (c\i);
    }
  
      \path(3,0) coordinate (origin22);
          \foreach \i in {0,1,2,3,4,5}
  {
    \path  (origin22)++(120:0.6*\i cm)  coordinate (d\i);
    \path (origin22)++(180:0.6*\i cm)  coordinate (e\i);
     \path (origin22)++(-120:0.6*\i cm)  coordinate (f\i);
    }

  \foreach \i in {0,1,2,3,4,5}
  {
    \draw[gray, densely dotted] (a\i)--(b\i);
        \draw[gray, densely dotted] (c\i)--(b\i);
    \draw[gray, densely dotted] (d\i)--(e\i);
        \draw[gray, densely dotted] (f\i)--(e\i);
            \draw[gray, densely dotted] (a\i)--(d\i);
                \draw[gray, densely dotted] (c\i)--(f\i);
 
     }

   \path(0,0)--++(60:3) coordinate (origin33);
          \foreach \i in {0,1,2,3,4,5}
  {
    \path  (origin33)++(-60+60:0.6*\i cm)  coordinate (a\i);
    \path (origin33)++(-60+0:0.6*\i cm)  coordinate (b\i);
     \path (origin33)++(-60+-60:0.6*\i cm)  coordinate (c\i);
    }
  
      \path(3,0) coordinate (origin33);
          \foreach \i in {0,1,2,3,4,5}
  {
    \path  (origin33)++(-60+120:0.6*\i cm)  coordinate (d\i);
    \path (origin33)++(-60+180:0.6*\i cm)  coordinate (e\i);
     \path (origin33)++(-60+-120:0.6*\i cm)  coordinate (f\i);
    }

  \foreach \i in {0,1,2,3,4,5}
  {
    \draw[gray, densely dotted] (a\i)--(b\i);
        \draw[gray, densely dotted] (c\i)--(b\i);
    \draw[gray, densely dotted] (d\i)--(e\i);
        \draw[gray, densely dotted] (f\i)--(e\i);
            \draw[gray, densely dotted] (a\i)--(d\i);
                \draw[gray, densely dotted] (c\i)--(f\i);
 
     }

   \path(0,0)--++(-60:3) coordinate (origin44);
          \foreach \i in {0,1,2,3,4,5}
  {
    \path  (origin44)--++(+120:0.6*\i cm)  coordinate (a\i);
    \path (origin44)--++(+60:0.6*\i cm)  coordinate (b\i);
     \path (origin44)--++(0:0.6*\i cm)  coordinate (c\i);
    }
  
      \path(origin44)--++(60:3) coordinate (origin55);
          \foreach \i in {0,1,2,3,4,5}
  {
    \path  (origin55)--++(+60+120:0.6*\i cm)  coordinate (d\i);
    \path (origin55)--++(+60+180:0.6*\i cm)  coordinate (e\i);
     \path (origin55)--++(+60-120:0.6*\i cm)  coordinate (f\i);
    }

  \foreach \i in {0,1,2,3,4,5}
  {
    \draw[gray, densely dotted] (a\i)--(b\i);
        \draw[gray, densely dotted] (c\i)--(b\i);
    \draw[gray, densely dotted] (d\i)--(e\i);
        \draw[gray, densely dotted] (f\i)--(e\i);
            \draw[gray, densely dotted] (a\i)--(d\i);
                \draw[gray, densely dotted] (c\i)--(f\i);
 
     }

        \path(origin44)--++(0:3) coordinate (origin44);
          \foreach \i in {0,1,2,3,4,5}
  {
    \path  (origin44)--++(+120:0.6*\i cm)  coordinate (a\i);
    \path (origin44)--++(+60:0.6*\i cm)  coordinate (b\i);
     \path (origin44)--++(0:0.6*\i cm)  coordinate (c\i);
    }
  
      \path(origin44)--++(60:3) coordinate (origin55);
          \foreach \i in {0,1,2,3,4,5}
  {
    \path  (origin55)--++(+60+120:0.6*\i cm)  coordinate (d\i);
    \path (origin55)--++(+60+180:0.6*\i cm)  coordinate (e\i);
     \path (origin55)--++(+60-120:0.6*\i cm)  coordinate (f\i);
    }

  \foreach \i in {0,1,2,3,4,5}
  {
    \draw[gray, densely dotted] (a\i)--(b\i);
        \draw[gray, densely dotted] (c\i)--(b\i);
    \draw[gray, densely dotted] (d\i)--(e\i);
        \draw[gray, densely dotted] (f\i)--(e\i);
            \draw[gray, densely dotted] (a\i)--(d\i);
                \draw[gray, densely dotted] (c\i)--(f\i);}

                     \path(origin44)--++(120:3) coordinate (origin44);
          \foreach \i in {0,1,2,3,4,5}
  {
    \path  (origin44)--++(+120:0.6*\i cm)  coordinate (a\i);
    \path (origin44)--++(+60:0.6*\i cm)  coordinate (b\i);
     \path (origin44)--++(0:0.6*\i cm)  coordinate (c\i);
    }
  
      \path(origin44)--++(60:3) coordinate (origin55);
          \foreach \i in {0,1,2,3,4,5}
  {
    \path  (origin55)--++(+60+120:0.6*\i cm)  coordinate (d\i);
    \path (origin55)--++(+60+180:0.6*\i cm)  coordinate (e\i);
     \path (origin55)--++(+60-120:0.6*\i cm)  coordinate (f\i);
    }

  \foreach \i in {0,1,2,3,4,5}
  {
    \draw[gray, densely dotted] (a\i)--(b\i);
        \draw[gray, densely dotted] (c\i)--(b\i);
    \draw[gray, densely dotted] (d\i)--(e\i);
        \draw[gray, densely dotted] (f\i)--(e\i);
            \draw[gray, densely dotted] (a\i)--(d\i);
                \draw[gray, densely dotted] (c\i)--(f\i);}

      \path(0,0) --++(0:6*0.6)--++(120:0.6*2)coordinate (hi);
     \path(0,0)--++(0:4*0.6)--++(120:3*0.6)           coordinate(hi2) ; 
     \path(0,0)--++(0:4*0.6)          coordinate(hi3) ;

          \path(hi)  --++(120:0.6)
           coordinate(step1) 
          --++(0:0.6)
                     coordinate(step2) 
      --++(120:0.6)
                 coordinate(step3) 
                 --++(0:0.6)           coordinate(step4) 
      --++(120:0.54) 
                 coordinate(step5)  --++(0:0.62) 
                 coordinate(step6)   ;

   \draw[very thick,->]    (hi)

   --++(120:0.6) 
      --++(-120:0.6*3) 
            --++(0:0.6) 
                  --++(-120:0.6*2) 
                        --++(0:0.6*2) 
    ;   
   
 \end{tikzpicture}\end{minipage}
  \quad   \begin{minipage}{4.05cm}\begin{tikzpicture}[scale=0.6]

  \draw[very thick, cyan] (0,0)--++(0:3) ;
    \draw[very thick, darkgreen] (3,0)--++(120:3) coordinate (hi);
      \draw[very thick, darkgreen] (3,0)--++(0:3)  ;
            \draw[very thick, cyan] (3,0)--++(60:3)  ;
                        \draw[very thick, cyan] (3,0)--++(-60:3)  ;
        \draw[very thick, magenta] (hi)--++(-120:3) coordinate  (hi);

        \draw[very thick, magenta] (hi)--++(-60:3) coordinate  (hi);

 \draw[very thick, magenta] (0,0)--++(60:3)--++(0:3)--++(-60:3)--++(-120:3)--++(180:3)--++(120:3);
    \draw[very thick, darkgreen] (hi)--++(60:3) coordinate (hi);

\clip(0,0)--++(60:3)--++(0:3)--++(-60:3)--++(-120:3)--++(180:3)--++(120:3);

    \path(0,0) coordinate (origin);
          \foreach \i in {0,1,2,3,4,5}
  {
    \path  (origin)++(60:0.6*\i cm)  coordinate (a\i);
    \path (origin)++(0:0.6*\i cm)  coordinate (b\i);
     \path (origin)++(-60:0.6*\i cm)  coordinate (c\i);
    }
  
      \path(3,0) coordinate (origin);
          \foreach \i in {0,1,2,3,4,5}
  {
    \path  (origin)++(120:0.6*\i cm)  coordinate (d\i);
    \path (origin)++(180:0.6*\i cm)  coordinate (e\i);
     \path (origin)++(-120:0.6*\i cm)  coordinate (f\i);
    }

  \foreach \i in {0,1,2,3,4,5}
  {
    \draw[gray, densely dotted] (a\i)--(b\i);
        \draw[gray, densely dotted] (c\i)--(b\i);
    \draw[gray, densely dotted] (d\i)--(e\i);
        \draw[gray, densely dotted] (f\i)--(e\i);
            \draw[gray, densely dotted] (a\i)--(d\i);
                \draw[gray, densely dotted] (c\i)--(f\i);
 
     }

   \path(0,0)--++(0:3) coordinate (origin22);
          \foreach \i in {0,1,2,3,4,5}
  {
    \path  (origin22)++(60:0.6*\i cm)  coordinate (a\i);
    \path (origin22)++(0:0.6*\i cm)  coordinate (b\i);
     \path (origin22)++(-60:0.6*\i cm)  coordinate (c\i);
    }
  
      \path(3,0) coordinate (origin22);
          \foreach \i in {0,1,2,3,4,5}
  {
    \path  (origin22)++(120:0.6*\i cm)  coordinate (d\i);
    \path (origin22)++(180:0.6*\i cm)  coordinate (e\i);
     \path (origin22)++(-120:0.6*\i cm)  coordinate (f\i);
    }

  \foreach \i in {0,1,2,3,4,5}
  {
    \draw[gray, densely dotted] (a\i)--(b\i);
        \draw[gray, densely dotted] (c\i)--(b\i);
    \draw[gray, densely dotted] (d\i)--(e\i);
        \draw[gray, densely dotted] (f\i)--(e\i);
            \draw[gray, densely dotted] (a\i)--(d\i);
                \draw[gray, densely dotted] (c\i)--(f\i);
 
     }

   \path(0,0)--++(60:3) coordinate (origin33);
          \foreach \i in {0,1,2,3,4,5}
  {
    \path  (origin33)++(-60+60:0.6*\i cm)  coordinate (a\i);
    \path (origin33)++(-60+0:0.6*\i cm)  coordinate (b\i);
     \path (origin33)++(-60+-60:0.6*\i cm)  coordinate (c\i);
    }
  
      \path(3,0) coordinate (origin33);
          \foreach \i in {0,1,2,3,4,5}
  {
    \path  (origin33)++(-60+120:0.6*\i cm)  coordinate (d\i);
    \path (origin33)++(-60+180:0.6*\i cm)  coordinate (e\i);
     \path (origin33)++(-60+-120:0.6*\i cm)  coordinate (f\i);
    }

  \foreach \i in {0,1,2,3,4,5}
  {
    \draw[gray, densely dotted] (a\i)--(b\i);
        \draw[gray, densely dotted] (c\i)--(b\i);
    \draw[gray, densely dotted] (d\i)--(e\i);
        \draw[gray, densely dotted] (f\i)--(e\i);
            \draw[gray, densely dotted] (a\i)--(d\i);
                \draw[gray, densely dotted] (c\i)--(f\i);
 
     }

   \path(0,0)--++(-60:3) coordinate (origin44);
          \foreach \i in {0,1,2,3,4,5}
  {
    \path  (origin44)--++(+120:0.6*\i cm)  coordinate (a\i);
    \path (origin44)--++(+60:0.6*\i cm)  coordinate (b\i);
     \path (origin44)--++(0:0.6*\i cm)  coordinate (c\i);
    }
  
      \path(origin44)--++(60:3) coordinate (origin55);
          \foreach \i in {0,1,2,3,4,5}
  {
    \path  (origin55)--++(+60+120:0.6*\i cm)  coordinate (d\i);
    \path (origin55)--++(+60+180:0.6*\i cm)  coordinate (e\i);
     \path (origin55)--++(+60-120:0.6*\i cm)  coordinate (f\i);
    }

  \foreach \i in {0,1,2,3,4,5}
  {
    \draw[gray, densely dotted] (a\i)--(b\i);
        \draw[gray, densely dotted] (c\i)--(b\i);
    \draw[gray, densely dotted] (d\i)--(e\i);
        \draw[gray, densely dotted] (f\i)--(e\i);
            \draw[gray, densely dotted] (a\i)--(d\i);
                \draw[gray, densely dotted] (c\i)--(f\i);
 
     }

        \path(origin44)--++(0:3) coordinate (origin44);
          \foreach \i in {0,1,2,3,4,5}
  {
    \path  (origin44)--++(+120:0.6*\i cm)  coordinate (a\i);
    \path (origin44)--++(+60:0.6*\i cm)  coordinate (b\i);
     \path (origin44)--++(0:0.6*\i cm)  coordinate (c\i);
    }
  
      \path(origin44)--++(60:3) coordinate (origin55);
          \foreach \i in {0,1,2,3,4,5}
  {
    \path  (origin55)--++(+60+120:0.6*\i cm)  coordinate (d\i);
    \path (origin55)--++(+60+180:0.6*\i cm)  coordinate (e\i);
     \path (origin55)--++(+60-120:0.6*\i cm)  coordinate (f\i);
    }

  \foreach \i in {0,1,2,3,4,5}
  {
    \draw[gray, densely dotted] (a\i)--(b\i);
        \draw[gray, densely dotted] (c\i)--(b\i);
    \draw[gray, densely dotted] (d\i)--(e\i);
        \draw[gray, densely dotted] (f\i)--(e\i);
            \draw[gray, densely dotted] (a\i)--(d\i);
                \draw[gray, densely dotted] (c\i)--(f\i);}

                     \path(origin44)--++(120:3) coordinate (origin44);
          \foreach \i in {0,1,2,3,4,5}
  {
    \path  (origin44)--++(+120:0.6*\i cm)  coordinate (a\i);
    \path (origin44)--++(+60:0.6*\i cm)  coordinate (b\i);
     \path (origin44)--++(0:0.6*\i cm)  coordinate (c\i);
    }
  
      \path(origin44)--++(60:3) coordinate (origin55);
          \foreach \i in {0,1,2,3,4,5}
  {
    \path  (origin55)--++(+60+120:0.6*\i cm)  coordinate (d\i);
    \path (origin55)--++(+60+180:0.6*\i cm)  coordinate (e\i);
     \path (origin55)--++(+60-120:0.6*\i cm)  coordinate (f\i);
    }

  \foreach \i in {0,1,2,3,4,5}
  {
    \draw[gray, densely dotted] (a\i)--(b\i);
        \draw[gray, densely dotted] (c\i)--(b\i);
    \draw[gray, densely dotted] (d\i)--(e\i);
        \draw[gray, densely dotted] (f\i)--(e\i);
            \draw[gray, densely dotted] (a\i)--(d\i);
                \draw[gray, densely dotted] (c\i)--(f\i);}

      \path(0,0) --++(0:6*0.6)--++(120:0.6*2)coordinate (hi);
     \path(0,0)--++(0:4*0.6)--++(120:3*0.6)           coordinate(hi2) ; 
     \path(0,0)--++(0:4*0.6)          coordinate(hi3) ;

          \path(hi)  --++(120:0.6)
           coordinate(step1) 
          --++(0:0.6)
                     coordinate(step2) 
      --++(120:0.6)
                 coordinate(step3) 
                 --++(0:0.6)           coordinate(step4) 
      --++(120:0.54) 
                 coordinate(step5)  --++(0:0.62) 
                 coordinate(step6)   ;

   \draw[very thick,->]    (hi)

   --++(0:0.6) 
      --++(-120:0.6*1) 
       --++(0:0.6*2)  
      --++(-120:0.6*3)
            --++(120:0.6*1)  
                  --++(-120:0.6*1)     ;   
   
 \end{tikzpicture}\end{minipage}$$
 \caption{
 We let $h=3$, $\ell=1$, $e=5$ and  $\al=\varepsilon_3-\varepsilon_1$ and $\bet=\eps_1-\eps_2$ and $\gam=\eps_2-\eps_3$.   
The paths  $  \SSTP_{\al\bet\al}$, $  \SSTP_{ \bet\al\bet}$,
 $  \SSTP_{\gam\bet\gam}$ and  $  \SSTP_{ \bet\gam\bet}$ are pictured.  }
 \label{Steinberg}\end{figure}

  We   define the {\sf KLR-hexagon} to be the element 
 $$ 
 {\sf hex}^{\al\bet\al}_{\bet\al\bet}:=
 \Upsilon ^{ \SSTP_{\al \bet \al}}_{ \SSTP_{\emp-\empb}\otimes\SSTP_{  \bet \al  \bet}} \qquad\text{or}\qquad
  {\sf hex}^{\al\bet\al}_{\bet\al\bet}:=
  \Upsilon ^{  \SSTP_{\empb-\emp}\otimes \SSTP_{\al \bet \al}}_{ \SSTP_{  \bet \al  \bet}}
  $$ 
 for $b_\al\geq b_\bet$ or $b_\al \leq b_\bet$ respectively.  
Two such pairs of paths are despited in \cref{Steinberg}.  
For the latter pair, the corresponding KLR-hexagon element is depicted in \cref{kjhdfljkdfghlsdkfhg}.

 \begin{figure}[ht!] 
  $$  
 \begin{tikzpicture}
  [xscale=0.5,yscale=  
  0.5]

     \foreach \i in {0,1.5,3,...,12}
 {\draw(0,\i)--++(0:9);}
      \draw(0,0) rectangle (9,1.5*9);
        \foreach \i in {0.5,1.5,...,8.5}
  {
   \fill(\i,0) circle(1.5pt) coordinate (a\i);
          \fill(\i,1.5)circle(1.5pt)  coordinate (d\i);
   \fill(\i,3) circle(1.5pt) coordinate (a\i);
          \fill(\i,4.5)circle(1.5pt)  coordinate (d\i);
   \fill(\i,6) circle(1.5pt) coordinate (a\i);
          \fill(\i,7.5)circle(1.5pt)  coordinate (d\i);
   \fill(\i,9) circle(1.5pt) coordinate (a\i);
          \fill(\i,10.5)circle(1.5pt)  coordinate (d\i);
   \fill(\i,12) circle(1.5pt) coordinate (a\i);
          \fill(\i,13.5)circle(1.5pt)  coordinate (d\i);
                \fill(\i,12+1.5)circle(1.5pt)  coordinate (d\i);
       }


\draw[darkgreen] (3.5-3,0) node[below] {$0$};
\draw[darkgreen] (4.5-3,0) node[below] {$2$};
\draw[darkgreen] (5.5-3,0) node[below] {$1$};

\draw[cyan] (3.5,0) node[below] {$0$};
\draw[cyan] (4.5,0) node[below] {$ 1$};
\draw[cyan] (5.5,0) node[below] {$4 $};

 \draw[darkgreen] (6.5,0) node[below] {$3 $};
\draw[darkgreen] (7.5,0) node[below] {$0  $};
\draw[darkgreen] (8.5,0) node[below] {$ 4$};

\draw[darkgreen] (3.5-3,-1) node[below] {$\eps_1$};
\draw[darkgreen] (4.5-3,-1) node[below] {$\eps_3$};
\draw[darkgreen] (5.5-3,-1) node[below] {$\eps_3$};

\draw[cyan] (3.5,-1) node[below] {$\eps_3$};
\draw[cyan] (4.5,-1) node[below] {$\eps_2$};
\draw[cyan] (5.5,-1) node[below] {$\eps_3$};

 \draw[darkgreen] (6.5,-1) node[below] {$\eps_3$};
\draw[darkgreen] (7.5,-1) node[below] {$\eps_2$};
\draw[darkgreen] (8.5,-1) node[below] {$\eps_2$};

  \foreach \i in {13.5 }
  {
    \foreach \j in {0.5,1.5,...,8.5} 
   { \draw(\j,\i)--++(-90:1.5) ;
}}

   \foreach \i in {13.5-1.5}
  {
 \draw(0.5,\i)--(1.5,\i-1.5)--++(-90:3);
  \draw(1.5,\i)--(0.5,\i-1.5)--++(-90:4.5);
   \foreach \j in {2.5,3.5,...,8.5} 
   { \draw(\j,\i)--++(-90:1.5)--++(-90:3);
}}

  \foreach \i in {9-1.5}
  {
   \foreach \j in {1.5,2.5,...,3.5} 
{    \draw(\j,\i)--(\j+1,\i-1.5);}  
\draw(4.5,\i)--(1.5,\i-1.5);
    \foreach \j in {5.5,6.5,...,8.5} 
   { \draw(\j,\i)--++(-90:1.5);
}}

\draw(0.5,7.5-1.5)--++(-90:3);
\draw(1.5,7.5-1.5)--++(-90:3);

  \foreach \i in {7.5-1.5}
  {
   \foreach \j in {2.5,3.5,...,4.5} 
{    \draw(\j,\i)--(\j+1,\i-1.5);}  
\draw(5.5,\i)--(2.5,\i-1.5);
    \foreach \j in {6.5,7.5,...,8.5} 
   { \draw(\j,\i)--++(-90:1.5);
}}

 \foreach \i in {7.5-3}
  {
   \foreach \j in {4.5,5.5} 
{    \draw(\j,\i)--(\j+1,\i-1.5);}  
\draw(6.5,\i)--(4.5,\i-1.5);
    \foreach \j in {7.5,...,8.5} 
   { \draw(\j,\i)--++(-90:1.5);
}
   \foreach \j in {2.5,3.5} 
   { \draw(\j,\i)--++(-90:1.5);}
   }

 \foreach \i in {7.5-4.5}
  {
   \foreach \j in {0.5,1.5,...,6.5} 
{    \draw(\j,\i)--(\j+1,\i-1.5);}  
\draw(7.5,\i)--(0.5,\i-1.5);
\draw(8.5,\i)--++(-90:1.5);
 }

  \foreach \i in {7.5-6}
  {
   \foreach \j in {6.5,7.5} 
{    \draw(\j,\i)--(\j+1,\i-1.5);}  
\draw(8.5,\i)--(6.5,\i-1.5);
    \foreach \j in {0.5,1.5,...,5.5} 
   { \draw(\j,\i)--++(-90:1.5);
}
   }

\draw[cyan] (3.5-3,1+1.5*9) node[above] {$\eps_2$};
\draw[cyan] (4.5-3,1+1.5*9) node[above] {$\eps_3$};
\draw[cyan] (5.5-3,1+1.5*9) node[above] {$\eps_2$};

\draw[darkgreen] (3.5,1+1.5*9) node[above] {$\eps_2$};
\draw[darkgreen] (4.5,1+1.5*9) node[above] {$\eps_3$};
\draw[darkgreen] (5.5,1+1.5*9) node[above] {$\eps_3$};

 \draw[cyan] (6.5,1+1.5*9) node[above] {$\eps_3$};
\draw[cyan] (7.5,1+1.5*9) node[above] {$\eps_1$};
\draw[cyan] (8.5,1+1.5*9) node[above] {$\eps_3$};

\draw[cyan] (3.5-3,1.5*9) node[above] {$1$};
\draw[cyan] (4.5-3,1.5*9) node[above] {$2$};
\draw[cyan] (5.5-3,1.5*9) node[above] {$0$};

\draw[darkgreen] (3.5,1.5*9) node[above] {$4$};
\draw[darkgreen] (4.5,1.5*9) node[above] {$1$};
\draw[darkgreen] (5.5,1.5*9) node[above] {$0$};

 \draw[cyan] (6.5,1.5*9) node[above] {$4$};
\draw[cyan] (7.5,1.5*9) node[above] {$0$};
\draw[cyan] (8.5,1.5*9) node[above] {$3$};

    \end{tikzpicture} $$

    \caption{  We let $h=3$, $\ell=1$, $e=5$ and $\bet=\eps_1-\eps_2$ and $\gam=\eps_2-\eps_3$.   
  We picture ${\sf hex}^
    {\SSTP_{\bet\gam\bet}}
    _{\SSTP_{\gam\bet\gam}}
    $.   }
    \label{kjhdfljkdfghlsdkfhg}    \end{figure}

     \subsubsection{The KLR commutator}

Let $\gam,\bet\in \Pi$ be roots labelling commuting reflections. 
We wish to understand the  morphism relating the paths 
    $\SSTP_{\gam }\otimes \SSTP_{\bet}$ to $\SSTP_{\bet}\otimes \SSTP_\gam$.  
     We define the {\sf KLR-commutator} to be the element 
 $$ 
 {\sf com}^{\gam\bet}_{\bet\gam}:= 
\Upsilon   ^{\SSTP_{\gam }\otimes \SSTP_{\bet} }
  _{\SSTP_{\bet}\otimes \SSTP_\gam}. $$

  \begin{figure}[ht!]
  $$
    \begin{minipage}{3.4cm}\begin{tikzpicture}[scale=0.8]

    \draw[very thick,->](1.2,0.6)--++(0:0.6*4)--++(90:0.6*4);
    
    \clip(0.6*2,0.6) rectangle (0.6*6,0.6*5);

    \draw[very thick, darkgreen ]  (0.4*6,0)--++(90:0.6*6); 
    \draw[very thick, cyan]  (0,0.6*3)--++(0:0.6*8);

    \clip(0.6*2,0.6) rectangle (0.6*6,0.6*5);

    \path(0,0) coordinate (origin);
          \foreach \i in {-1,0,1,2,3,4,5,6,7,8,9}
  {
    \path  (origin)++(90:0.6*\i cm)  coordinate (a\i);
     }
     
               \foreach \i in {-1,0,1,2,3,4,5,6,7,8,9}
  {
     \path (origin)++(0:0.6*\i cm)  coordinate (b\i);
    }

      \path(0.6*8,0) coordinate (origin);
           \foreach \i in {-1,0,1,2,3,4,5,6,7,8,9}
  {
    \path  (origin)++(90:0.6*\i cm)  coordinate (c\i);
     }

      \path(0,0.6*6) coordinate (origin);
           \foreach \i in {-1,0,1,2,3,4,5,6,7,8,9}
  {
    \path  (origin)++(90:0.6*\i cm)  coordinate (d\i);
     }

  \foreach \i in {0,1,2,3,4,5,6,7,8}
  {
    \draw[gray, densely dotted] (a\i)--(c\i);
         \draw[gray, densely dotted] (b\i)--++(90:6*0.6); 
     
     }

 \end{tikzpicture}\end{minipage}
 \quad   \begin{minipage}{3.4cm}\begin{tikzpicture}[scale=0.8]

    \draw[very thick,->](1.2,0.6)--++(90:0.6*4)--++(0:0.6*4);
    
    \clip(0.6*2,0.6) rectangle (0.6*6,0.6*5);

    \draw[very thick, darkgreen ]  (0.4*6,0)--++(90:0.6*6); 
    \draw[very thick, cyan]  (0,0.6*3)--++(0:0.6*8);

    \clip(0.6*2,0.6) rectangle (0.6*6,0.6*5);

    \path(0,0) coordinate (origin);
          \foreach \i in {-1,0,1,2,3,4,5,6,7,8,9}
  {
    \path  (origin)++(90:0.6*\i cm)  coordinate (a\i);
     }
     
               \foreach \i in {-1,0,1,2,3,4,5,6,7,8,9}
  {
     \path (origin)++(0:0.6*\i cm)  coordinate (b\i);
    }

      \path(0.6*8,0) coordinate (origin);
           \foreach \i in {-1,0,1,2,3,4,5,6,7,8,9}
  {
    \path  (origin)++(90:0.6*\i cm)  coordinate (c\i);
     }

      \path(0,0.6*6) coordinate (origin);
           \foreach \i in {-1,0,1,2,3,4,5,6,7,8,9}
  {
    \path  (origin)++(90:0.6*\i cm)  coordinate (d\i);
     }

  \foreach \i in {0,1,2,3,4,5,6,7,8}
  {
    \draw[gray, densely dotted] (a\i)--(c\i);
         \draw[gray, densely dotted] (b\i)--++(90:6*0.6); 
     
     }

 \end{tikzpicture}\end{minipage} 
  \begin{minipage}{3.4cm}\begin{tikzpicture}[scale=0.58]     
  
    \draw[very thick,   <-> ]  (0.4*6,0.9)--++(90:0.6*3); 
    \draw[very thick, <->]  (0.6*2.5,0.6*3)--++(0:0.6*3); 
   
 \draw(0.6,0.6*3) node {$+\eps_3$};
  \draw(0.6*7,0.6*3.1) node {$+\eps_{4}$};

  \draw(0.6*4,0.6*0.75) node {$+\eps_{1}$};
  \draw(0.6*4,0.6*6-0.6*0.75) node {$+\eps_{2}$};



 \end{tikzpicture}\end{minipage}
 $$  
  \caption{
  We let $h=1$, $\ell=4$, $\kappa=(0,2,4,6)\in (\ZZ/8\ZZ)^4$ and $\bet=\eps_1-\eps_2$ and $\gam=\eps_3-\eps_4$. 
We picture the     paths    $\SSTP_{\gam\bet}$ and $\SSTP_{\bet\gam}$,  the corresponding element $ {\sf com}_{\gam\bet}^{\bet\gam}$ is depicted in \cref{kjhdfljkdfghlsdkfhg222222222222}.}
\label{chopin}   \end{figure}

 \begin{figure}[ht!] 
  $$  
 \begin{tikzpicture}
  [xscale=0.5,yscale=  
 - 0.5]
      \foreach \i in {0,1.5,3,...,16.5}
 {\draw(0,\i)--++(0:16);}
      \draw(0,0) rectangle (16,1.5*11);
        \foreach \i in {0.5,1.5,...,15.5}
  {
   \fill(\i,0) circle(1.5pt) coordinate (a\i);
          \fill(\i,1.5)circle(1.5pt)  coordinate (d\i);
   \fill(\i,3) circle(1.5pt) coordinate (a\i);
          \fill(\i,4.5)circle(1.5pt)  coordinate (d\i);
   \fill(\i,6) circle(1.5pt) coordinate (a\i);
          \fill(\i,7.5)circle(1.5pt)  coordinate (d\i);
   \fill(\i,9) circle(1.5pt) coordinate (a\i);
          \fill(\i,10.5)circle(1.5pt)  coordinate (d\i);
   \fill(\i,12) circle(1.5pt) coordinate (a\i);
          \fill(\i,13.5)circle(1.5pt)  coordinate (d\i);
                \fill(\i,12+1.5)circle(1.5pt)  coordinate (d\i);
                \fill(\i,12+1.5*2)circle(1.5pt);
                                \fill(\i,12+1.5*3)circle(1.5pt);        
                               }

\draw[cyan] (3.5-3,0) node[above] {$2$};
\draw[cyan] (4.5-3,0) node[above] {$4$};
\draw[cyan] (5.5-3,0) node[above] {$6$};

\draw[cyan] (3.5,0) node[above] {$1$};
\draw[cyan] (4.5,0) node[above] {$ 3$};
\draw[cyan] (5.5,0) node[above] {$5 $};

 \draw[cyan] (6.5,0) node[above] {$0 $};
\draw[cyan] (7.5,0) node[above] {$7 $};

\draw[cyan] (3.5-3,-1) node[above] {$\eps_2$};
\draw[cyan] (4.5-3,-1) node[above] {$\eps_3$};
\draw[cyan] (5.5-3,-1) node[above] {$\eps_4$};

\draw[cyan] (3.5,-1) node[above] {$\eps_2$};
\draw[cyan] (4.5,-1) node[above] {$\eps_3$};
\draw[cyan] (5.5,-1) node[above] {$\eps_4$};

 \draw[cyan] (6.5,-1) node[above] {$\eps_2$};
\draw[cyan] (7.5,-1) node[above] {$\eps_2$};

 \draw[darkgreen] (8+3.5-3,-1) node[above] {$\eps_2$};
\draw[darkgreen] (8+4.5-3,-1) node[above] {$\eps_1$};
\draw[darkgreen] (8+5.5-3,-1) node[above] {$\eps_4$};

\draw[darkgreen] (8+3.5,-1) node[above] {$\eps_2$};
\draw[darkgreen] (8+4.5,-1) node[above] {$\eps_1$};
\draw[darkgreen] (8+5.5,-1) node[above] {$\eps_4$};

 \draw[darkgreen] (8+6.5,-1) node[above] {$\eps_4$};
\draw[darkgreen] (8+7.5,-1) node[above] {$\eps_4$};

\draw[darkgreen] (8+3.5-3,0) node[above] {$6$};
\draw[darkgreen] (8+4.5-3,0) node[above] {$0$};
\draw[darkgreen] (8+5.5-3,0) node[above] {$4$};

\draw[darkgreen] (8+3.5,0) node[above] {$5$};
\draw[darkgreen] (8+4.5,0) node[above] {$7$};
\draw[darkgreen] (8+5.5,0) node[above] {$3 $};

 \draw[darkgreen] (8+6.5,0) node[above] {$2 $};
\draw[darkgreen] (8+7.5,0) node[above] {$1 $};

  \foreach \i in {1*1.5 }
  {
    \foreach \j in {3.5,4.5,...,15.5} 
   { \draw(\j,\i)--++(-90:1.5) ;
}
\draw(0.5,\i)--++(-90:1.5) ;
\draw(1.5,\i)-- (2.5,\i-1.5) ;
\draw(2.5,\i)-- (1.5,\i-1.5) ;
}

  \foreach \i in {2*1.5 }
  {
    \foreach \j in {4.5,5.5,...,15.5} 
   { \draw(\j,\i)--++(-90:1.5) ;
}
\draw(0.5,\i)--++(-90:1.5) ;
\draw(1.5,\i)--++(-90:1.5) ;

\draw(1+1.5,\i)-- (1+2.5,\i-1.5) ;
\draw(1+2.5,\i)-- (1+1.5,\i-1.5) ;
}

  \foreach \i in {3*1.5 }
  {
    \foreach \j in {6.5,7.5,...,15.5} 
   { \draw(\j,\i)--++(-90:1.5) ;
}
\draw(0.5,\i)--++(-90:1.5) ;
\draw(1.5,\i)--++(-90:1.5) ;
\draw(2.5,\i)--++(-90:1.5) ;

\draw(1+4.5,\i)-- (1+3.5,\i-1.5) ;
\draw(1+3.5,\i)-- (1+2.5,\i-1.5) ;
\draw(1+2.5,\i)-- (1+4.5,\i-1.5) ;
}

  \foreach \i in {4*1.5 }
  {
    \foreach \j in {7.5,8.5,...,15.5} 
   { \draw(\j,\i)--++(-90:1.5) ;
}
\draw(0.5,\i)--++(-90:1.5) ;
\draw(1.5,\i)--++(-90:1.5) ;
\draw(2.5,\i)--++(-90:1.5) ;
\draw(3.5,\i)--++(-90:1.5) ;

\draw(1+1+4.5,\i)-- (1+1+3.5,\i-1.5) ;
\draw(1+1+3.5,\i)-- (1+1+2.5,\i-1.5) ;
\draw(1+1+2.5,\i)-- (1+1+4.5,\i-1.5) ;
}

  \foreach \i in {5*1.5 }
  {
    \foreach \j in {8.5,9.5,...,15.5} 
   { \draw(\j,\i)--++(-90:1.5) ;
}
\draw(0.5,\i)--++(-90:1.5) ;
\draw(1.5,\i)--++(-90:1.5) ;
\draw(2.5,\i)--++(-90:1.5) ;
\draw(3.5,\i)--++(-90:1.5) ;
\draw(4.5,\i)--++(-90:1.5) ;
\draw(5.5,\i)--++(-90:1.5) ;

\draw(1+2+4.5,\i)-- (1+2+3.5,\i-1.5) ;
\draw(1+2+3.5,\i)-- (1+2+4.5,\i-1.5) ;
 }

  \foreach \i in {6*1.5 }
  {
    \foreach \j in {10.5,11.5,...,15.5} 
   { \draw(\j,\i)--++(-90:1.5) ;
} 
\draw(9.5,\i)-- (8.5,\i-1.5) ;
\draw(8.5,\i)-- (7.5,\i-1.5) ;
\draw(-1+8.5,\i)-- (-1+7.5,\i-1.5) ;
\draw(-1+-1+8.5,\i)-- (-1+-1+7.5,\i-1.5) ;
\draw(-1+-1+-1+8.5,\i)-- (-1+-1+-1+7.5,\i-1.5) ;
\draw(-1+-1+-1+-1+8.5,\i)-- (-1+-1+-1+-1+7.5,\i-1.5) ;
\draw(-1+-1+-1+-1+-1+8.5,\i)-- (-1+-1+-1+-1+-1+7.5,\i-1.5) ;
\draw(-1+-1+-1+-1+-1+-1+8.5,\i)-- (-1+-1+-1+-1+-1+-1+7.5,\i-1.5) ;
\draw(-1-1+-1+-1+-1+-1+-1+8.5,\i)-- (-1-1+-1+-1+-1+-1+-1+7.5,\i-1.5) ;
\draw(0.5,\i)-- (9.5,\i-1.5) ;
 }

  \foreach \i in {7*1.5 }
  {
    \foreach \j in {11.5,12.5,...,15.5} 
   { \draw(\j,\i)--++(-90:1.5) ;
}
\draw(0.5,\i)--++(-90:1.5) ;
\draw(1.5,\i)--++(-90:1.5) ;
\draw(2.5,\i)--++(-90:1.5) ;
\draw(3.5,\i)--++(-90:1.5) ;
\draw(4.5,\i)--++(-90:1.5) ;
 
\draw(6.5,\i)--(5.5,\i-1.5) ;
\draw(1+6.5,\i)--(1+5.5,\i-1.5) ;
\draw(1+1+6.5,\i)--(1+1+5.5,\i-1.5) ;
\draw(1+1+1+6.5,\i)--(1+1+1+5.5,\i-1.5) ;
\draw(1+1+1+1+6.5,\i)--(1+1+1+1+5.5,\i-1.5) ;

\draw(5.5,\i)--(10.5,\i-1.5) ;
  }

  \foreach \i in {8*1.5 }
  {
    \foreach \j in {13.5,14.5,...,15.5} 
   { \draw(\j,\i)--++(-90:1.5) ;
}
\draw(0.5,\i)--++(-90:1.5) ;
\draw(1.5,\i)--++(-90:1.5) ;
\draw(2.5,\i)--++(-90:1.5) ;

 \draw(-1+2+6.5,\i)--(-1+2+5.5,\i-1.5) ;
 \draw(-1+-1+2+6.5,\i)--(-1+-1+2+5.5,\i-1.5) ;
 \draw(-1+-1+-1+2+6.5,\i)--(-1+-1+-1+2+5.5,\i-1.5) ;
 \draw(-1+-1+-1+-1+2+6.5,\i)--(-1+-1+-1+-1+2+5.5,\i-1.5) ;

\draw(2+6.5,\i)--(2+5.5,\i-1.5) ;
\draw(2+1+6.5,\i)--(2+1+5.5,\i-1.5) ;
\draw(2+1+1+6.5,\i)--(2+1+1+5.5,\i-1.5) ;
\draw(2+1+1+1+6.5,\i)--(2+1+1+1+5.5,\i-1.5) ;
\draw(2+1+1+1+1+6.5,\i)--(2+1+1+1+1+5.5,\i-1.5) ;

\draw(3.5,\i)--(2+10.5,\i-1.5) ;
  }

  \foreach \i in {9*1.5 }
  {
    \foreach \j in {14.5,15.5} 
   { \draw(\j,\i)--++(-90:1.5) ;
}
\draw(0.5,\i)--++(-90:1.5) ;
\draw(1.5,\i)--++(-90:1.5) ;
\draw(2.5,\i)--++(-90:1.5) ;

 \draw(-1+-1+2+6.5,\i)--(-1+-1+2+5.5+1,\i-1.5) ;
 \draw(-1+-1+-1+2+6.5,\i)--(-1+-1+-1+2+5.5+1,\i-1.5) ;
  \draw(-1-1+-1+-1+2+6.5,\i)--(-1-1+-1+-1+2+5.5+1,\i-1.5) ;
   \draw(-1-1-1+-1+-1+2+6.5,\i)--(-1-1-1+-1+-1+2+5.5+1,\i-1.5) ;

\draw(2+6.5,\i)--(2+5.5,\i-1.5) ;
\draw(2+1+6.5,\i)--(2+1+5.5,\i-1.5) ;
\draw(2+1+1+6.5,\i)--(2+1+1+5.5,\i-1.5) ;
\draw(2+1+1+1+6.5,\i)--(2+1+1+1+5.5,\i-1.5) ;
\draw(2+1+1+1+1+6.5,\i)--(2+1+1+1+1+5.5,\i-1.5) ;
\draw(2+1+1+1+1+1+6.5,\i)--(2+1+1+1+1+1+5.5,\i-1.5) ;

\draw(7.5,\i)--(2+11.5,\i-1.5) ;
  }

  \foreach \i in {10*1.5 }
  {
    \foreach \j in { 15.5} 
   { \draw(\j,\i)--++(-90:1.5) ;
}
\draw(0.5,\i)--++(-90:1.5) ;
\draw(1.5,\i)--++(-90:1.5) ;
\draw(2.5,\i)--++(-90:1.5) ;
\draw(7.5,\i)--++(-90:1.5) ;
\draw(8.5,\i)--++(-90:1.5) ;

 \draw(-1+-1+2+6.5,\i)--(-1+-1+2+5.5+1,\i-1.5) ;
 \draw(-1+-1+-1+2+6.5,\i)--(-1+-1+-1+2+5.5+1,\i-1.5) ;
  \draw(-1-1+-1+-1+2+6.5,\i)--(-1-1+-1+-1+2+5.5+1,\i-1.5) ;
   \draw(-1-1-1+-1+-1+2+6.5,\i)--(-1-1-1+-1+-1+2+5.5+1,\i-1.5) ;

\draw(1+2+1+6.5,\i)--(1+2+1+5.5,\i-1.5) ;
\draw(1+2+1+1+6.5,\i)--(1+2+1+1+5.5,\i-1.5) ;
\draw(1+2+1+1+1+6.5,\i)--(1+2+1+1+1+5.5,\i-1.5) ;
\draw(1+2+1+1+1+1+6.5,\i)--(1+2+1+1+1+1+5.5,\i-1.5) ;
\draw(1+2+1+1+1+1+1+6.5,\i)--(1+2+1+1+1+1+1+5.5,\i-1.5) ;

\draw(1+1+7.5,\i)--(1+2+11.5,\i-1.5) ;
  }

  \foreach \i in {11*1.5 }
{\draw(0.5,\i)--++(-90:1.5) ;
\draw(1.5,\i)--++(-90:1.5) ;
\draw(2.5,\i)--++(-90:1.5) ;
\draw(7.5,\i)--++(-90:1.5) ;
\draw(8.5,\i)--++(-90:1.5) ;
\draw(9.5,\i)--++(-90:1.5) ;

\draw(10.5,\i)--++(-90:1.5) ;

 \draw(-1+-1+2+6.5,\i)--(-1+-1+2+5.5+1,\i-1.5) ;
 \draw(-1+-1+-1+2+6.5,\i)--(-1+-1+-1+2+5.5+1,\i-1.5) ;
  \draw(-1-1+-1+-1+2+6.5,\i)--(-1-1+-1+-1+2+5.5+1,\i-1.5) ;
   \draw(-1-1-1+-1+-1+2+6.5,\i)--(-1-1-1+-1+-1+2+5.5+1,\i-1.5) ;

  \draw(1+1+2+1+1+5.5,\i)--(1+1+2+1+1+5.5,\i-1.5) ;
\draw(1+1+2+1+1+1+6.5,\i)--(1+1+2+1+1+1+5.5,\i-1.5) ;
\draw(1+1+2+1+1+1+1+6.5,\i)--(1+1+2+1+1+1+1+5.5,\i-1.5) ;
\draw(1+1+2+1+1+1+1+1+6.5,\i)--(1+1+2+1+1+1+1+1+5.5,\i-1.5) ;

\draw(1+1+1+1+8.5,\i)--(2 +2+11.5,\i-1.5) ;
  }

\draw[darkgreen] (3.5-3,16*1.5+1-4.5-3) node[above] {$0$};
\draw[darkgreen] (4.5-3,16*1.5+1-4.5-3) node[above] {$2$};
\draw[darkgreen] (5.5-3,16*1.5+1-4.5-3) node[above] {$6$};

\draw[darkgreen] (3.5,16*1.5+1-4.5-3) node[above] {$7$};
\draw[darkgreen] (4.5,16*1.5+1-4.5-3) node[above] {$ 1$};
\draw[darkgreen] (5.5,16*1.5+1-4.5-3) node[above] {$5 $};

 \draw[darkgreen] (6.5,16*1.5+1-4.5-3) node[above] {$4 $};
\draw[darkgreen] (7.5,16*1.5+1-4.5-3) node[above] {$3  $};

\draw[darkgreen] (3.5-3,16*1.5+1+1-4.5-3) node[above] {$\eps_1$};
\draw[darkgreen] (4.5-3,16*1.5+1+1-4.5-3) node[above] {$\eps_2$};
\draw[darkgreen] (5.5-3,16*1.5+1+1-4.5-3) node[above] {$\eps_4$};

\draw[darkgreen] (3.5,16*1.5+1+1-4.5-3) node[above] {$\eps_1$};
\draw[darkgreen] (4.5,16*1.5+1+1-4.5-3) node[above] {$\eps_2$};
\draw[darkgreen] (5.5,16*1.5+1+1-4.5-3) node[above] {$\eps_4$};

 \draw[darkgreen] (6.5,16*1.5+1+1-4.5-3) node[above] {$\eps_4$};
\draw[darkgreen] (7.5,16*1.5+1+1-4.5-3) node[above] {$\eps_4$};

 \draw[cyan] (8+3.5-3,16*1.5+1+1-4.5-3) node[above] {$\eps_2$};
\draw[cyan] (8+4.5-3,16*1.5+1+1-4.5-3) node[above] {$\eps_4$};
\draw[cyan] (8+5.5-3,16*1.5+1+1-4.5-3) node[above] {$\eps_3$};

\draw[cyan] (8+3.5,16*1.5+1+1-4.5-3) node[above] {$\eps_2$};
\draw[cyan] (8+4.5,16*1.5+1+1-4.5-3) node[above] {$\eps_4$};
\draw[cyan] (8+5.5,16*1.5+1+1-4.5-3) node[above] {$\eps_3$};

 \draw[cyan] (8+6.5,16*1.5+1+1-4.5-3) node[above] {$\eps_2$};
\draw[cyan] (8+7.5,16*1.5+1+1-4.5-3) node[above] {$\eps_2$};

\draw[cyan] (8+3.5-3,16*1.5+1-4.5-3) node[above] {$0$};
\draw[cyan] (8+4.5-3,16*1.5+1-4.5-3) node[above] {$2$};
\draw[cyan] (8+5.5-3,16*1.5+1-4.5-3) node[above] {$4$};

\draw[cyan] (8+3.5,16*1.5+1-4.5-3) node[above] {$7$};
\draw[cyan] (8+4.5,16*1.5+1-4.5-3) node[above] {$1$};
\draw[cyan] (8+5.5,16*1.5+1-4.5-3) node[above] {$3 $};

 \draw[cyan] (8+6.5,16*1.5+1-4.5-3) node[above] {$6 $};
\draw[cyan] (8+7.5,16*1.5+1-4.5-3) node[above] {$5 $};

    \end{tikzpicture} $$

    \caption{     We let $h=1$, $\ell=4$, $\kappa=(0,2,4,6)\in (\ZZ/8\ZZ)^4$ and $\bet=\eps_1-\eps_2$ and $\gam=\eps_3-\eps_4$. 
We picture the element $ {\sf com}_{\gam\bet}^{\bet\gam}$, the corresponding paths are depicted in \cref{chopin}.}
    \label{kjhdfljkdfghlsdkfhg222222222222}    \end{figure}

   \subsection{Generator morphisms in non-zero degree}\label{KLRSPOTSECTION}
We have already seen how to pass between 
  $\SSTS,\SSTT\in \Std_{n,\sigma}(\la)$ any two {\em reduced} paths.  
  We will now see how to inflate a reduced path to obtain a non-reduced path.  
    Given $\SSTS,\SSTT\in \Std_{n,\sigma}(\la)$, we suppose that the former is obtained from the latter by 
    inflating by a path through a  single hyperplane $\al\in \Pi$.  
    Of course, since $\SSTS $ and $\SSTT$ have the same shape, this inflation must add an  
      $\SSTP_\al^\flat$  at some point (and will involve removing  an occurrence of $\SSTT_\emp$ in order to preserve $n$).     
        There are two ways which one can approach a hyperplane: from above or from below.   
Adding an   upward/downward occurrence of $\SSTP_\al^\flat$ corresponds to the spot/fork Soergel generator.

 \subsubsection{The spot morphism}
We now define the     morphism which 
corresponds to   reflection towards the origin through the hyperplane labelled by $\al \in \Pi$.
We   consider the paths 
\begin{align*}
\SSTP_{\emp}  = 
(\varepsilon_{1 }, \mydots, \varepsilon_{i-1 }
 ,  {\varepsilon_{ i  }},   \varepsilon_{ i+1  },\mydots, \varepsilon_{\aatch   })^\exx 
\qquad 
\SSTP_{\al }^\flat 
= 
(\varepsilon_{1 }, \mydots, \varepsilon_{i-1 }
 , \widehat{\varepsilon_{ i  }},   \varepsilon_{ i+1  },\mydots, \varepsilon_{\aatch   })^\exx 
 \boxtimes (\eps_i)^\exx
 \end{align*}
 examples of these paths are depicted in \cref{pathbendingforkspot}.  
We define the    {\sf KLR-spot} to be the element 
 $$ 
 {\sf spot}^{\emp}_{\al}:= 
\Upsilon   ^{\SSTP_{\emp } }
  _{\SSTP_{\al }^\flat} $$ 
which is of degree $+1$ (corresponding to the unique step of off the $\al$-hyperplane).  
 We have already  constructed an example of  an element ${\sf spot}^\emp_\al$ in great detail over the course of 
\cref{concanetaion,concanetaion2,concanetaion3,reducedeg}.

\begin{figure}[ht!]
$$  
  \begin{minipage}{2.6cm}\begin{tikzpicture}[scale=0.65] 
    
    \path(0,0) coordinate (origin);
          \foreach \i in {0,1,2,3,4,5}
  {
    \path  (origin)++(60:0.6*\i cm)  coordinate (a\i);
    \path (origin)++(0:0.6*\i cm)  coordinate (b\i);
     \path (origin)++(-60:0.6*\i cm)  coordinate (c\i);
    }
  
      \path(3,0) coordinate (origin);
          \foreach \i in {0,1,2,3,4,5}
  {
    \path  (origin)++(120:0.6*\i cm)  coordinate (d\i);
    \path (origin)++(180:0.6*\i cm)  coordinate (e\i);
     \path (origin)++(-120:0.6*\i cm)  coordinate (f\i);
    }

  \foreach \i in {0,1,2,3,4,5}
  {
    \draw[gray, densely dotted] (a\i)--(b\i);
        \draw[gray, densely dotted] (c\i)--(b\i);
    \draw[gray, densely dotted] (d\i)--(e\i);
        \draw[gray, densely dotted] (f\i)--(e\i);
            \draw[gray, densely dotted] (a\i)--(d\i);
                \draw[gray, densely dotted] (c\i)--(f\i);
 
     }
  \draw[very thick, magenta] (0,0)--++(0:3) ;
    \draw[very thick, cyan] (3,0)--++(120:3) coordinate (hi);
        \draw[very thick, darkgreen] (hi)--++(-120:3) coordinate  (hi);

        \draw[very thick, darkgreen] (hi)--++(-60:3) coordinate  (hi);

    \draw[very thick, cyan] (hi)--++(60:3) coordinate (hi);

      \path(0,0)--++(-60:5*0.6)--++(120:2*0.6)--++(0:0.6) coordinate (hi);
     \path(0,0)--++(0:4*0.6)--++(120:3*0.6)           coordinate(hi2) ; 
     \path(0,0)--++(0:4*0.6)          coordinate(hi3) ;

          \path(hi)  --++(120:0.6)
           coordinate(step1) 
          --++(0:0.6)
                     coordinate(step2) 
      --++(120:0.6)
                 coordinate(step3) 
                 --++(0:0.6)           coordinate(step4) 
      --++(120:0.6) 
                 coordinate(step5)  --++(0:0.6) 
     coordinate(step6)   ;

       \path(step6)   --++(-120:0.6*3)--++(30:0.1)  coordinate(hi4)   ;

        \draw[ thick,->]    (hi)  to [out=90,in=-30]  (step1) 
          to [out=-30,in=-150]
                              (step2) 
 to [out=90,in=-30] 
            (step3) 
                     to [out=-30,in=-150]       (step4) 
  to [out=90,in=-30] 
             (step5)       to [out=-30,in=-150]   (step6) 
              (step6) 
              to [out=-100,in=35] (hi4)  ;

    \end{tikzpicture}\end{minipage}
\quad
  \begin{minipage}{2.6cm}\begin{tikzpicture}[scale=0.65] 
    
    \path(0,0) coordinate (origin);
          \foreach \i in {0,1,2,3,4,5}
  {
    \path  (origin)++(60:0.6*\i cm)  coordinate (a\i);
    \path (origin)++(0:0.6*\i cm)  coordinate (b\i);
     \path (origin)++(-60:0.6*\i cm)  coordinate (c\i);
    }
  
      \path(3,0) coordinate (origin);
          \foreach \i in {0,1,2,3,4,5}
  {
    \path  (origin)++(120:0.6*\i cm)  coordinate (d\i);
    \path (origin)++(180:0.6*\i cm)  coordinate (e\i);
     \path (origin)++(-120:0.6*\i cm)  coordinate (f\i);
    }

  \foreach \i in {0,1,2,3,4,5}
  {
    \draw[gray, densely dotted] (a\i)--(b\i);
        \draw[gray, densely dotted] (c\i)--(b\i);
    \draw[gray, densely dotted] (d\i)--(e\i);
        \draw[gray, densely dotted] (f\i)--(e\i);
            \draw[gray, densely dotted] (a\i)--(d\i);
                \draw[gray, densely dotted] (c\i)--(f\i);
 
     }
  \draw[very thick, magenta] (0,0)--++(0:3) ;
    \draw[very thick, cyan] (3,0)--++(120:3) coordinate (hi);
        \draw[very thick, darkgreen] (hi)--++(-120:3) coordinate  (hi);

        \draw[very thick, darkgreen] (hi)--++(-60:3) coordinate  (hi);

    \draw[very thick, cyan] (hi)--++(60:3) coordinate (hi);

      \path(0,0)--++(-60:5*0.6)--++(120:2*0.6)--++(0:0.6) coordinate (hi);
     \path(0,0)--++(0:4*0.6)--++(120:3*0.6)           coordinate(hi2) ; 
     \path(0,0)--++(0:4*0.6)          coordinate(hi3) ;

          \path(hi)  --++(120:0.6)
           coordinate(step1) 
          --++(0:0.6)
                     coordinate(step2) 
      --++(120:0.6)
                 coordinate(step3) 
                 --++(0:0.6)           coordinate(step4) 
      --++(120:0.6) 
                 coordinate(step5)  --++(0:0.6) 
                 coordinate(step6)   ;

           \path(0,0)--++(0:4*0.6)--++(120:3*0.6)           coordinate(hin) ; 

                     \path(0,0)--++(0:4*0.6)--++(120:2.85*0.6)           coordinate(hi4) ;

                      \path(hi) --++(-60:0.05) coordinate (boo);
    \path(step1) --++(135:0.05) coordinate (boo2);
        \path(step2) --++(135:0.05) coordinate (boo3);
        \draw[ thick,->  ]    
        (hi)  to [out=90,in=-30]  (step1) 
          to [out=-30,in=-150]  (step2) 
         to  [out=-90,in=30]  (boo); 
          

    \end{tikzpicture}\end{minipage}
$$
\caption{
We let $h=3$, $\ell=1$, $e=5$ and $\al=\eps_3-\eps_1$  and we depict the paths $\SSTP_\al^\flat $ and $\SSTP_\emp $ (we actually only depict $\SSTP_\emptyset$ which is a third of the path $\SSTP_\emp $).  
 We have already  constructed the corresponding element ${\sf spot}^\emp_\al$ in great detail over the course of 
\cref{concanetaion,concanetaion2,concanetaion3,reducedeg}.  
}
\label{pathbendingforkspot}
\end{figure}

  \subsubsection{The fork morphism}

We wish to understand the  morphism from 
    $\SSTP_{\al }\otimes\reflectpath$ to $\SSTP_{\emp}\otimes \SSTP_{\al}$.  
We  define the {\sf KLR-fork} to be the elements  
 $${\sf fork}_{\al\al}^{\emp\al}:=		 
 \Upsilon _{ \SSTP_{\al }\otimes\reflectpath }^{\SSTP_{{\emp}}\otimes \SSTP_{{\al}}}
 $$  
 The element $ {\sf fork}_{\al\al}^{\emp\al}$ is of degree $-1$.   
 \begin{figure}[ht!]
   \begin{minipage}{12cm}
 \begin{tikzpicture}
  [xscale=0.6,yscale=  
 -0.6]

      \foreach \i in {0,1.5,3,...,10.5}
 {\draw(0,\i)--++(0:18);}
      \draw(0,0) rectangle (18,10.5);
        \foreach \i in {0.5,1.5,...,17.5}
  {
   \fill(\i,0) circle(1.5pt) coordinate (a\i);
          \fill(\i,1.5)circle(1.5pt)  coordinate (d\i);
   \fill(\i,3) circle(1.5pt) coordinate (a\i);
          \fill(\i,4.5)circle(1.5pt)  coordinate (d\i);
   \fill(\i,6) circle(1.5pt) coordinate (a\i);
          \fill(\i,7.5)circle(1.5pt)  coordinate (d\i);
   \fill(\i,9) circle(1.5pt) coordinate (a\i);
          \fill(\i,10.5)circle(1.5pt)  coordinate (d\i); 
   }


\draw(9.5,0)--++(90:1*1.5)--(2.5,3)--++(90:5*1.5);

\draw(11.5,0)--++(90:3*1.5)--(5.5,4*1.5)--++(90:3*1.5);

\draw(13.5,0)--++(90:5*1.5)--(8.5,6*1.5)--++(90:1*1.5);

    \scalefont{0.9}
\draw (-3+3.5,0) node[above] {$\eps_1$};
\draw (-3+4.5,0) node[above] {$\eps_2$};
\draw (-3+5.5,0) node[above] {$\eps_1$};
\draw (-3+6.5,0) node[above] {$\eps_2$};
\draw (-3+7.5,0) node[above] {$\eps_1$};
\draw (-3+8.5,0) node[above] {$\eps_2$};
\draw (-3+9.5,0) node[above] {$\eps_1$};
\draw (-3+10.5,0) node[above] {$\eps_1$};
\draw (-3+11.5,0) node[above] {$\eps_1$};
\draw (9+-3+3.5,0) node[above] {$\eps_3$};
\draw (9+-3+4.5,0) node[above] {$\eps_2$};
\draw (9+-3+5.5,0) node[above] {$\eps_3$};
\draw (9+-3+6.5,0) node[above] {$\eps_2$};
\draw (9+-3+7.5,0) node[above] {$\eps_3$};
\draw (9+-3+8.5,0) node[above] {$\eps_2$};
\draw (9+-3+9.5,0) node[above] {$\eps_1$};
\draw (9+-3+10.5,0) node[above] {$\eps_1$};
\draw (9+-3+11.5,0) node[above] {$\eps_1$};

\draw (20,0) node  {$ \SSTP_\al \otimes \SSTP_\al^\flat $};
\draw (20,10.5) node  {$ \SSTP_\emp  \otimes \SSTP_\al  $};

\draw (-3+3.5,10.5) node[below] {$\eps_1$};
\draw (-3+4.5,10.5) node[below] {$\eps_2$};
\draw (-3+5.5,10.5) node[below] {$\eps_3$};
\draw (-3+6.5,10.5) node[below] {$\eps_1$};
\draw (-3+7.5,10.5) node[below] {$\eps_2$};
\draw (-3+8.5,10.5) node[below] {$\eps_3$};
\draw (-3+9.5,10.5) node[below] {$\eps_1$};
\draw (-3+10.5,10.5) node[below] {$\eps_2$};
\draw (-3+11.5,10.5) node[below] {$\eps_3$};
\draw (9+-3+3.5,10.5) node[below] {$\eps_1$};
\draw (9+-3+4.5,10.5) node[below] {$\eps_2$};
\draw (9+-3+5.5,10.5) node[below] {$\eps_1$};
\draw (9+-3+6.5,10.5) node[below] {$\eps_2$};
\draw (9+-3+7.5,10.5) node[below] {$\eps_1$};
\draw (9+-3+8.5,10.5) node[below] {$\eps_2$};
\draw (9+-3+9.5,10.5) node[below] {$\eps_1$};
\draw (9+-3+10.5,10.5) node[below] {$\eps_1$};
\draw (9+-3+11.5,10.5) node[below] {$\eps_1$};

\draw(2+4.5,0)--(2+4.5,1.5)--(2+5.5,3)--(2+5.5,4.5)--(2+6.5,6) --(2+6.5,7.5)
--(3+6.5,9)--(3+6.5,10.5)  ;

\draw(0.5,0)--++(90:7*1.5);
\draw(1.5,0)--++(90:7*1.5);
\draw(9+-3+11.5,0)--++(90:7*1.5);
\draw(9+-4+11.5,0)--++(90:7*1.5);
\draw(9+-5+11.5,0)--++(90:7*1.5);
\draw(9+-6+11.5,0)--++(90:7*1.5);
\draw(2.5,0)--(2.5,1.5)--(3.5,3)--++(90:5*1.5);
\draw(3.5,0)--(3.5,1.5)--(4.5,3)--++(90:5*1.5);

\draw(4.5,0)--(4.5,1.5)--(5.5,3)--(5.5,4.5)--(6.5,6)--++(90:3*1.5);
\draw(1+4.5,0)--(1+4.5,1.5)--(1+5.5,3)--(1+5.5,4.5)--(1+6.5,6)--++(90:3*1.5);

\draw(8.5,0)--(8.5,1.5)--(8.5+5,6*1.5)--++(90:1.5)  ;

 \draw(12.5,0)--(12.5,4*1.5) --(11.5,5*1.5) 
--(12.5,6*1.5)--(12.5,7*1.5)
  ;

\draw(-1+1+1+2+4.5,0)--(-1+1+1+2+4.5,1.5)--(-1+1+1+2+5.5,3)--(-1+1+1+1+2+5.5,4.5)--( +1+1+2+6.5,6) --(-1+1+2+2+6.5,7.5)
--(-1+1+2+3+6.5,9)--(-2+2+2+3+6.5,10.5)  ;

\draw(10.5,0)--(10.5,2*1.5) --(8.5,3*1.5) --(9.5,4*1.5) 
--(9.5,5*1.5)
--(10.5,6*1.5)
--(10.5,7*1.5)
 ;

  \end{tikzpicture} 
    \end{minipage}  
  \caption{ Fix $\ell=1$ and $h=3$ and $e=5$ and $\al=\varepsilon_3-\varepsilon_1$
 (so that $\exx=3$).    We picture the element  $ \Upsilon^{\SSTP_\al\otimes\reflectpath}_
{\SSTP_\emp\otimes \SSTP_\al } $. The first 9 and final 4  of the branching coefficients are trivial and so we do not waste trees by picturing all of them. 
The corresponding paths are pictured in 
\cref{pathbendingfork}.    }
 \label{reducedeasfjkdhklfghsdlkfghsdfgkjsdhfglsdifjghdsfgksdfgjhsdfglg}
 \end{figure}

\begin{figure}[ht!]
$$
  \begin{minipage}{2.6cm}\begin{tikzpicture}[scale=0.65] 
    
    \path(0,0) coordinate (origin);
          \foreach \i in {0,1,2,3,4,5}
  {
    \path  (origin)++(60:0.6*\i cm)  coordinate (a\i);
    \path (origin)++(0:0.6*\i cm)  coordinate (b\i);
     \path (origin)++(-60:0.6*\i cm)  coordinate (c\i);
    }
  
      \path(3,0) coordinate (origin);
          \foreach \i in {0,1,2,3,4,5}
  {
    \path  (origin)++(120:0.6*\i cm)  coordinate (d\i);
    \path (origin)++(180:0.6*\i cm)  coordinate (e\i);
     \path (origin)++(-120:0.6*\i cm)  coordinate (f\i);
    }

  \foreach \i in {0,1,2,3,4,5}
  {
    \draw[gray, densely dotted] (a\i)--(b\i);
        \draw[gray, densely dotted] (c\i)--(b\i);
    \draw[gray, densely dotted] (d\i)--(e\i);
        \draw[gray, densely dotted] (f\i)--(e\i);
            \draw[gray, densely dotted] (a\i)--(d\i);
                \draw[gray, densely dotted] (c\i)--(f\i);
 
     }
  \draw[very thick, magenta] (0,0)--++(0:3) ;
    \draw[very thick, cyan] (3,0)--++(120:3) coordinate (hi);
        \draw[very thick, darkgreen] (hi)--++(-120:3) coordinate  (hi);

        \draw[very thick, darkgreen] (hi)--++(-60:3) coordinate  (hi);

    \draw[very thick, cyan] (hi)--++(60:3) coordinate (hi);

      \path(0,0)--++(-60:5*0.6)--++(120:2*0.6)--++(0:0.6) coordinate (hi);
     \path(0,0)--++(0:4*0.6)--++(120:3*0.6)           coordinate(hi2) ; 
     \path(0,0)--++(0:4*0.6)          coordinate(hi3) ;

          \path(hi)  --++(120:0.6)
           coordinate(step1) 
          --++(0:0.6)
                     coordinate(step2) 
      --++(120:0.6)
                 coordinate(step3) 
                 --++(0:0.6)           coordinate(step4) 
      --++(120:0.54) 
                 coordinate(step5)  --++(0:0.62) 
                 coordinate(step6)   ;

        \draw[ thick]    (hi)  to [out=90,in=-30]  (step1) 
          to [out=-30,in=-150]
                              (step2) 
 to [out=90,in=-30] 
            (step3) 
                     to [out=-30,in=-150]       (step4) 
  to [out=90,in=-30] 
             (step5)       to [out=-30,in=-150]   (step6) to [out=80,in=-15] (hi2)  ;

    \path(hi2)  --++(-120:0.6)
           coordinate(step1) 
          --++(0:0.6)
                     coordinate(step2) 
      --++(-120:0.6)
                 coordinate(step3) 
                 --++(0:0.6)           coordinate(step4) 
      --++(-120:0.56)    coordinate    (step5)        --++(0:0.54) 
                 coordinate(step6)   ;

               \path(0,0)--++(0:4*0.6)--++(120:2.85*0.6)           coordinate(hi4) ;

   \draw[ thick,->]    (hi2)  to [out=-90,in=30]  (step1) 
          to [out=30,in=150]
                              (step2) 
 to [out=-90,in=30] 
            (step3) 
                     to [out=30,in=150]       (step4) 
  to [out=-90,in=30] 
             (step5)       to [out=30,in=150]   (step6) 
              to [out=110,in=-50] (hi4)  ;  ;

 \end{tikzpicture}\end{minipage}
  \qquad\quad 
  \begin{minipage}{2.6cm}\begin{tikzpicture}[scale=0.65] 
    
    \path(0,0) coordinate (origin);
          \foreach \i in {0,1,2,3,4,5}
  {
    \path  (origin)++(60:0.6*\i cm)  coordinate (a\i);
    \path (origin)++(0:0.6*\i cm)  coordinate (b\i);
     \path (origin)++(-60:0.6*\i cm)  coordinate (c\i);
    }
  
      \path(3,0) coordinate (origin);
          \foreach \i in {0,1,2,3,4,5}
  {
    \path  (origin)++(120:0.6*\i cm)  coordinate (d\i);
    \path (origin)++(180:0.6*\i cm)  coordinate (e\i);
     \path (origin)++(-120:0.6*\i cm)  coordinate (f\i);
    }

  \foreach \i in {0,1,2,3,4,5}
  {
    \draw[gray, densely dotted] (a\i)--(b\i);
        \draw[gray, densely dotted] (c\i)--(b\i);
    \draw[gray, densely dotted] (d\i)--(e\i);
        \draw[gray, densely dotted] (f\i)--(e\i);
            \draw[gray, densely dotted] (a\i)--(d\i);
                \draw[gray, densely dotted] (c\i)--(f\i);
 
     }
  \draw[very thick, magenta] (0,0)--++(0:3) ;
    \draw[very thick, cyan] (3,0)--++(120:3) coordinate (hi);
        \draw[very thick, darkgreen] (hi)--++(-120:3) coordinate  (hi);

        \draw[very thick, darkgreen] (hi)--++(-60:3) coordinate  (hi);

    \draw[very thick, cyan] (hi)--++(60:3) coordinate (hi);

      \path(0,0)--++(-60:5*0.6)--++(120:2*0.6)--++(0:0.6) coordinate (hi);
     \path(0,0)--++(0:4*0.6)--++(120:3*0.6)           coordinate(hi2) ; 
     \path(0,0)--++(0:4*0.6)          coordinate(hi3) ;

          \path(hi)  --++(120:0.6)
           coordinate(step1) 
          --++(0:0.6)
                     coordinate(step2) 
      --++(120:0.6)
                 coordinate(step3) 
                 --++(0:0.6)           coordinate(step4) 
      --++(120:0.54) 
                 coordinate(step5)  --++(0:0.62) 
                 coordinate(step6)   ;

           \path(0,0)--++(0:4*0.6)--++(120:2*0.6)           coordinate(hin) ; 

                     \path(0,0)--++(0:4*0.6)--++(120:2.85*0.6)           coordinate(hi4) ; 

        \draw[ thick,->]    (hi)  to [out=90,in=-30]  (step1) 
          to [out=-30,in=-150]
                              (step2) 
 to [out=90,in=-30] 
            (step3) 
                     to [out=-30,in=-150]       (step4) 
  to [out=90,in=-30] 
             (step5)       to [out=-30,in=-150]   (step6) 
              (step6) 
              to [out=100,in=-35] (hi4)  ;

    \end{tikzpicture}\end{minipage}
$$
\caption{    Fix $\ell=1$ and $h=3$ and $e=5$ and $\al=\varepsilon_3-\varepsilon_1$
 (so that $\exx=3$).  
We depict the paths ${\sf P}_{\al}\otimes {\sf P}_{\al}^\flat$ 
and ${\sf P}_{\emp}  \otimes {\sf P}_{\al} $  (although we do not depict the determinant path). 
The corresponding  fork generator is pictured in 
 \cref{reducedeasfjkdhklfghsdlkfghsdfgkjsdhfglsdifjghdsfgksdfgjhsdfglg}. }
\label{pathbendingfork}
\end{figure}

\subsection{Light leaves for the Bott--Samelson truncation}
We now rewrite the truncated basis of  \cref{cellularitybreedscontempt2222223323232}  in terms of the Bott--Samelson generators (thus showing that these are, indeed, generators of the truncated algebra).  
 Of course, the idempotent of \cref{redue} is specifically chosen so that the truncated algebra 
$$
     {\sf f}_{n,\sigma} (\mathcal{H}^\sigma_n/ \mathcal{H}^\sigma_n {\sf y}_{\underline{\tau} }\mathcal{H}^\sigma_n )      {\sf f}_{n,\sigma} 
  $$  
has basis  indexed by the  (sub)set of alcove-tableaux (and this basis is simply obtained from that of  \cref{cellularitybreedscontempt}  by truncation).  
It only remains to illustrate how the  reduced-path-vectors can be chosen  to mirror the construction of 
 paths  in $\Std_{n,\sigma}(\la)$ via  concatenation.  

  We can extend a path    $\SSTT'\in\Std_{n,\sigma}(\la)$   to obtain a new path $\SSTT$ in one of three possible ways 
$$
\SSTT=\SSTT'\otimes \SSTP_\al 	\qquad 	\SSTT=\SSTT'\otimes  \reflectpath   		\qquad \SSTT=\SSTT'\otimes \SSTP_\emptyset 
$$
for some  $\al \in \Pi$. 
  The first two cases each subdivide into a further two cases based on whether $\al$ is an upper or lower wall of the alcove containing $\la$.    These four cases are pictured in \cref{upanddown} (for $\SSTS_\emptyset$ we refer the reader to \cref{figure1}).  
  Any two reduced paths $ \SSTP_{\w} , \SSTP_{\underline{v}}\in \Std_{n,\sigma} (\la)$  can be obtained from one another 
by some iterated application of hexagon, adjustment, and commutativity permutations.
    We let 
 $${\sf rex}_{\SSTP_{\w} }^{\SSTP_{\underline{v}}  }  $$ denote the corresponding path-morphism in the  
   algebras $\mathcal{H}_n^\sigma/\mathcal{H}_n^\sigma{\sf y}_\aatchpair  \mathcal{H}_n^\sigma$ (so-named as they permute  reduced expressions).  
   In the following construction, we will assume that 
   the elements $c^{\SSTS'}_{\SSTT'}$  exist for any choice of
   reduced path  $\SSTS'$.  
 We then extend $\SSTS'$ using one of the 
$   U_0, U_1,  D_0, $ and $D_1$ paths (which puts a restriction on the form of the reduced expression)
 but then use a ``rex move" to   
 obtain  cellular basis elements ``glued  together'' along an idempotent corresponding to an arbitrary reduced path.

\begin{figure}[ht!]
 
$$
   \begin{minipage}{2.6cm}\begin{tikzpicture}[scale=0.75] 
    
    \path(0,0) coordinate (origin);
          \foreach \i in {0,1,2,3,4,5}
  {
    \path  (origin)++(60:0.6*\i cm)  coordinate (a\i);
    \path (origin)++(0:0.6*\i cm)  coordinate (b\i);
     \path (origin)++(-60:0.6*\i cm)  coordinate (c\i);
    }
  
      \path(3,0) coordinate (origin);
          \foreach \i in {0,1,2,3,4,5}
  {
    \path  (origin)++(120:0.6*\i cm)  coordinate (d\i);
    \path (origin)++(180:0.6*\i cm)  coordinate (e\i);
     \path (origin)++(-120:0.6*\i cm)  coordinate (f\i);
    }

  \foreach \i in {0,1,2,3,4,5}
  {
    \draw[gray, densely dotted] (a\i)--(b\i);
        \draw[gray, densely dotted] (c\i)--(b\i);
    \draw[gray, densely dotted] (d\i)--(e\i);
        \draw[gray, densely dotted] (f\i)--(e\i);
            \draw[gray, densely dotted] (a\i)--(d\i);
                \draw[gray, densely dotted] (c\i)--(f\i);
 
     }
  \draw[very thick, magenta] (0,0)--++(0:3) ;
    \draw[very thick, cyan] (3,0)--++(120:3) coordinate (hi);
        \draw[very thick, darkgreen] (hi)--++(-120:3) coordinate  (hi);

        \draw[very thick, darkgreen] (hi)--++(-60:3) coordinate  (hi);

    \draw[very thick, cyan] (hi)--++(60:3) coordinate (hi);

\path                 (hi)  --++(180:0.6*4) 
 coordinate (hiyer);

      \path(0,0)--++(-60:5*0.6)--++(120:2*0.6)--++(0:0.6) coordinate (hi);
     \path(0,0)--++(0:4*0.6)--++(120:3*0.6)           coordinate(hi2) ; 
     \path(0,0)--++(0:4*0.6)          coordinate(hi3) ;

          \path(hi)  --++(120:0.6)
           coordinate(step1) 
          --++(0:0.6)
                     coordinate(step2) 
      --++(120:0.6)
                 coordinate(step3) 
                 --++(0:0.6)           coordinate(step4) 
      --++(120:0.6) 
                 coordinate(step5)  --++(0:0.6) 
                 coordinate(step6)   ;

           \path(0,0)--++(0:4*0.6)--++(-120:3*0.6) --++(25:0.1)          coordinate(hin) ;

        \draw[ thick,->]    (hi)  to    (step1) 
          to [out=0,in=-180]
                              (step2) 
 to   
            (step3) 
                     to [out=0,in=-180]       (step4) 
  to   
             (step5)       to [out=0,in=-180]   (step6) to [out=-90,in=45] (hin)  ;

 \end{tikzpicture}\end{minipage} 
 \quad  \begin{minipage}{2.6cm}\begin{tikzpicture}[scale=0.75] 
    
    \path(0,0) coordinate (origin);
          \foreach \i in {0,1,2,3,4,5}
  {
    \path  (origin)++(60:0.6*\i cm)  coordinate (a\i);
    \path (origin)++(0:0.6*\i cm)  coordinate (b\i);
     \path (origin)++(-60:0.6*\i cm)  coordinate (c\i);
    }
  
      \path(3,0) coordinate (origin);
          \foreach \i in {0,1,2,3,4,5}
  {
    \path  (origin)++(120:0.6*\i cm)  coordinate (d\i);
    \path (origin)++(180:0.6*\i cm)  coordinate (e\i);
     \path (origin)++(-120:0.6*\i cm)  coordinate (f\i);
    }

  \foreach \i in {0,1,2,3,4,5}
  {
    \draw[gray, densely dotted] (a\i)--(b\i);
        \draw[gray, densely dotted] (c\i)--(b\i);
    \draw[gray, densely dotted] (d\i)--(e\i);
        \draw[gray, densely dotted] (f\i)--(e\i);
            \draw[gray, densely dotted] (a\i)--(d\i);
                \draw[gray, densely dotted] (c\i)--(f\i);
 
     }
  \draw[very thick, magenta] (0,0)--++(0:3) ;
    \draw[very thick, cyan] (3,0)--++(120:3) coordinate (hi);
        \draw[very thick, darkgreen] (hi)--++(-120:3) coordinate  (hi);

        \draw[very thick, darkgreen] (hi)--++(-60:3) coordinate  (hi);

    \draw[very thick, cyan] (hi)--++(60:3) coordinate (hi);

\path                 (hi)  --++(180:0.6*4) 
 coordinate (hiyer);
      
      \path(0,0)--++(-60:5*0.6)--++(120:2*0.6)--++(0:0.6) coordinate (hi);
     \path(0,0)--++(0:4*0.6)--++(120:3*0.6)           coordinate(hi2) ; 
     \path(0,0)--++(0:4*0.6)          coordinate(hi3) ;

          \path(hi)  --++(120:0.6)
           coordinate(step1) 
          --++(0:0.6)
                     coordinate(step2) 
      --++(120:0.6)
                 coordinate(step3) 
                 --++(0:0.6)           coordinate(step4) 
      --++(120:0.6) 
                 coordinate(step5)  --++(0:0.6) 
                 coordinate(step6)   ;    
     
                \path(0,0)--++(0:4*0.6)--++(120:2*0.6)           coordinate(hin) ; 

                     \path(0,0)--++(0:4*0.6)--++(120:2.85*0.6)           coordinate(hi4) ; 

        \draw[ thick,->]    (hi)  to    (step1) 
          to [out=0,in=-180]
                              (step2) 
 to   
            (step3) 
                     to [out=0,in=-180]       (step4) 
  to   
             (step5)       to [out=0,in=-180]   (step6) 
              (step6) 
              to [out=100,in=-35] (hi4)  ;

 \end{tikzpicture}\end{minipage}
\qquad
  \begin{minipage}{2.6cm}\begin{tikzpicture}[xscale=0.75,yscale=-0.75] 
    
    \path(0,0) coordinate (origin);
          \foreach \i in {0,1,2,3,4,5}
  {
    \path  (origin)++(60:0.6*\i cm)  coordinate (a\i);
    \path (origin)++(0:0.6*\i cm)  coordinate (b\i);
     \path (origin)++(-60:0.6*\i cm)  coordinate (c\i);
    }
  
      \path(3,0) coordinate (origin);
          \foreach \i in {0,1,2,3,4,5}
  {
    \path  (origin)++(120:0.6*\i cm)  coordinate (d\i);
    \path (origin)++(180:0.6*\i cm)  coordinate (e\i);
     \path (origin)++(-120:0.6*\i cm)  coordinate (f\i);
    }

  \foreach \i in {0,1,2,3,4,5}
  {
    \draw[gray, densely dotted] (a\i)--(b\i);
        \draw[gray, densely dotted] (c\i)--(b\i);
    \draw[gray, densely dotted] (d\i)--(e\i);
        \draw[gray, densely dotted] (f\i)--(e\i);
            \draw[gray, densely dotted] (a\i)--(d\i);
                \draw[gray, densely dotted] (c\i)--(f\i);
 
     }
  \draw[very thick, magenta] (0,0)--++(0:3) ;
    \draw[very thick, cyan] (3,0)--++(120:3) coordinate (hi);
        \draw[very thick, darkgreen] (hi)--++(-120:3) coordinate  (hi);

        \draw[very thick, darkgreen] (hi)--++(-60:3) coordinate  (hi);

    \draw[very thick, cyan] (hi)--++(60:3) coordinate (hi);

\path                 (hi)  --++(180:0.6*4) 
 coordinate (hiyer);
      
      \path(0,0)--++(-60:5*0.6)--++(120:2*0.6)--++(0:0.6) coordinate (hi);
     \path(0,0)--++(0:4*0.6)--++(120:3*0.6)           coordinate(hi2) ; 
     \path(0,0)--++(0:4*0.6)          coordinate(hi3) ;

          \path(hi)  --++(120:0.6)
           coordinate(step1) 
          --++(0:0.6)
                     coordinate(step2) 
      --++(120:0.6)
                 coordinate(step3) 
                 --++(0:0.6)           coordinate(step4) 
      --++(120:0.6) 
                 coordinate(step5)  --++(0:0.6) 
                 coordinate(step6)   ;    
     
                \path(0,0)--++(0:4*0.6)--++(120:2*0.6)           coordinate(hin) ; 

                     \path(0,0)--++(0:4*0.6)--++(120:2.85*0.6)           coordinate(hi4) ; 

        \draw[ thick,->]    (hi)  to    (step1) 
          to [out=0,in=-180]
                              (step2) 
 to   
            (step3) 
                     to [out=0,in=-180]       (step4) 
  to   
             (step5)       to [out=0,in=-180]   (step6) 
              (step6) 
              to [out=100,in=-35] (hi4)  ;

 \end{tikzpicture}\end{minipage}
\qquad
   \begin{minipage}{2.6cm}\begin{tikzpicture}[xscale=0.75,yscale=-0.75] 
    
    \path(0,0) coordinate (origin);
          \foreach \i in {0,1,2,3,4,5}
  {
    \path  (origin)++(60:0.6*\i cm)  coordinate (a\i);
    \path (origin)++(0:0.6*\i cm)  coordinate (b\i);
     \path (origin)++(-60:0.6*\i cm)  coordinate (c\i);
    }
  
      \path(3,0) coordinate (origin);
          \foreach \i in {0,1,2,3,4,5}
  {
    \path  (origin)++(120:0.6*\i cm)  coordinate (d\i);
    \path (origin)++(180:0.6*\i cm)  coordinate (e\i);
     \path (origin)++(-120:0.6*\i cm)  coordinate (f\i);
    }

  \foreach \i in {0,1,2,3,4,5}
  {
    \draw[gray, densely dotted] (a\i)--(b\i);
        \draw[gray, densely dotted] (c\i)--(b\i);
    \draw[gray, densely dotted] (d\i)--(e\i);
        \draw[gray, densely dotted] (f\i)--(e\i);
            \draw[gray, densely dotted] (a\i)--(d\i);
                \draw[gray, densely dotted] (c\i)--(f\i);
 
     }
  \draw[very thick, magenta] (0,0)--++(0:3) ;
    \draw[very thick, cyan] (3,0)--++(120:3) coordinate (hi);
        \draw[very thick, darkgreen] (hi)--++(-120:3) coordinate  (hi);

        \draw[very thick, darkgreen] (hi)--++(-60:3) coordinate  (hi);

    \draw[very thick, cyan] (hi)--++(60:3) coordinate (hi);

\path                 (hi)  --++(180:0.6*4) 
 coordinate (hiyer);

      \path(0,0)--++(-60:5*0.6)--++(120:2*0.6)--++(0:0.6) coordinate (hi);
     \path(0,0)--++(0:4*0.6)--++(120:3*0.6)           coordinate(hi2) ; 
     \path(0,0)--++(0:4*0.6)          coordinate(hi3) ;

          \path(hi)  --++(120:0.6)
           coordinate(step1) 
          --++(0:0.6)
                     coordinate(step2) 
      --++(120:0.6)
                 coordinate(step3) 
                 --++(0:0.6)           coordinate(step4) 
      --++(120:0.6) 
                 coordinate(step5)  --++(0:0.6) 
                 coordinate(step6)   ;

           \path(0,0)--++(0:4*0.6)--++(-120:3*0.6) --++(25:0.1)          coordinate(hin) ;

        \draw[ thick,->]    (hi)  to    (step1) 
          to [out=0,in=-180]
                              (step2) 
 to   
            (step3) 
                     to [out=0,in=-180]       (step4) 
  to   
             (step5)       to [out=0,in=-180]   (step6) to [out=-90,in=45] (hin)  ;

 \end{tikzpicture}\end{minipage} 
 $$
\caption{The first (respectively last) two paths   are  $\SSTS_\al$ and $\reflectpath $ originating in an alcove with $\al$ labelling an upper (respectively lower) wall. 
Here we take the convention that the origin is below the pink hyperplane.  
The degrees of these paths are $  1, 0,  0, -1$ respectively.  
We call these paths $  U_0, U_1,  D_0, $ and $D_1$ respectively.  
}
 \label{upanddown}
\end{figure}

\renewcommand{\SSTS}{\SSTP}
   \begin{defn}\label{twins} 
Suppose that $\la $ belongs to an alcove which has a hyperplane labelled by $\al$ as an   upper alcove wall.  
Let 
   $\SSTT'\in\Std_{n,\sigma}(\la)$. 
 If $\SSTT= \SSTT' \otimes \SSTP_\al$ then we 
    inductively define  
\begin{align*}
   &c^{\SSTT}_\SSTS= 
 (c^{\SSTT'}_{\SSTS'} \;\otimes \; e_{\SSTP_\al})
  {\sf rex}^{\SSTS'\otimes \SSTP_\al}_{\SSTS}. 
\intertext{
If $\SSTT= \SSTT' \otimes  \reflectpath  $ then
we  inductively define  
}  
&c^{\SSTT}_\SSTS= 
 (c^{\SSTT'}_{\SSTS'} \;\otimes \; {\sf spot}^{\al}_\emp)
  {\sf rex}^{\SSTS'\otimes \SSTP_\emp}_{\SSTS}. 
\end{align*}
Now suppose that $\la $ belongs to an alcove which has a hyperplane labelled by $\al$ as a lower alcove wall. 
Thus we can choose $ \SSTS_{\underline{v}} \otimes \SSTP_\al=\SSTS'\in\Std(\la) $.  
For 
$\SSTT= \SSTT' \otimes \SSTP_\al$, 
we 
   inductively define  
  \begin{align*}  c^{\SSTT}_\SSTS=  &%
\big(  c_{\SSTS'}^{\SSTT'} 
\otimes e_{\SSTP_\al}\big)   
\big( 
{\sf e} _ {\SSTS_{\underline{v}}}\otimes
(   {\sf fork}^{\al\al}_{\al\emp} \circ {\sf spot}^\al_\emp  )\big) 
{\sf rex}_{		
  \SSTS  }^{\SSTS_{\underline{v} \emp\emp}}
\end{align*}
  and if $\SSTT= \SSTT' \otimes  \reflectpath  $ then 
  then
we  
  inductively define    
\begin{align*}
 c^{\SSTT}_\SSTS=  
 \big(   c_{\SSTS'}^{\SSTT'} \otimes e_{\SSTP_\al}\big)
 \big( { e}_{\SSTS_{\underline{v}}}\otimes {\sf fork}^{\al\al}_{\al\emp}\big) 
{\sf rex}_{	  \SSTS  }^{ \SSTS_{\underline{v}\al \emp}} .  
 \end{align*}
  \end{defn}

\begin{thm}[The Libedinsky--Williamson  light leaves basis] \label{Theorembbyanyothername}
Given weakly increasing $\sigma \in \ZZ^\ell$, we let    $\underline{\tau}= (\tau_0,\dots,\tau_{\ell-1})  \in \NN^\ell$
      be such that  $\tau_m\leq   \sigma_{m+1}-\sigma_{m} $ for $0\leq m< \ell-1$ and  $\tau_{\ell-1}<  e+\sigma_0-\sigma_{\ell-1}  $.  
Suppose that  
 $n$ is divisible by $\aatch  $.  
For each   $\la \in \mathscr{P}_{\underline{\tau}}(n,\sigma)$  we fix an arbitrary 
reduced path   $
 \SSTP_\la \in \Std_{n,\sigma}(\la )
 $.   
The algebra 
  $ {\sf f}_{n,\sigma}(\mathcal{H}^\sigma_n/ \mathcal{H}^\sigma_n {\sf y}_\aatchpair  \mathcal{H}^\sigma_n) {\sf f}_{n,\sigma}$ is   quasi-hereditary   with  graded  integral cellular basis  
 $$
 \{ c^{\SSTS}_{\SSTP_\la}  
   c^{\SSTP_\la}_\SSTT
 \mid
  \SSTS,\SSTT \in \Std_{n,\sigma}(\la ), \la  \in  \mathscr{P}_{\underline{\tau}}(n ,\sigma)  \}  
 $$
 with respect to the    ordering $\;\psucc $ on  $ \mathscr{P}_{\underline{\tau}} (n )$,  the anti-involution $\ast$ given by flipping a diagram through the horizontal axis and the map $\deg :   \Std_{n,\sigma}  (\la ) \to \ZZ$.  
\end{thm}

\begin{proof}
Suppose that $\SSTQ, \SSTU \in \Std_{k,\sigma}( \nu)$ with $\SSTQ$ reduced and $k<n$ divisible by $h$.  
 By induction, we may assume that 
 $c^{\SSTQ}_{\SSTU} =\Upsilon ^{\underline{\SSTP}_{\SSTU}}_{\SSTU} $ for some reduced path vector ${\underline{\SSTP}_{\SSTU}}$ such that 
 ${\underline{\SSTP}_{\SSTU}}=(\SSTP_{\SSTU,0},\SSTP_{\SSTU,1},\dots , \SSTP_{\SSTU,k})$ with 
$ \SSTP_{\SSTU,k}=\SSTQ$.  
 By \cref{cellularitybreedscontempt2222223323232} and our inductive assumption, the result holds for all $k<n$ divisible by $h$.   
 Now suppose that $\la \in \mathscr{P}_\aatchpair(n,\sigma)$ and that 
 $\la$ belongs to an alcove, $A_\la$, which has a hyperplane labelled by $\al$      and that $\mu = \la \cdot s_\al$.    
  We now   reconstruct the element $c^{\SSTS }_{\SSTT }$ in terms of the basis of modified branching coefficients (as in  \cref{cellularitybreedscontempt2222223323232}) with $\SSTS:=\SSTP_\la$ reduced and $\SSTT$  equal to either 
  $ \SSTU\otimes \SSTP_\al$ 
   or 
 $  \SSTU\otimes \SSTP_\al^\flat$.  
This amounts to defining a reduced path vector,
 $$
\underline{ \SSTP}_{{\SSTT}}= 
(\underline{ \SSTP}_{{\SSTU}}, 
 { \SSTP}_{\SSTT,k+1}, { \SSTP}_{\SSTT,k+2},\dots, 
 { \SSTP}_{{\SSTT,n}})
 $$
 for which 
   $c^{\SSTP}_{\SSTT} =\Upsilon ^{\underline{\SSTP}_{\SSTT}}_{\SSTT} $.  
To do this,  we  simply  set 
 $$  {\SSTP}_{\SSTT,j}
 =
 \begin{cases}
 (\SSTQ \otimes \SSTP_\al)\downarrow_{\leq j}	&\text{if }\SSTT= \SSTU \otimes \SSTP_\al \text{ and } k < j <n \\
 (\SSTQ \otimes \SSTP_\al^\flat)\downarrow_{\leq j}	&\text{if }\SSTT= \SSTU \otimes \SSTP^\flat_\al \text{ and } k < j <n \\ 
 \SSTP &\text{if }j=n.
 \end{cases} 
 $$
      To summarise: we incorporate the ``rex" move into the final branching coefficient (and all other branching coefficients are left unmodified).  Choosing the reduced path vectors in this fashion, we obtain the required basis as a special case of  \cref{cellularitybreedscontempt2222223323232} .  
  \end{proof}
   
    We have shown that we can write a basis for our algebra entirely in terms of the elements  $$
   e_{\SSTP_\al}, \quad  {\sf fork}_{\al\al}^{\al\emp},
 \quad
  {\sf spot}_{\al}^{\emp},
 \quad
  {\sf hex}_{\al\bet\al}^{\bet\al\bet},  
   \quad
  {\sf com}_{\bet\gam}^{\gam\bet},
   \quad e_{\SSTP_\emptyset}, \ \text{and} \ \
  {\sf adj}_{\al\emptyset}^{\emptyset\al}
 $$
  for $\al,\bet,\gam \in \Pi$ such that $\al$ and $\bet$ label an arbitrary pair  of non-commuting reflections and 
$\bet$ and $\gam$ label  an arbitrary pair  of commuting reflections.  
Thus we deduce the following:  

\begin{cor}
Theorem B of the introduction holds.  
\end{cor}

\begin{Acknowledgements*}  The first and third authors thank the Institut Henri Poincar\'e for hosting us during the thematic trimester on representation theory.  
The first author was funded by EPSRC grant EP/V00090X/1 and the 
third author was funded by the Royal Commission for the Exhibition of 1851.
The authors would like to express their gratitude 
to the referee for their incredibly helpful comments and careful reading of the paper.    
 \end{Acknowledgements*}

\bibliographystyle{amsalpha}   
\bibliography{master} 
\end{document}